\documentclass[11pt]{article}

\usepackage[pagebackref,colorlinks=true,pdfpagemode=none,urlcolor=blue,
linkcolor=blue,citecolor=blue]{hyperref}

\usepackage{amsmath,amsfonts,amssymb,amsthm}
\usepackage{mathrsfs}
\usepackage{bm}
\usepackage{color}
\usepackage{algorithm}
\usepackage{algorithmic}

\usepackage{nostyle}
\usepackage{graphicx}
\usepackage{subfigure}

\def\nm{\noalign{\medskip}}

\title{Target Identification Using Dictionary Matching
of Generalized Polarization Tensors\thanks{\footnotesize This work
was supported by ERC Advanced Grant Project MULTIMOD--267184 and
NRF grants No. 2009-0090250 and 2010-0017532.}}
\author{}

\begin{document}

\author{Habib Ammari\thanks{\footnotesize Department of Mathematics and Applications,
Ecole Normale Sup\'erieure, 45 Rue d'Ulm, 75005 Paris, France
(habib.ammari@ens.fr, boulier@dma.ens.fr, wjing@dma.ens.fr,
han.wang@ens.fr).} \and Thomas Boulier\footnotemark[2] \and
Josselin Garnier\thanks{\footnotesize Laboratoire de
Probabilit\'es et Mod\`eles Al\'eatoires \& Laboratoire
Jacques-Louis Lions, Universit\'e Paris VII, 75205 Paris Cedex 13,
France (garnier@math.jussieu.fr).} \and Wenjia
Jing\footnotemark[2] \and Hyeonbae Kang\thanks{Department of
Mathematics, Inha University, Incheon 402-751, Korea
(hbkang@inha.ac.kr).}  \and Han Wang\footnotemark[2]}

\maketitle

\begin{abstract}
The aim of this paper is to provide a fast and efficient procedure
for (real-time) target identification in imaging based on matching
on a dictionary of precomputed generalized polarization tensors
(GPTs). The approach is based on some important properties of the
GPTs and new invariants. A new shape representation is given and
numerically tested in the presence of measurement noise. The
stability and resolution of the proposed identification algorithm
is numerically quantified.
\end{abstract}

\bigskip

\noindent {\footnotesize Mathematics Subject Classification
(MSC2000): 35R30, 35B30}


\section{Introduction}\label{sec:introduction}

With each domain and material parameter, an infinite number of
tensors, called the Generalized Polarization Tensors (GPTs), is
associated. The concept of GPTs was introduced in
\cite{AK_SIMA_03, AK04}. The GPTs contain significant information
on the shape of the domain \cite{AK_MMS_03}. It occurs in several
interesting contexts, in particular, in low-frequency scattering
\cite{dassios, AK04}, asymptotic models of dilute composites (see
\cite{milton} and \cite{AKT_AA_05}), in invisibility cloaking in
the quasi-static regime \cite{AKLL11} and in potential theory
related to certain questions arising in hydrodynamics \cite{PS51}.

Another important use of this concept is for imaging diametrically
small inclusions from boundary measurements. In fact, the GPTs are
the basic building blocks for the asymptotic expansions of the
boundary voltage perturbations due to the presence of small
conductivity inclusions inside a conductor \cite{FV_ARMA_89,CMV98,
AK_SIMA_03}. Based on this expansion, efficient algorithms to
determine the location and some geometric features of the
inclusions were proposed. We refer to \cite{AK04,
ammari_polarization_2007} and the references therein for recent
developments of this theory.

In \cite{AKLZ12}, a recursive optimal control scheme to recover
fine shape details of a given domain using GPTs is proposed. In
\cite{AGKLY11}, it is shown that high-frequency oscillations of
the boundary of a domain are only contained in its high-order
GPTs. Moreover, by developing a level set version of the recursive
optimization scheme, it is also shown that the GPTs can capture
the topology of the domain. An efficient algorithm for computing
the GPTs has been presented in \cite{yves}.

The aim of this paper is to show that the GPTs can be used for
target identification from imaging data. In fact, the GPTs can be
accurately obtained from multistatic measurements by solving a
linear system. Based on this, we design a fast algorithm which
identifies a target using a dictionary of precomputed GPTs data.
We first provide a stability analysis for the reconstruction of
the GPTs in the presence of measurement noise which quantifies the
ill-posedness of the imaging problem. Then, suppose that we have a
dictionary which is a collection of standard shapes (for example
alphabetic letters or flowers). Our aim is to identify from
imaging data a shape which is obtained from one element of the
dictionary after some rotation, scaling and translation. We design
a dictionary matching procedure which operates directly in the
GPTs data. Our procedure is based on some important properties of
the GPTs and new invariants. We test the robustness of our
procedure with respect to a measurement noise in the imaging data.
Our approach is quite natural since it uses geometric quantities
obtained from the imaging data by simply inverting a linear
system. Moreover, there is an infinite number of invariants
associated with the GPTs. Furthermore, for a given dictionary, the
GPT-based representation may lead to better distinguishibility
between the dictionary elements.

Over the last decades, a considerable amount of work has been
devoted to nonlinear optimization techniques for solving the
imaging problem; see, for instance, \cite{opt,opt1,opt2} and the
references therein. More recently, new regularized optimal control
formulations for target imaging have been proposed in
\cite{optnew1,optnew2}. As far as we know, our approach in this
paper provides for the first time an alternative approach to
solving the full inverse problem for target identification and
characterization. It opens a way for real-time target
identification and tracking algorithms in wave imaging.

The paper is organized as follows. In section
\ref{sec:struct-mult-resp}, we introduce a particular linear
combination of the GPTs  to obtain what we call the contracted
GPTs (CGPTs) \cite{AKLL11}. In Section
\ref{sec:reconstr-cgpt-stab}, we investigate the reconstruction of
contracted GPTs, defined in \eqref{defc1}--\eqref{defc2} below,
from the multistatic response matrix of a conductivity problem. We
also consider the effect of the presence of measurement noise in
the MSR on the reconstruction of the CGPTs. Given a
signal-to-noise ratio, we determine  the statistical stability in
the reconstruction of the CGPTs, and show that such inverse
problem is exponentially unstable. This is the well-known
ill-posedness of the inverse conductivity problem. In section
\ref{sec:complex-cgpt-under} it is shown that the CGPTs have some
nice properties, such as simple rotation and translation formulas,
simple relation with shape symmetry, etc. More importantly, we
derive new invariants for the CGPTs. One of the matching
algorithms presented in section \ref{sec:shape-ident-cgpt} is
based on those invariants. Section \ref{sec:numer-exper} presents
a variety of numerical results for the target identification
problem and shows the viability of the proposed procedure.

\section{Structure of the Multistatic Response Matrix}\label{sec:struct-mult-resp}

The first part of this paper is to reconstruct CGPTs from the
multistatic response  (MSR) matrix, which measures the change in
potential field due to a conductivity inclusion. In this section,
we present the mathematical model for MSR and write it in terms of
the CGPTs associated to the conductivity inclusion.

We consider a two dimensional conductivity medium with uniform
conductivity equal to one, except in an inclusion where the
conductivity is $\kappa
> 1$; we denote by $\lambda$ the contrast of this inclusion, that
is,  $\lambda = (\kappa + 1)/(2(\kappa - 1))$. Let $D = z + \delta
B = \{ x = z + \delta y ~|~ y \in B\}$ model the conductivity
inclusion.  Here, $B$ is some $\mathcal{C}^2$ and bounded domain
in $\R^2$ whose typical length scale is of order one; $z$ is a
point in $\R^2$ and is taken here to be an estimation of the
location of the inclusion; $\delta$ is the typical length scale of
the inclusion. We refer to \cite{BHV01, AK04} for efficient
location search algorithms and to \cite{AGJ} for correcting the
effect of measurement noise on the localization procedure.

The MSR matrix is constructed as follows. Let
$\{x_r\}_{r=1}^{N_r}$ and $\{x_s\}_{s=1}^{N_s}$ model a set of
electric potential point detectors and electric point sources. We
assume in this paper that the two sets of locations coincide and
$N_r = N_s = N$. The MSR matrix $\mathbf{V}$ is an $N$-by-$N$
matrix whose $rs$-element is the difference of electric potentials
with and without the conductivity inclusions:
\begin{equation}
V_{rs} = u_s(x_r) - \Gamma_s(x_r), \quad r, s = 1, \ldots, N.
\end{equation}
Here, $\Gamma_s (x) = \Gamma(x-x_s)$ and $\Gamma(x) =
\frac{1}{2\pi} \log |x|$ is the fundamental solution of the
Laplace equation in $\R^2$, and $u_s(x)$ is the solution to the
transmission problem
\begin{equation}
\left\{
\begin{aligned}
\nabla \cdot (1 + (\kappa-1) \chi_{D}) \nabla u_s(x) &= \delta_{x_s}(x),  & & x \in \R^2 \backslash \partial D,\\
u_s (x) \big|_{+} &= u_s (x) \big|_{-}, & & x \in \partial D,\\
\nu_x \cdot (\nabla u_s) \big|_{+} &= \kappa \nu_x \cdot (\nabla u_s) \big|_{-}, & & x \in \partial D,\\
u_s(x) - \Gamma_s (x) & = \mathcal{O}(|x|^{-1}), & & |x-x_s| \to
\infty.
\end{aligned}
\right. \label{eq:transm}
\end{equation}
In the second and third equations above, the notation
$\phi\big|_{\pm}(x)$ denotes the limit $\lim_{t \downarrow 0}
\phi(x\pm t\nu_x)$, where $x \in \partial D$ and $\nu_x$ is the
outward unit normal of $\partial D$ at $x$.

\subsection{The asymptotic expansion of the perturbed potential field}

As modeled above, the MSR matrix characterizes the perturbed
potential field $u_s(x_r) - \Gamma_s(x_r)$. In this section we
recall, from \cite{AK04},  the asymptotic expansion of this
perturbation and some key notions along the way.

Let $\mathcal{S}_D$ be the single layer potential associated with
$D$, that is,
\begin{equation}
\mathcal{S}_D[\phi](x) :  = \int_{\partial D} \Gamma(x-y) \phi(y)
ds(y), \quad x \in \R^2,
\end{equation}
and let $\mathcal{K}_D: L^2(\partial D) \to L^2(\partial D)$
denote the Poincar\'e-Neumann operator
\begin{equation}
\mathcal{K}_D[\phi] (x) : = \frac{1}{2 \pi} \int_{\partial D}
\frac{\langle y - x, \nu_y\rangle}{|x-y|^2} \phi(y) ds(y), \quad x
\in \partial D. \label{eq:Kdef}
\end{equation}
Here, $\langle , \rangle$ denotes the scalar product in $\R^2$ and
$\nu_y$ is the unit normal vector along the boundary at $y$. It is
well known that the single layer potential $\mathcal{S}_D[\phi]$
is a harmonic function satisfying $\mathcal{S}_D[\phi]\big|_- =
\mathcal{S}_D[\phi]\big|_+$ and the jump condition
\begin{equation}
\frac{\partial}{\partial \nu} \mathcal{S}_D[\phi] \Big|_{\pm} =
\left( \pm \frac{1}{2} I + \mathcal{K}_D^*  \right)[\phi],
\label{eq:Sjump}
\end{equation}
where $\mathcal{K}_D^*$ is the adjoint operator of $\mathcal{K}_D$
and it has a similar expression as \eqref{eq:Kdef} with the
numerator of the integrand replaced by $\langle x - y,
\nu_x\rangle$. Using \eqref{eq:Sjump}, we verify that $\Gamma_s(x)
+ \mathcal{S}_D[\phi_s]$ with $\phi_s \in L^2(\partial D)$ solving
\begin{equation}
\left( \lambda I - \mathcal{K}_D^* \right)[\phi_s] =
\frac{\partial \Gamma_s}{\partial \nu} \Big|_{\partial D},
\end{equation}
is a solution to the transmission problem \eqref{eq:transm}. In
fact, this solution is unique and we conclude that
\begin{equation}
u_s(x) - \Gamma_s(x) = \mathcal{S}_D[\phi_s] = \int_{\partial D}
\Gamma(x - y) (\lambda I - \mathcal{K}_D^*)^{-1}
\bigg[\frac{\partial \Gamma_s}{\partial \nu}\Big|_{\partial D}
\bigg](y)  ds(y). \label{eq:dfield}
\end{equation}
To verify the formal derivation above, we refer the reader to
Section 2.4 of \cite{AK04}.

We assume that the inclusion $D$ and the point $z$ is away from
the sources. As a result, the functions $\Gamma(x_r - y)$ and
$\Gamma_s(y)$ are smooth for $y \in \overline{D}$, and the
perturbed field \eqref{eq:dfield} is well defined. For $y \in
\partial D$ and $z$ away from $x$, the $K$-th order Taylor
expansion formula with remainder $e_K$ states
\begin{equation}
\Gamma(x - y) = \Gamma(x - z - (y-z)) = \sum_{|\alpha| = 0}^K
\frac{(-1)^{|\alpha|}}{\alpha!} \partial^{\alpha} \Gamma(x - z)
(y-z)^\alpha + e_K.
\end{equation}
Throughout this section, we use Greek letters to denote double
indices: $\alpha = (\alpha_1, \alpha_2) \in \mathbb{N}^2$,
$\alpha! = \alpha_1 ! \alpha_2 !$ and $|\alpha| = \alpha_1 +
\alpha_2$. Substitution of this expansion into \eqref{eq:dfield}
yields the following expansion of $V_{rs}$ plus an error term
denoted by $E_{rs}$:
\begin{equation*}
V_{rs}= \sum_{|\alpha|,|\beta|=1}^K \frac{(-1)^{|\alpha|}}{\alpha
! \beta !} \partial^\alpha \Gamma(x_r-z) Q_{\alpha\beta}(z)
\partial^\beta \Gamma(z-x_s) + E_{rs},
\end{equation*}
with
$$
Q_{\alpha\beta}(z)= \int_{\partial D} (y-z)^\alpha (\lambda I -
\mathcal{K}_D^*)^{-1} \bigg[\frac{\partial}{\partial \nu} (\cdot
-z)^\beta\bigg](y) ds(y).
$$
The zeroth order term with $\beta = 0$ vanishes because the
differentiation $\partial/\partial \nu$; the zeroth order term
corresponding to $\alpha = 0$ vanishes because $(\lambda I -
\mathcal{K}^*_D)^{-1}$ maps a zero mean value function on
$\partial D$ to another zero mean value function.

For a generic conductivity inclusion $D$ with the contrast factor
$\lambda$, the GPT of order $\alpha\beta$ associated with the
inclusion is defined by
\begin{equation}
M_{\alpha \beta}(\lambda, D) : = \int_{\partial D} y^\beta
(\lambda I - \mathcal{K}_{D}^*)^{-1}[\frac{\partial}{\partial \nu}
y^\alpha] \, ds(y). \label{eq:Mdef}
\end{equation}

Using the change of variable $y - z \mapsto \tilde{y}$, the
integral term $Q_{\alpha\beta}(z)$ inside the expansion of
$V_{rs}$ above can be written as
\begin{equation}
Q_{\alpha\beta}(z) = \int_{\partial (\delta B)} \tilde{y}^\alpha
(\lambda I - \mathcal{K}_{\delta
B}^*)^{-1}[\frac{\partial}{\partial \nu} \tilde{y}^\beta]\,
ds(\tilde{y}),
\end{equation}
which is independant of $z$. Moreover, by the definition of GPT,
this term is $M_{ \beta \alpha}(\lambda, \delta B)$. As a result,
we have\begin{equation} V_{rs} = \sum_{|\alpha|,|\beta|=1}^K
\frac{1}{\alpha ! \beta !}
\partial^\alpha \Gamma(z-x_s) M_{\alpha \beta}(\lambda, \delta B)
\partial^\beta \Gamma(z-x_r) + E_{rs}, \label{eq:Vrsexp}
\end{equation}
where $E_{rs}$ is the truncation error resulted from the finite
expansion. Note also that we have switched the indices $\alpha$
and $\beta$.

The MSR matrix $\mathbf{V}$ consisting of $u_s(x_r) -
\Gamma_s(x_r)$ depends only on the inclusion $(\lambda,D)$.
However, the GPTs involved in the representation \eqref{eq:Vrsexp}
depend on the (non-unique) characterization $(z, \delta B)$ of
$D$. We note that the remainder $e_K$ and the truncation error
$E_{rs}$ can be evaluated; see Appendix \ref{sec:app1}. Moreover,
since the sensors and the receivers coincide, the MSR matrix is
symmetric; see \eqref{eq:Vsym}.

\subsection{Expansion for MSR using contracted GPT}

In this section, we further simplify the expression of MSR using
the notion of contracted GPT (CGPT), which has been introduced in
\cite{AKLL11}. Using CGPT, we can write the MSR matrix
$\mathbf{V}$ as a product of  a CGPT matrix with coefficient
matrices, which is a very convenient form for inversion.

Let $P_m(x)$ be the complex valued polynomial
\begin{equation}
P_m(x) = (x_1 + ix_2)^m := \sum_{|\alpha| = m} a^m_\alpha x^\alpha
+ i\sum_{|\beta| = m} b^m_\beta x^\beta. \label{eq:Pdef}
\end{equation}
Using polar coordinate $x = re^{i\theta}$, the above coefficients
$a^m_\alpha$ and $b^m_\beta$ can also be characterized by
\begin{equation}
\sum_{|\alpha| = m} a^m_\alpha x^\alpha = r^m \cos m\theta, \text{
and } \sum_{|\alpha| = m} b^m_\alpha x^\beta = r^m \sin m\theta.
\label{eq:abcomp}
\end{equation}
For a generic conductivity inclusion $D$ with contrast $\lambda$,
the associated GPT $M_{\alpha \beta}(\lambda, D)$ is defined as in
\eqref{eq:Mdef}. The associated CGPT is the following combination
of GPTs using the coefficients in \eqref{eq:Pdef}:
\begin{align}
M^{cc}_{mn} = \sum_{|\alpha| = m} \sum_{|\beta| = n} a^m_\alpha a^n_\beta M_{\alpha \beta}, \label{defc1}\\
M^{cs}_{mn} = \sum_{|\alpha| = m} \sum_{|\beta| = n} a^m_\alpha b^n_\beta M_{\alpha \beta},\\
M^{sc}_{mn} = \sum_{|\alpha| = m} \sum_{|\beta| = n} b^m_\alpha a^n_\beta M_{\alpha \beta},\\
M^{ss}_{mn} = \sum_{|\alpha| = m} \sum_{|\beta| = n} b^m_\alpha
b^n_\beta M_{\alpha \beta}. \label{defc2}
\end{align}

Using the complex coordinate $x = r_x e^{i\theta_x}$, we have (see
Appendix \ref{sec:app2}) that
\begin{equation}
\frac{(-1)^{|\alpha|}}{\alpha!} \partial^{\alpha} \Gamma(x) =
\frac{-1}{2\pi |\alpha|} \left[ a^{|\alpha|}_\alpha \frac{\cos
|\alpha| \theta_x}{r^{|\alpha|}_x}  + b^{|\alpha|}_\alpha
\frac{\sin |\alpha| \theta_x}{r^{|\alpha|}_x} \right].
\label{eq:DGamma}
\end{equation}
Recall that $\{x_r\}_{r=1}^N$ and $\{x_s\}_{s=1}^N$ denote the
locations of the receivers and electric sources. Define $R_r$ and
$\theta_r$ so that the complex representation of $x_r - z$ is $R_r
e^{i\theta_r}$ with $z$ being the location of the target.
Similarly define $R_s$ and $\theta_s$. Substituting formula
\eqref{eq:DGamma} into the expression \eqref{eq:Vrsexp} of the
MSR, we get
\begin{equation}
\begin{aligned}
V_{rs} &= \sum_{|\alpha|= 1, |\beta|=1}^{K}
\frac{a^{|\alpha|}_\alpha \cos |\alpha| \theta_s +
b^{|\alpha|}_\alpha \sin |\alpha| \theta_s}{2\pi |\alpha|
R_s^{|\alpha|}}
 M_{\alpha \beta} (\lambda, \delta B) \frac{a^{|\beta|}_\beta \cos |\beta| \theta_r + b^{|\beta|}_\beta \sin |\beta| \theta_r}{2\pi |\beta| R_r^{|\beta|}} + E_{rs}\\
&= \sum_{m,n=1}^K \underbrace{
\begin{pmatrix} \displaystyle
\frac{\cos m\theta_s}{2\pi m R_s^m} & \displaystyle \frac{\sin
m\theta_s}{2\pi m R_s^m} \end{pmatrix}}_{\mathbf{A}_{sm}}
\underbrace{\begin{pmatrix}
M^{cc}_{mn} & M^{cs}_{mn} \\
M^{sc}_{mn} & M^{ss}_{mn}
\end{pmatrix}}_{\mathbf{M}_{mn}}
\underbrace{\begin{pmatrix}
\cos n\theta_r \\
\sin n\theta_r
\end{pmatrix}
\frac{1}{2\pi n R_r^n}}_{(\mathbf{A}_{rn})^t} + E_{rs}.
\end{aligned}
\label{eq:Vrsform}
\end{equation}
Here, the short-hand notations $\Mcgpt_{mn}$ and $\mathbf{A}_{sm}$
represent the two-by-two and one-by-two  matrices respectively,
and $(\mathbf{A}_{rn})^t$ is the transpose. As $m,n$ run from one
to $K$, which is the truncation order of CGPT, and $r,s$ run from
one to $N$, which is the number of receivers (sources), these
matrices build up the $2K \times 2K$ CGPT block matrix $\Mcgpt$
and the $N \times 2K$ coefficient matrix $\mathbf{A}$ as follows:
\begin{equation}
\Mcgpt = \begin{pmatrix}
\Mcgpt_{11} & \Mcgpt_{12} & \cdots & \Mcgpt_{1K}\\
\Mcgpt_{21} & \Mcgpt_{22} & \cdots & \Mcgpt_{2K}\\
\cdots & \cdots & \ddots & \cdots\\
\Mcgpt_{K1} & \Mcgpt_{K2} & \cdots & \Mcgpt_{KK}
\end{pmatrix};
\Acoef = \begin{pmatrix}
\Acoef_{11} & \Acoef_{12} & \cdots & \Acoef_{1K}\\
\Acoef_{21} & \Acoef_{22} & \cdots & \Acoef_{2K}\\
\cdots & \cdots & \ddots & \cdots\\
\Acoef_{N1} & \Acoef_{N2} & \cdots & \Acoef_{NK}
\end{pmatrix}.
\label{eq:Mcgpt}
\end{equation}

Using these notations, the MSR matrix $\mathbf{V}$ can be written
as
\begin{equation}
\mathbf{V} = \Acoef  \Mcgpt \Acoef^t+ \mathbf{E}, \label{eq:Vexp}
\end{equation}
where $\Acoef^t$ denotes the transpose of $\Acoef$ and the matrix
$\mathbf{E} = (E_{rs})$ represents the truncation error. We
precise again that the CGPT above is for the ``shifted" inclusion
$\delta B$. We note also that the dimension of $\mathbf{V}$
depends on the number of sources/receivers but does not depend on
the expansion order $K$ in \eqref{eq:Vrsexp}.

Due to the symmetry of harmonic combination of GPTs
\cite{ammari_polarization_2007}, the matrix $\Mcgpt$ is symmetric.
Since $\mathbf{V}$ is symmetric as shown in \eqref{eq:Vsym}, the
truncation error $\mathbf{E}$ is also symmetric.

\section{Reconstruction of CGPTs and Stability Analysis}\label{sec:reconstr-cgpt-stab}

The first step in the target identification procedure is to
reconstruct CGPTs from the MSR matrix $\mathbf{V}$,
which has expression \eqref{eq:Vexp}. 
Define the linear operator $L: \R^{2K \times 2K} \to \R^{N\times
N}$ by
\begin{equation}
L(\Mcgpt) := \Acoef \Mcgpt \Acoef^t. \label{eq:linsys}
\end{equation}
We reconstruct CGPTs as the least squares solution of the above
linear system, {\it i.e.},
\begin{equation}
\Mcgpt^{\mathrm{est}} = \min_{\Mcgpt^{\mathrm{test}} \perp
\mathrm{ker } (L)} \| \mathbf{V} - L(\Mcgpt^{\mathrm{test}})
\|_{F}, \label{eq:lsqr}
\end{equation}
where $\mathrm{ker } (L)$ denotes the kernel of $L$ and
$\|\cdot\|_F$ denotes the Frobenius norm of matrices \cite{LH95}.
In general we take $N$ large enough so that $2K <N$. When $\Acoef$
has full rank $2K$, $L$ is rank preserving and $\mathrm{ker }(L)$
is trivial; in that case, the admissible set above can be replaced
by $\R^{2K\times 2K}$ and $$ \Mcgpt =   (\Acoef^t \Acoef)^{-1}
\Acoef^t \mathbf{V} \Acoef (\Acoef^t \Acoef)^{-1}.$$

From the structure of the matrix $\Acoef$ in \eqref{eq:Mcgpt} and
the expression of the MSR matrix, we observe that the contribution
of a CGPT decays as its order grows. Consequently, one does not
expect the inverse procedure to be stable for higher order CGPTs.
The remainder of this section is devoted to such stability
analysis.

\subsection{Analytical formula in the concentric setting}

To simplify the analysis, we assume that the receivers (sources)
are evenly distributed along a circle of radius $R$ centered at
$z$. That is, $\theta_r = 2\pi r/N$, $r = 1, 2, \ldots, N$, and
$R_r = R$. In this setting, we have $\Acoef = \Ccoef \Dcoef$,
where $\Ccoef$ is an $N \times 2K$ matrix constructed from the
block $\Ccoef_{rm} = (\cos m \theta_r\ \sin m\theta_r)$ and
$\Dcoef$ is $2K \times 2K$ diagonal matrix:
\begin{equation*}
\Ccoef = \begin{pmatrix}
\Ccoef_{11} & \Ccoef_{12} & \cdots & \Ccoef_{1K}\\
\Ccoef_{21} & \Ccoef_{22} & \cdots & \Ccoef_{2K}\\
\cdots & \cdots & \ddots & \cdots\\
\Ccoef_{N1} & \Ccoef_{N2} & \cdots & \Ccoef_{NK}
\end{pmatrix};
\Dcoef = \frac{1}{2\pi} \begin{pmatrix}
  \mathbf{I}_2/R&  &  &  \\
  &   \mathbf{I}_2/(2 R^2) &   &  \\
  &   & \ddots &  \\
  &   &   & \mathbf{I}_2/(K R^K)
\end{pmatrix}.
\end{equation*}
Here $\mathbf{I}_2$ is the $2\times 2$ identity matrix. We note
that $\Ccoef$ and $\Dcoef$ account for the angular and radial
coefficients in the expansion of MSR, respectively. The matrix
$\Ccoef$ satisfies the following important property; see Appendix
\ref{sec:app3}.

\begin{proposition} Suppose that $2K < N$ holds. Then
\begin{equation}
\Ccoef^t \Ccoef = \frac{N}{2} \mathbf{I}_{2K}. \label{eq:ortho}
\end{equation}
\end{proposition}

Henceforth, we assume that the number of receivers is large enough
so that $2K < N$. In this setting, the least squares problem
\eqref{eq:lsqr} admits an analytical expression as follows.

\begin{lemma}\label{lem:inv} In the above concentric setting with sufficiently many receivers, {\it i.e.},
 $2K < N$, the least squares estimation \eqref{eq:lsqr} is given by
\begin{equation}
\Mcgpt^{\mathrm{est}} = (\frac{2}{N})^2 \Dcoef^{-1} \Ccoef^t
\mathbf{V} \Ccoef \Dcoef^{-1}. \label{eq:Mest}
\end{equation}
\end{lemma}
\begin{proof} Firstly, \eqref{eq:ortho} implies that $\Acoef$ has full rank,
so $\mathrm{ker }(L) =\{0\}$. Moreover, $$ (\Acoef^t \Acoef)^{-1}
= \frac{2}{N} \Dcoef^{-2}.$$ Hence,
$$
\Mcgpt^{\mathrm{est}} = (\frac{2}{N})^2 \Dcoef^{-2} \Dcoef
\Ccoef^t \mathbf{V} \Ccoef \Dcoef \Dcoef^{-2},
$$
which yields (\ref{eq:Mest}).
\end{proof}

\subsection{Measurement noise and stability analysis}

We develop in the rest of this section a stability analysis for
the least squares reconstruction of CGPT from the MSR matrix, in
the setting of concentric receivers (sources).

Counting some additive measurement noise, we modify the expression
of MSR to
\begin{equation}
\mathbf{V} = \Ccoef \Dcoef \Mcgpt \Dcoef \Ccoef^t + \mathbf{E} +
\sigma_{\mathrm{noise}} \mathbf{W}. \label{eq:Vmodel}
\end{equation}
Here, $\mathbf{E}$ is the truncation error due to the finite order
$K$ in expansion \eqref{eq:Vrsexp}, $\mathbf{W}$ is an $N \times
N$ real valued random matrix with independent and identically
Gaussian entries with mean zero and unit variance, and
$\sigma_{\mathrm{noise}}$ is a small positive number modeling the
standard deviation of the noise.

Recall that the unknown $\Mcgpt$ consists of CGPTs of order up to
$K$ of the relative domain $\delta B = D - z$, where $\delta$
denote the typical length scale of the domain $D$. The receivers
and sources are located along a circle of radius $R$ centered at
$z$. Let $\eps = \delta/R$ be the ratio between the two scales,
and it is assumed to be smaller than one. Due to the scaling
property of CGPT (see (\ref{eq:CGPT_scl_Nt})), the entries of the
CGPT block $\Mcgpt_{mn}(\delta B)$ is
$\delta^{m+n}\Mcgpt_{mn}(B)$. Consequently, the size of
$\mathbf{V}$ itself is of order $\eps^2$, which is the order of
the first term in the expansion \eqref{eq:Vrsform}. The truncation
error $\mathbf{E}$ is of order $\eps^{K+2}$; see Appendix
\ref{sec:app1}.

According to the above analysis, we assume that the size of the
noise satisfies
\begin{equation}
N \eps^{K+2} \ll \sigma_{\mathrm{noise}} \ll \eps^2.
\label{eq:nregime}
\end{equation}
This is the regime where the measurement noise is much smaller
than the signal but much larger than the truncation error. The
presence of $N$ in (\ref{eq:nregime}) will be clear later; see
remark \ref{rem:whyN}. We define the signal-to-noise ratio (SNR)
to be
$$
\SNR = \frac{\eps^2}{\sigma_{\mathrm{noise}}}.
$$
We will investigate the error made by the least squares estimation
of the CGPT matrix, in particular the manner of its growth with
respect to the order of the CGPTs. Given a $\SNR$ and a tolerance
number $\tau_0$, we can define the resolving order $m_0$ to be
\begin{equation}
m_0 = \min \left\{ 1 \le m \le K ~:~ \sqrt{\frac{\E
\|\Mcgpt^{\mathrm{est}}_{mm} -
\Mcgpt_{mm}\|^2_F}{\|\Mcgpt_{mm}\|^2_F}} \le \tau_0 \right\}.
\label{eq:msdef}
\end{equation}
We are interested in the growth of $m_0$ with respect to $\SNR$.

We have used the notation $\Mcgpt_{mn}$, $m,n=1, \ldots, K$, to
denote the building block of the CGPT matrix $\Mcgpt$ in
(\ref{eq:Mcgpt}). In the following, we also use the notation
$(\Mcgpt)_{jk}$, $j,k=1, \ldots, 2K$, to denote the real valued
entries of the CGPT matrix.

\begin{theorem} Assume that the condition of Lemma \ref{lem:inv} holds; assume also that the additive noise is in the regime \eqref{eq:nregime}, Then for $j,k$ so that $(\Mcgpt)_{jk}$ is non-zero, the relative error in its reconstructed CGPT satisfies
\begin{equation}
\sqrt{\frac{\E |(\Mcgpt^{\mathrm{est}})_{jk} -
(\Mcgpt)_{jk}|^2}{|(\Mcgpt)_{jk}|^2}} \le C
\frac{\sigma_{\mathrm{noise}}}{N} \eps^{-\lceil j/2 \rceil -
\lceil k/2 \rceil} \left\lceil \frac{j}{2} \right\rceil
\left\lceil \frac{k}{2} \right\rceil. \label{eq:rerrM}
\end{equation}
Here, the symbol $\lceil l \rceil$ is the smallest natural number
larger than or equal to $l$. For vanishing $(\Mcgpt)_{jk}$, the
error $\sqrt{\E  |(\Mcgpt^{\mathrm{est}})_{jk} -
(\Mcgpt)_{jk}|^2}$ can be bounded by the right-hand side above
with $\eps$ replaced by $R^{-1}$. In particular, the resolving
order $m_0$ satisfies
\begin{equation}
(m_0 \eps^{1-m_0})^2 \simeq \tau_0 \SNR, \label{eq:snrest}
\end{equation}
where $\tau_0$ is the tolerance number.
\end{theorem}

\begin{proof}
From the analytical formula of the least squares reconstruction
\eqref{eq:Mest} and the expression of $\mathbf{V}$
\eqref{eq:Vmodel}, we see that
for each fixed $j,k = 1, \ldots, 2K$,
\begin{equation*}
(\Mcgpt^{\mathrm{est}} - \Mcgpt)_{jk} =  \frac{2^2
\sigma_{\mathrm{noise}}}{N^2} (\Dcoef^{-1} \Ccoef^t \mathbf{W}
\Ccoef \Dcoef^{-1})_{jk} + \frac{2^2}{N^2}(\Dcoef^{-1} \Ccoef^t
\mathbf{E} \Ccoef \Dcoef^{-1})_{jk}.
\end{equation*}

Let us denote these two terms by $\mathcal{I}_{jk1}$ and
$\mathcal{I}_{jk2}$ respectively. For the first term, define
$\widetilde{\mathbf{W}}$ to be $(\sqrt{2/N}\Ccoef)^t \mathbf{W}
(\sqrt{2/N}\Ccoef)$, which is an $N \times N$ random matrix. Due
to the orthogonality \eqref{eq:ortho}, $\widetilde{\mathbf{W}}$
remains to have mean zero Gaussian entries with unit variance.
Because $\Dcoef$ is diagonal, we have for each $j,k = 1, \ldots,
2K$,
\begin{equation*}
\E (\mathcal{I}_{jk1})^2 =
\frac{2^2\sigma^2_{\mathrm{noise}}}{N^2} (\Dcoef_{jj})^{-2} \E
|\widetilde{\mathbf{W}}_{jk}|^2 (\Dcoef_{kk})^{-2} =
\frac{2^6\pi^4 \sigma^2_{\mathrm{noise}}}{N^2} R^{2(\lceil j/2
\rceil + \lceil k/2 \rceil)} \left\lceil \frac{j}{2} \right
\rceil^2 \left\lceil \frac{k}{2} \right\rceil^2.
\end{equation*}
Note that $\lceil j/2 \rceil \lceil k/2 \rceil$ is the order of
CGPT element $(\Mcgpt)_{jk}$; see \eqref{eq:Mcgpt}. It is known
that $(\Mcgpt)_{jk}(\delta B) = \delta^{\lceil j/2 \rceil + \lceil
k/2 \rceil} (\Mcgpt)_{jk}(B)$. When this term is non-zero, it is
of order $\delta^{\lceil j/2 \rceil + \lceil k/2 \rceil}$. This
fact and the above control of $\mathcal{I}_{jk1}$ show that
$\sqrt{\E |\mathcal{I}_{jk1}|^2/|(\Mcgpt)_{jk}|^2}$ satisfies the
estimate in \eqref{eq:rerrM}.

For the second term, since $\mathbf{E}$ is symmetric, it has the
decomposition $\mathbf{E} = \mathbf{P}^t \mathcal{E} \mathbf{P}$,
where $\mathbf{P}$ is an $N \times N$ orthonormal matrix, and
$\mathcal{E}$ is an $N \times N$ diagonal matrix consisting of
eigenvalues of $\mathbf{E}$. Then $(\sqrt{2/N}\Ccoef)^t \mathbf{E}
(\sqrt{2/N}\Ccoef)$ can be written as $\mathbf{Q}^t \mathcal{E}
\mathbf{Q}$ where $\mathbf{Q} = \sqrt{2/N} \mathbf{P} \Ccoef$ is
an $N \times 2K$ matrix satisfying $\mathbf{Q}^t \mathbf{Q} =
\mathbf{I}_{2K}$. Then the calculation for $\mathcal{I}_{jk1}$
shows that
\begin{equation*}
(\mathcal{I}_{jk2})^2  = \frac{2^6\pi^4}{N^2} R^{2(\lceil j/2
\rceil + \lceil k/2 \rceil)} \left\lceil \frac{j}{2} \right
\rceil^2 \left\lceil \frac{k}{2} \right\rceil^2 \left(\sum_{l=1}^N
\mathcal{E}_{ll} \mathbf{Q}^t_{jl} \mathbf{Q}_{lk} \right)^2.
\end{equation*}
Since $\mathbf{E}$ is of order $\eps^{K+2}$ as shown in
\eqref{eq:prop:Ers}, the sum is of order $N\eps^{K+2}$. Therefore,
we have
$$\sqrt{\E |\mathcal{I}_{jk2}|^2} \le C\eps^{K+2 -\lceil j/2 \rceil
- \lceil k/2 \rceil} \lceil \frac{j}{2}\rceil \lceil \frac{k}{2}
\rceil .$$ Since we assumed that \eqref{eq:nregime} holds, this
error is dominated by the one due to the noise. Hence,
\eqref{eq:rerrM} is proved.

For diagonal blocks $\Mcgpt_{mm}$, their Frobenius norms do not
vanish and \eqref{eq:msdef} is well defined. In particular,
\eqref{eq:rerrM} applied to the case $j,k = 2m-1, 2m$, shows that
the relative error made in the block $\Mcgpt_{mm}$ is of order
$\sigma_{\mathrm{noise}} m^2 \eps^{-2m}$. Using the definition of
$\SNR$, we verify \eqref{eq:snrest}.
\end{proof}

\begin{remark} \label{rem:whyN}
If $\mathbf{E}$ has only several (of order one) non-zero
eigenvalues, then the preceding calculation shows that
$(\mathcal{I}_{jk2})^2 \le C\eps^{2(K+2)}$ and condition
\eqref{eq:nregime} can be replaced with $\eps^{K+2} \ll
\sigma_{\mathrm{noise}} \ll \eps^2$.
\end{remark}

\section{Complex CGPTs under Rigid Motions and Scaling}\label{sec:complex-cgpt-under}
As we will see later, a complex combination of CGPTs is most
convenient when we consider the transforms of CGPTs under
dilatation and rigid motions, {\it i.e.}, shift and rotation.
Therefore, for a double index $mn$, with $m,n=1,2,\ldots$, we
introduce the following complex combination of CGPTs:
\begin{equation}
\begin{aligned}
\Mcc^{(1)}_{mn}(\lambda, D) &= (M^{cc}_{mn} - M^{ss}_{mn}) +
i(M^{cs}_{mn} + M^{sc}_{mn}),\\
\Mcc^{(2)}_{mn}(\lambda, D) &= (M^{cc}_{mn} + M^{ss}_{mn}) +
i(M^{cs}_{mn} - M^{sc}_{mn}).
\end{aligned}
\label{eq:Mccdef}
\end{equation}
Then, from (\ref{eq:Mdef}),  we observe that
\begin{equation*}
\begin{aligned}
\Mcc^{(1)}_{mn}(\lambda, D) &= \int_{\partial D} P_n(y)  (\lambda
I - \K^*_{D})^{-1}[\langle \nu, \nabla P_m \rangle](y)\,
ds(y),\\
\Mcc^{(2)}_{mn}(\lambda, D) &= \int_{\partial D} P_n(y)  (\lambda
I - \K^*_{D})^{-1}[\langle \nu, \nabla \overline{P_m} \rangle] (y) \, ds(y),\\
\end{aligned}
\end{equation*}
where $P_n$ and $P_m$ are defined by (\ref{eq:Pdef}). In order to
simplify the notation, we drop $\lambda$ in the following and
write simply $\No_{mn}(D), \Nt_{mn}(D)$.

We consider the translation, the rotation and the dilatation of
the domain $D$ by
 introducing the following notation:

\begin{itemize}
\item Shift: $T_zD = \{x+z, \mbox{ for } x\in D\}$, for $z \in
\R^2$; \item Rotation: $R_\theta D = \{e^{i\theta}x, \mbox{ for }
x\in D\}$, for $\theta\in[0,2\pi)$; \item Scaling: $sD=\{sx,
\mbox{ for } x\in D\}$, for $s>0$.
\end{itemize}

\begin{proposition}
\label{prop:complex-cgpt-under}
  For all integers $m,n$, and geometric parameters $\theta$, $s$, and $z$, the following
  holds:
  \begin{align}
    \No_{mn}(R_\theta D) = e^{i(m+n)\theta}\No_{mn}(D), \quad
    \Nt_{mn}(R_\theta D) = e^{i(n-m)\theta}\Nt_{mn}(D),    \label{eq:CGPT_rot_Nt}\\
    \nm
    \No_{mn}(sD) = {s^{m+n}}\No_{mn}(D),   \quad
    \Nt_{mn}(sD) = {s^{m+n}}\Nt_{mn}(D),    \label{eq:CGPT_scl_Nt}\\
    \nm
    \No_{mn}(T_zD) = \sum_{l=1}^m \sum_{k=1}^n \mathbf{C}^z_{ml}\No_{lk}(D) \mathbf{C}^z_{nk},
    \quad
    \Nt_{mn}(T_zD) = \sum_{l=1}^m \sum_{k=1}^n \overline{\mathbf{C}^z_{ml}}\Nt_{lk}(D) \mathbf{C}^z_{nk},    \label{eq:CGPT_trans_Nt}
  \end{align}
  where $\mathbf{C}^z$ is a lower triangle matrix with the $m,n$-th entry given
  by
  \begin{equation} \label{defcz}
  \mathbf{C}^z_{mn}= \binom{m}{n}  z^{m-n},
  \end{equation}
  and $\overline{\mathbf{C}^z}$ denotes its conjugate. Here, we
  identify $z=(z_1,z_2)$ with $z= z_1 + i z_2$.
\end{proposition}

An ingredient that we will need in the proof is the following
chain rule between the gradient of a function and its push forward
under transformation.  In fact, for any diffeomorphism $T$ from
$\R^2$ to $\R^2$ and any scalar-valued differentiable map $f$ on
$\R^2$, we have
\begin{equation}
\textrm{d} (f\circ T) \big|_x (h) = \left(\textrm{d} f
\big|_{T(x)} \circ \textrm{d}T\big|_x \right) (h),
\label{eq:chainrule}
\end{equation}
for any tangent vector $h \in \R^2$, with $\textrm{d}T$ being the
differential of $T$.

\begin{proof}[Proof of Proposition \ref{prop:complex-cgpt-under}] We will follow proofs of similar relations that can be found in \cite{AGKLY11}.
Let us first show (\ref{eq:CGPT_rot_Nt}) for the rotated domain
$D_\theta:=R_\theta D$. For a function $\varphi(y), y\in
\partial D$, we define a function $\varphi^\theta(y_\theta)$,
$y_\theta := R_\theta y \in \partial D_\theta$ by
\begin{equation*}
\varphi^\theta(y_\theta) = \varphi \circ R_{-\theta} (y_\theta) =
\varphi(y).
\end{equation*}
It is proved in \cite{AGKLY11} that $\lambda I - \K^*_D$ is
invariant under the rotation map, that is,
\begin{equation}
(\lambda I - \K^*_{D_\theta})[\varphi^\theta] (y_\theta) =
(\lambda I - \K^*_{D})[\varphi](y). \label{eq:ImKrot}
\end{equation}
We also check that $P_m(R_\theta y) = e^{im\theta} P_m(y)$.

We will focus on the relation for $\No_{mn}$, the other one can be
proved in the same way. By definition, we have
\begin{equation}
\begin{aligned}
\No_{mn}(D) &= \int_{\partial D} P_n(y) \varphi_{D,m}(y) ds(y),\\
\No_{mn}(D_\theta) &= \int_{\partial D_\theta} P_n(y_\theta) \varphi_{D_\theta,m}(y_\theta) ds(y_\theta),\\
\end{aligned}
\label{eq:Nnm-Nnm_theta}
\end{equation}
where
\begin{equation*}
\begin{aligned}
\varphi_{D,m}(y) &= (\lambda I - \K^*_{D})^{-1} [\langle \nu, \nabla P_m \rangle](y), \\
\varphi_{D_\theta,m}(y_\theta) &= (\lambda I -
\K^*_{D_\theta})^{-1}[\langle \nu, \nabla P_m \rangle ](y_\theta).
\end{aligned}
\end{equation*}
Note that the last function differs from $\varphi^{\theta}_{D,m}$.
By the change of variables $y_\theta=R_{\theta} y$ in the first
expression of  (\ref{eq:Nnm-Nnm_theta}), we obtain
\begin{equation*}
\begin{aligned}
\No_{mn}(D) &= \int_{\partial D_\theta} P_n(R_{-\theta}y_\theta)
\varphi_{D,m}(R_{-\theta}y_\theta)
ds(y_\theta)\\
&= e^{-i n\theta}\int_{\partial D_\theta} P_n(y_\theta) \varphi^{\theta}_{D,m}(y_\theta) ds(y_\theta).\\
\end{aligned}
\end{equation*}
From (\ref{eq:ImKrot}), we have
\begin{equation*}
\begin{aligned}
(\lambda I - \K^*_{D_\theta})[\varphi^\theta_{D,m}](y_\theta) &= (\lambda I - \K^*_{D})[\varphi_{D,m}](y)\\
&= \langle \nu_y, \nabla P_m(y) \rangle.
\end{aligned}
\end{equation*}
Moreover, $P_m(y)=e^{-im\theta}P_m(y_\theta)$ so that, by applying
the chain rule (\ref{eq:chainrule}) with $f=P_m$, $T=R_\theta$,
$x=y$ and $h=\nu_y$, we can conclude that
\begin{equation*}
\begin{aligned}
\langle \nu_y,\nabla P_m(y) \rangle &= e^{-im\theta}\langle R_\theta\nu_y, \nabla P_m(y_\theta) \rangle \\
&= e^{-im\theta}\langle \nu_{y_\theta}, \nabla P_m(y_\theta)
\rangle.
\end{aligned}
\end{equation*}
Therefore,
$\varphi^\theta_{D,m}=e^{-im\theta}\varphi_{D_\theta,m}$, and we
conclude that
$N_{mn}^{(1)}(D_\theta)=e^{i(m+n)\theta}N_{mn}^{(1)}(D)$.

The second identity in (\ref{eq:CGPT_rot_Nt}) results from the
same computation as above (the minus sign comes form the conjugate
in the definition of $\Nt$), and the two equations in
(\ref{eq:CGPT_scl_Nt}) are proved in the same way, replacing the
transformed function $\varphi^\theta$ by
\begin{equation*}
\varphi^s(sy)=\varphi(y).
\end{equation*}

Thus, only (\ref{eq:CGPT_trans_Nt}) remains. Since the difference
between these two comes from the conjugation, we will focus only
on the first identity in (\ref{eq:CGPT_trans_Nt}). The strategy
will be once again the following: for a function $\varphi(y), y
\in
\partial D$, we define a function $\varphi^z(y_z), y_z =y + z \in
\partial D_z$, with $D_z:=T_z D$,  by
\begin{equation*}
\varphi^z(y_z)=\varphi \circ T_{-z}(y_z)=\varphi(y),
\end{equation*}
which also verifies an invariance relation similar to
(\ref{eq:ImKrot})
\begin{equation}
(\lambda I - \K^*_{D_z})[\varphi^z](y_z) = (\lambda I - \K^*_{D})
[\varphi] (y). \label{eq:ImKtrans}
\end{equation}
Moreover, for every integer $q\in\mathbb{N}$ one has the following
\begin{equation}
P_q(y_z) = (y+z)^q = \sum_{r=0}^q  \binom {q} {r} y^r z^{q-r}.
\label{eq:develop_P_q}
\end{equation}
Equations (\ref{eq:Nnm-Nnm_theta}) become
\begin{equation*}
\begin{aligned}
\No_{mn}(D) &= \int_{\partial D} P_n(y) \varphi_{D,m}(y) ds(y),\\
\No_{mn}(D_z) &= \int_{\partial D_z} P_n(y_z) \varphi_{D_z,m}(y_z) ds(y_z),\\
\end{aligned}
\end{equation*}
where
\begin{equation*}
\begin{aligned}
\varphi_{D,m}(y) &= (\lambda I - \K^*_{D})^{-1}[\langle \nu, \nabla P_m\rangle](y), \\
\varphi_{D_z,m}(y_z) &= (\lambda I - \K^*_{D_z})^{-1}[\langle \nu,
\nabla P_m \rangle](y_z).
\end{aligned}
\end{equation*}
Thus, combining (\ref{eq:ImKtrans}) and (\ref{eq:develop_P_q})
leads us to
\begin{equation*}
\begin{aligned}
(\lambda I - \K^*_{D_z})[\varphi_{D_z,m}](y_z) & = \langle \nu_{y_z}, \nabla P_m(y_z) \rangle \\
 &= \langle \nu_{y},\sum_{l=1}^m  \binom {m} {l} z^{m-l}\nabla P_l(y)  \rangle \\
 &= \sum_{l=1}^m  \binom {m} {l} z^{m-l} (\lambda I-\K^{*}_D)[\varphi_{D,l}] (y) \\
 &= \sum_{l=1}^m  \binom {m} {l} z^{m-l} (\lambda I-\K^{*}_{D_z})[\varphi^z_{D,l}] (y_z), \\
\end{aligned}
\end{equation*}
so that we have
\begin{equation*}
\varphi_{D_z,m} (y)= \sum_{l=1}^m  \binom {m} {l} z^{m-l}
\varphi^z_{D,l}(y_z).
\end{equation*}
Hence, returning to the definition of $\No_{mn}(D_z)$ with the
substitution $y_z\leftrightarrow y$, we obtain
\begin{equation*}
\begin{aligned}
\No_{mn}(D_z) &= \sum_{l=1}^m \binom{m}{l} z^{m-l }\int_{\partial D_z}P_n(y_z) \varphi^z_{D,l}(y_z) ds(y_z), \\
 &= \sum_{l=1}^m \sum_{k=1}^n \binom{m}{l}\binom{n}{k}z^{m-l}z^{n-k}\No_{lk}(D),
\end{aligned}
\end{equation*}
which is the desired result. Note that the index $k$ begins with
$k=1$ because $\int_{\partial D_z}\varphi^z_{D,l} = 0$. This
completes the proof. \end{proof}

\subsection{Some properties of complex CGPTs}\label{sec:some-properties-ccgpt}

We define the complex  CGPT matrices by $\No:= (\No_{mn})_{m,n}$
and $\Nt:= (\Nt_{mn})_{m,n}$. We set $w=se^{i\theta}$ and
introduce the diagonal matrix $\mathbf{G}^w$ with the $m$-th
diagonal entry given by $s^me^{im\theta}$. Proposition
\ref{prop:complex-cgpt-under} implies immediately that
\begin{align}
  \No(\tsr D) &= \mathbf{C}^z \mathbf{G}^w\No(D) \mathbf{G}^w(\mathbf{C}^z)^t ,\label{eq:DB_tsr_No}\\
  \nm
  \Nt(\tsr D) &= \overline{\mathbf{C}^z \mathbf{G}^w}\Nt(D) \mathbf{G}^w(\mathbf{C}^z)^t, \label{eq:DB_tsr_Nt}
\end{align}
where $\mathbf{C}^z$ is defined by (\ref{defcz}). Relations
(\ref{eq:DB_tsr_No}) and (\ref{eq:DB_tsr_Nt}) still hold for the
truncated CGPTs of finite order, due to the triangular shape of
the matrix $\mathbf{C}^z$.
Using the symmetry  of the CGPTs (\cite[Theorem
4.11]{ammari_polarization_2007}) and the positivity of the GPTs as
proved in \cite{ammari_polarization_2007}, we  easily establish
the following result.
\begin{proposition}
  \label{prop:CGPT_symm_herm}
  The complex CGPT matrix $\No$ is symmetric: $(\No)^t = \No$, and $\Nt$ is
  Hermitian: $(\Nt)^H = \Nt$. Consequently, the diagonal elements of
  $\Nt$ are strictly positive if $\lambda >0$ and strictly negative if $\lambda <0$.
\end{proposition}

Furthermore, the CGPTs of rotation invariant shapes have special
structures:
\begin{proposition}
  \label{prop:CGPT_rotsymm_struct}
  Suppose that $D$ is invariant under rotation of angle $2\pi/p$ for
  some integer $p\geq 2$, {\it i.e.}, $R_{2\pi/p}D = D$, then
  \begin{align}
    \No_{mn}(D) = 0, \mbox{ if } p \mbox{ does not divide }  (m+n) \label{eq:CGPT_struct_No},\\
    \nm
    \Nt_{mn}(D) = 0, \mbox{ if } p \mbox{ does not divide } (m-n).
    \label{eq:CGPT_struct_Nt}
  \end{align}
\end{proposition}
\begin{proof}
  Suppose that $p$ does not divide $(m+n)$, and define $r:=2\pi(n+m)/p \mbox{ mod }
  2\pi$. Then by the rotation symmetry of $D$ and the symmetry property
  of the
  CGPTs, we have
  $$
  \No_{mn}(D) = \No_{mn}(R_{2\pi /p}D) = e^{i(m+n)2\pi/p}\No_{mn}(D) =
  e^{ir}\No_{mn}(D).
  $$
  Since $r<2\pi$ and $r\neq 0$, we conclude that $\No_{mn}(D)=0$. The
  proof of \eqref{eq:CGPT_struct_Nt} is similar.
\end{proof}

\section{Shape Identification by the CGPTs}\label{sec:shape-ident-cgpt}

We call a \emph{dictionary} $\dico$  a collection of standard
shapes, which are centered at the origin and with characteristic
sizes of order 1. Given the CGPTs of an unknown shape $D$, and
assuming that $D$ is obtained from a certain element $B\in\dico$
by applying some unknown rotation $\theta$, scaling $s$ and
translation $z$, \ie, $D=T_zsR_\theta B$, our objective is to
recognize $B$ from $\dico$. For doing so, one may proceed by first
reconstructing the shape $D$ using its CGPTs through some
optimization procedures as proposed in \cite{AKLZ12}, and then
match the reconstructed shape with $\dico$. However, such a method
may be time-consuming and the recognition efficiency depends on
the shape reconstruction algorithm.

We propose in subsections \ref{sec:cgpt-matching} and
\ref{sec:transf-invar-shape} two shape identification algorithms
using the CGPTs. The first one matches the CGPTs of data with that
of the dictionary element by estimating the transform parameters,
while the second one is based on a transform invariant shape
descriptor obtained from the CGPTs. The second approach is
computationally more efficient. Both of them operate directly in
the data domain which consists of CGPTs and avoid the need for
reconstructing the shape $D$. The heart of our approach is some
basic algebraic equations between the CGPTs of $D$ and $B$ that
can be deduced easily from \eqref{eq:DB_tsr_No} and
\eqref{eq:DB_tsr_Nt}. Particularly, the first four equations read:
\begin{align}
  \No_{11}(D) &= w^2 \No_{11}(B), \label{eq:No11}\\
  \No_{12}(D) &= 2\No_{11}(D)z + w^3 \No_{12}(B), \label{eq:No12}\\
  \Nt_{11}(D) &= s^2 \Nt_{11}(B),\label{eq:Nt11}\\
  \Nt_{12}(D) &= 2\Nt_{11}(D)z + s^2w\Nt_{12}(B),\label{eq:Nt12}
\end{align}
where $w=s e^{i \theta}$.

\subsection{CGPTs matching}
\label{sec:cgpt-matching}

\subsubsection{Determination of transform
  parameters}\label{sec:determ-transf-param}
Suppose that the complex CGPT matrices $\No(B), \Nt(B)$ of the
true shape $B$ are given. Then, from \eqref{eq:Nt11}, we obtain
that
\begin{align}
  s = \sqrt{\Nt_{11}(D) / \Nt_{11}(B)}. \label{eq:scaling}
\end{align}

\paragraph{Case 1: Rotational symmetric shape.}
If the shape $B$ has rotational symmetry, \ie, $R_{2\pi /p}B =B$
for some $p\geq 2$, then from Proposition
\ref{prop:CGPT_rotsymm_struct} we have $\Nt_{12}(B)=0$ and the
translation parameter $z$ is uniquely determined from
\eqref{eq:Nt12} by
\begin{align}
  z = \frac{\Nt_{12}(D)}{2\Nt_{11}(D)}. \label{eq:trans_P2}
\end{align}
On the contrary, the rotation parameter $\theta$ (or $\yr$) can
only be determined up to a multiple of ${2\pi/p}$, from CGPTs of
order $\lceil p/2 \rceil$ at least. Although explicit expressions
of $e^{ip\theta}$ can be deduced from
\eqref{eq:No11}~-~\eqref{eq:Nt12} (or higher order equations if
necessary), we propose to recover $e^{ip\theta}$ by solving the
least squares problem:
\begin{align}
  \min_\theta \left(\norm{\No(\tsr B) - \No(D)}_F^2 + \norm{\Nt(\tsr
      B) - \Nt(D)}_F^2\right).  \label{eq:rot_lst}
\end{align}
Here, $s$ and $z$ are given by (\ref{eq:scaling}) and
(\ref{eq:trans_P2}) respectively, and $\No(D)$ and $\Nt(D)$ are
the truncated complex CGPTs matrices of dimension $\lceil p/2
\rceil \times \lceil p/2 \rceil$.


\paragraph{Case 2: Non rotational symmetric shape.}
Consider a non rotational symmetric shape $B$ which satisfies the
assumption:
\begin{align}
  \label{eq:cond_P1}
 \No_{11}(B) \neq 0 \quad \mbox{and} \quad  \mbox{det}
  \begin{pmatrix}
    \No_{11}(B) & \Nt_{11}(B) \\
    \No_{12}(B) & \Nt_{12}(B)
  \end{pmatrix} \neq 0.
\end{align}
From \eqref{eq:No12} and \eqref{eq:Nt12}, it follows that we can
uniquely determine the translation $z$ and the rotation parameter
$w= e^{i\theta}$ from CGPTs of orders one and two by solving the
following linear system:
\begin{align}
  \label{eq:linsys_P1}
  \No_{12}(D)/\No_{11}(D) &= 2z + w \No_{12}(B)/\No_{11}(B), \notag \\
  \nm
  \Nt_{12}(D)/\Nt_{11}(D) &= 2z + w \Nt_{12}(B)/\Nt_{11}(B).
\end{align}


\subsubsection{Debiasing by least squares solutions} \label{secdeb} In practice (for both the rotational symmetric and non rotational
symmetric cases), the value of the parameters $z, s$ and $\theta$
provided by the analytical formulas and numerical procedures above
may be inexact, due to the noise in the data and the
ill-conditioned character of the linear system
(\ref{eq:linsys_P1}). Let $z^*, s^*, \theta^*$ be the true
transform parameters, which can be considered as  perturbations
around the estimations $z, s, \theta$ obtained above:
\begin{align}
  z^*=z+\delta_z, \ s^*=s \delta_s, \mbox{ and } \theta^*=\theta+\delta_\theta , \label{eq:parameters_pertub}
\end{align}
for $\delta_z, \delta_\theta$ small and $\delta_s$ close to 1. To
find these perturbations, we solve a nonlinear least squares
problem:
\begin{align}
  \min_{z',s',\theta'}
  \left(\norm{\No(T_{z'}s'R_{\theta'} B) -
      \No(D)}_F^2 +
    \norm{\Nt(T_{z'}s'R_{\theta'} B) - \Nt(
      D)}_F^2\right) , \label{eq:tsr_lst_debiasing}
\end{align}
with $(z,s,\theta)$ as an initial guess. Here, the order of the
CGPTs in (\ref{eq:tsr_lst_debiasing}) is taken to be $2$ in the
non rotational case and $\max(2, [p/2])$ in the rotational
symmetric case. Thanks to the relations \eqref{eq:DB_tsr_No} and
\eqref{eq:DB_tsr_Nt}, one can calculate explicitly the derivatives
of the objective function, therefore can solve
\eqref{eq:tsr_lst_debiasing} by means of standard gradient-based
optimization methods.

\subsubsection{First algorithm for shape identification}\label{sec:first-algor-shape}
For each dictionary element, we determine the transform parameters
as above, then measure the similarity of the complex CGPT matrices
using the Frobenius norm, and choose the most similar element as
the identified shape. Intuitively, the true dictionary element
will give the correct transform parameters hence the most similar
CGPTs. This procedure is described in Algorithm
\ref{algo:shape-ident-cgpt}.

\begin{algorithm}
  \caption{Shape identification based on CGPT matching}
  \label{algo:shape-ident-cgpt}
  \begin{algorithmic}
    \STATE Input: the first $k$-th order CGPTs $\No(D), \Nt(D)$ of an unknown shape $D$
    \FOR {$B_n\in\dico$}
    \STATE 1. Estimation of $z, s, \theta$ using the procedures described in
    subsections~\ref{sec:determ-transf-param} and \ref{secdeb};
    \STATE 2. $\tilde D \leftarrow R_{-\theta}s^{-1}T_{-z}D$, and calculate $\No(\tilde D)$ and  $\Nt(\tilde
    D)$;
    \STATE 3. $E^{(1)}\leftarrow \No(B_n) - \No(\tilde D)$, and $E^{(2)}\leftarrow \Nt(B_n) - \Nt(\tilde
    D)$;
    \STATE 4.
    $e_n\leftarrow (\norm{E^{(1)}}_F^2 + \norm{E^{(2)}}_F^2)^{1/2}/
    (\norm{\No(B_n)}_F^2 + \norm{\Nt(B_n)}_F^2)^{1/2}$;
    \STATE 5. $n\leftarrow n+1$;
    \ENDFOR
    \STATE Output: the true dictionary element $n^*\leftarrow \mbox{argmin}_n
    e_n$.
  \end{algorithmic}
\end{algorithm}

\subsection{Transform invariant shape descriptors}
\label{sec:transf-invar-shape} From \eqref{eq:Nt11} and
\eqref{eq:Nt12} we deduce the following identity:
\begin{align}
  \frac{\Nt_{12}(D)}{2\Nt_{11}(D)} = z +
  s\yr\frac{\Nt_{12}(B)}{2\Nt_{11}(B)}, \label{eq:centroid}
\end{align}
which is well defined since $\Nt_{11}\neq 0$ thanks to the
Proposition \ref{prop:CGPT_symm_herm}. Identity
\eqref{eq:centroid} shows a very simple relationship between
$\frac{\Nt_{12}(B)}{2\Nt_{11}(B)}$ and
$\frac{\Nt_{12}(D)}{2\Nt_{11}(D)}$ for $D= T_z s R_\theta B$. .

Let $u=\frac{\Nt_{12}(D)}{2\Nt_{11}(D)}$. We first define  the
following quantities which are translation invariant:
\begin{align}
  \label{eq:shape_descrp_trans}
  \Tauo(D) &= \No(T_{-u}D) = \mathbf{C}^{-u}\No(D)(\mathbf{C}^{-u})^t, \\
  \nm
  \Taut(D) &= \Nt(T_{-u}D) = \overline{\mathbf{C}^{-u}}\Nt(D)(\mathbf{C}^{-u})^t,
\end{align}
with the matrix $\mathbf{C}^{-u}$ being the same as in Proposition
\ref{prop:complex-cgpt-under}. From $\Tauo(D)
=(\Tauo_{mm}(D))_{m,n}$ and $\Taut(D) =(\Taut_{mm}(D))_{m,n}$, we
define, for any indices $m,n$, the scaling invariant quantities:
\begin{align}
  \label{eq:shape_descrp_scl}
  \Sclo_{mn}(D) =
  \frac{\Tauo_{mn}(D)}{\left(\Taut_{mm}(D)\Taut_{nn}(D)\right)^{1/2}}, \
  \Sclt_{mn}(D) =
  \frac{\Taut_{mn}(D)}{\left(\Taut_{mm}(D)\Taut_{nn}(D)\right)^{1/2}}
  .
\end{align}
Finally, we introduce the CGPT-based shape descriptors $\Dcrpo = (
\Dcrpo_{mn})_{m,n}$ and $\Dcrpt =  (\Dcrpt_{mn})_{m,n}$:
\begin{align}
  \label{eq:shape_descrp}
  \Dcrpo_{mn}(D) = |\Sclo_{mn}(D)|, \ \Dcrpt_{mn}(D) =
  |\Sclt_{mn}(D)|,
\end{align}
where $|\cdot|$ denotes the modulus of a complex number.
Constructed in this way, $\Dcrpo$ and $\Dcrpt$ are clearly
invariant under translation, rotation, and scaling.

It is worth emphasizing the symmetry property, $\Dcrpo_{mn} =
\Dcrpo_{nm}, \Dcrpt_{mn} = \Dcrpt_{nm}$, and the fact that
$\Dcrpt_{mm} =1$ for any $m$.

\subsubsection{Second algorithm for shape identification}\label{sec:second-algor-shape}
Thanks to the transform invariance of the new shape descriptors,
there is no need now for calculating the transform parameters, and
the similarity between a dictionary element and the unknown shape
can be directly measured from $\Dcrpo$ and $\Dcrpt$. As in
Algorithm~\ref{algo:shape-ident-cgpt}, we use the Frobenius norm
as the distance between two shape descriptors and compare with all
the elements of the dictionary. We propose a simplified method for
shape identification, as described in Algorithm
\ref{algo:shape-ident-inv}.

\begin{algorithm}
  \caption{Shape identification based on transform invariant descriptors}
  \label{algo:shape-ident-inv}
  \begin{algorithmic}
    \STATE Input: the first $k$-th order shape descriptors $\Dcrpo(D), \Dcrpt(D)$ of an unknown shape $D$
    \FOR {$B_n\in\dico$}

    \STATE 1.
    $e_n\leftarrow \left(\norm{\Dcrpo(B_n) - \Dcrpo(D)}_F^2 + \norm{\Dcrpt(B_n) -
    \Dcrpt(D)}_F^2\right)^{1/2}$;
    \STATE 2. $n\leftarrow n+1$;
    \ENDFOR
    \STATE Output: the true dictionary element $n^*\leftarrow \mbox{argmin}_n
    e_n$.
  \end{algorithmic}
\end{algorithm}

\section{Numerical Experiments} \label{sec:numer-exper}

In this section we present a variety of numerical results on the
theoretical framework discussed in this paper in the context of
target identification from noisy MSR  measurements. Given a shape
$D_0$ of characteristic size $\delta$, the procedure of our
numerical experiment can be summarized as follows:
\begin{enumerate}
\item Data simulation. $N$ sources (and also receivers) are
equally
  distributed on a circle of radius $R$, which is centered at an
  arbitrary point $z_0\in D_0$ and includes $D_0$, see Figure
  \ref{fig:data_acq}. The MSR matrix is obtained by evaluating
  numerically its integral expression \eqref{eq:dfield} then adding a
  white noise of variance $\sigma_{\mathrm{noise}}^2$. For simplicity, here
  we suppose that the reference point $z_0\in D_0$ can be estimated by
  means of  algorithms such as MUSIC (standing for MUltiple SIgnal Classification) \cite{AGJ, ammari_polarization_2007}.
\item Reconstruction of the CGPTs of $D=D_0-z_0$ using formula
  \eqref{eq:Mest} or the least squares algorithm \eqref{eq:lsqr}.
\item For a given dictionary $\dico$, apply
  Algorithm~\ref{algo:shape-ident-cgpt} (or
  Algorithm~\ref{algo:shape-ident-inv}) using the CGPTs of $D$ and
  identify the true shape from $\dico$.
\end{enumerate}
We emphasize that the reconstructed CGPTs of shape $D$ depend on
the reference point $z_0$. We fix the conductivity parameter
$\kappa=4/3$ throughout this section.

\begin{figure}[htp]
  \centering
  \includegraphics[width=7.5cm]{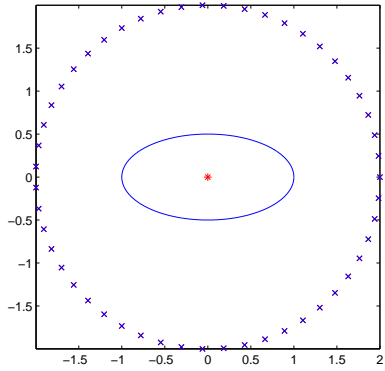}
  \caption{An example of the configuration for MSR data
    simulation. The unknown shape is an ellipse whose long and short
    axes are 2 and 1, respectively. $N=51$ sources/receivers (marked by
    ``x'') are equally placed on a circle of radius $R=2$ centered at
    $z_0=[0,0]$ (marked by ``*'').}
  \label{fig:data_acq}
\end{figure}

\subsection{Reconstruction of CGPTs}\label{sec:reconstruction-cgpt}
The theoretical analysis presented in section
\ref{sec:reconstr-cgpt-stab} suggests the following two step
method for the reconstruction of CGPTs. First we apply
\eqref{eq:Mest} (or equivalently solve the least squares
problem~\eqref{eq:lsqr}) by fixing the truncation order $K$ as in
\eqref{eq:nregime}:
\begin{align}
  K \leq \min\left ( \frac{\log(\sigma_{\mathrm{noise}}/N)}{\log\varepsilon}-2, N/2\right
  ).
  \label{eq:max_trunc_ord}
\end{align}
Here, $\sigma_{\mathrm{noise}}$ is the standard deviation of the
measurement noise and $\eps = \delta/R$ with $\delta$ being the
characteristic size of the target and $R$ the distance between the
target center and the circular array of transmitters/receivers.
Then, we keep  only the first $m_0$ orders in the reconstructed
CGPTs, with $m_0$ being the resolving order deduced from
estimation \eqref{eq:snrest}:
\begin{align}
  m_0 =
  \frac{\log\sigma_{\mathrm{noise}} - \log \tau_0}{2\log\varepsilon}, \label{eq:max_CGPT_order}
\end{align}
and $\tau_0\leq 1$ is the tolerance number introduced in
(\ref{eq:msdef}). In all our numerical experiments we set the
noise level $\sigma_{\mathrm{noise}}$ to:
\begin{align}
  \label{eq:noise_model}
  \sigma_{\mathrm{noise}} = (\V_{\mbox{max}} -
  \V_{\mbox{min}})\sigma_0,
\end{align}
with a positive constant $\sigma_0$ and $\V_{\mbox{max}}$ and
$\V_{\mbox{min}}$ being the maximal and the minimal coefficient in
the MSR matrix $\V$. Using the configuration given in
Figure~\ref{fig:data_acq} and for various noise level, we
reconstruct the CGPTs of the ellipse up to a truncation order $K$
which is determined as in \eqref{eq:max_trunc_ord}. For each $k
\leq K$, the relative error of the first $k$-th order
reconstructed CGPTs is evaluated by comparing with their
theoretical value (\cite[Proposition
4.7]{ammari_polarization_2007}). The results are shown in
Figure~\ref{fig:err_rec_CGPT}. In Figure
\ref{fig:rslv_ord_rel_err_CGPT} we plot the resolving order $m_0$
given by \eqref{eq:max_CGPT_order} and the relative error of the
reconstruction within this order, for $\sigma_0$ in the range
$[10^{-3}, 1]$.


\begin{figure}[htp]
  \centering
  \subfigure[$\sigma_0=0.01, m_0=6$]{\includegraphics[width=7.5cm]{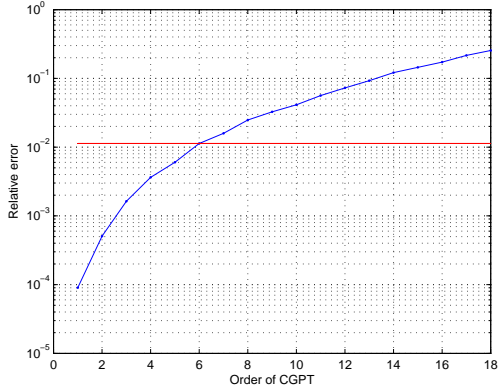}}
  \subfigure[$\sigma_0=0.1, m_0=4$]{\includegraphics[width=7.5cm]{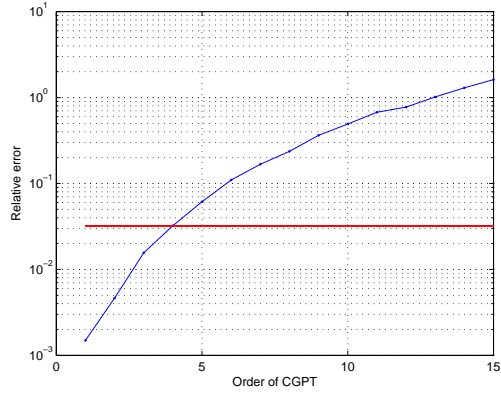}}
  \subfigure[$\sigma_0=0.5, m_0=3$]{\includegraphics[width=7.5cm]{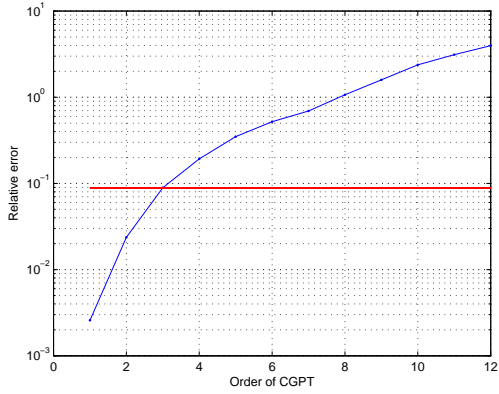}}
  \subfigure[$\sigma_0=1.0, m_0=2$]{\includegraphics[width=7.5cm]{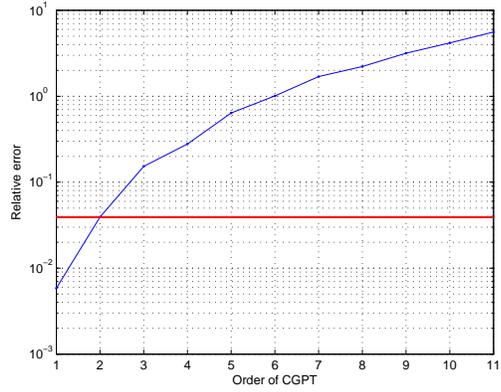}}
  \caption{Relative error of the reconstructed CGPTs. For each noise
    level, we repeat the experiment 100 times (corresponding to 100 realizations of the noise) and the
    reconstruction is taken as their mean value. The horizontal solid
    line in each figure indicates the resolving order $m_0$ given by
    \eqref{eq:max_CGPT_order} with the tolerance number $\tau_0=10^{-1}$.}
  \label{fig:err_rec_CGPT}
\end{figure}

\begin{figure}[htp]
  \centering
  \subfigure[Resolving order]{\includegraphics[width=7.5cm]{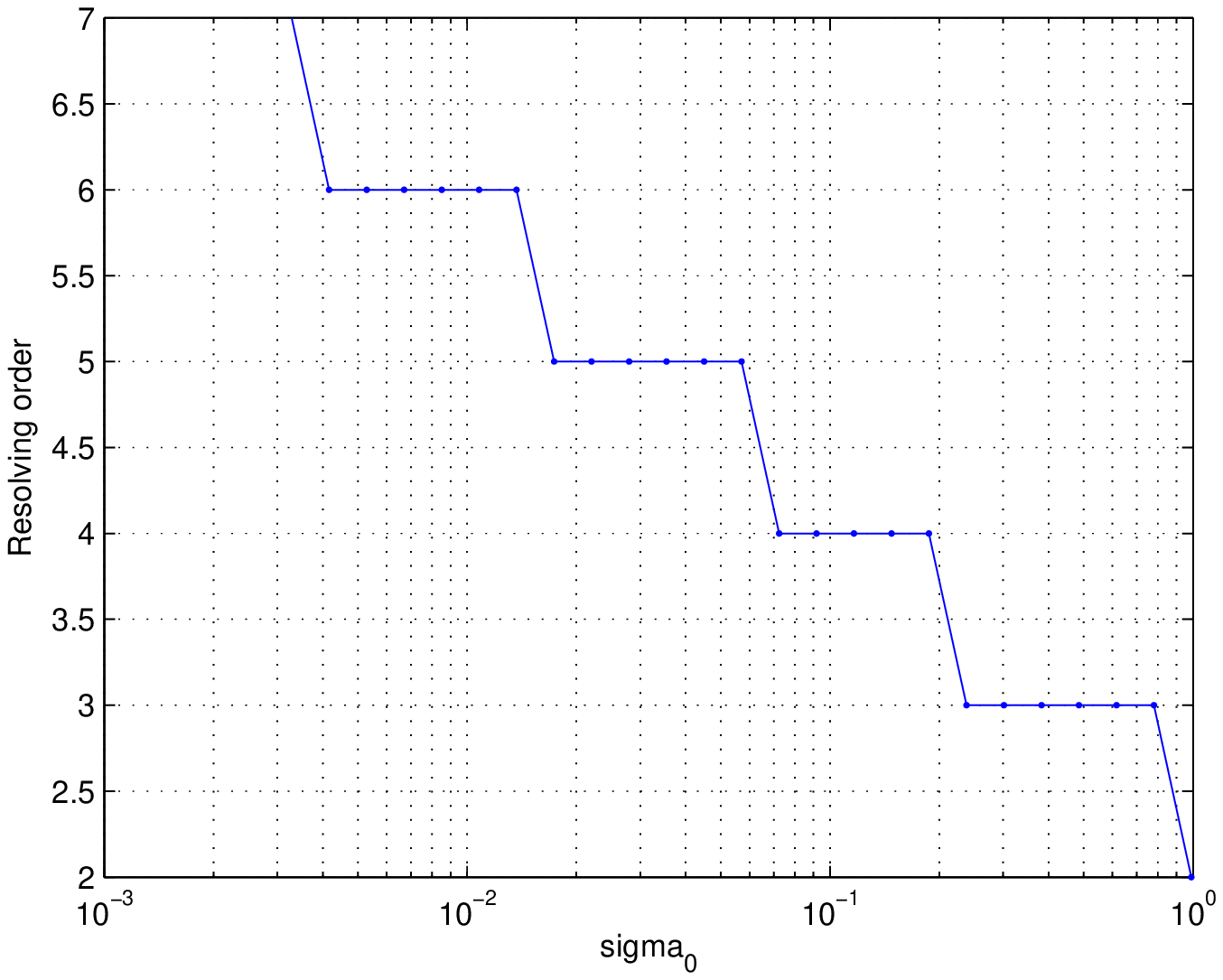}}
  \subfigure[Relative error]{\includegraphics[width=7.5cm]{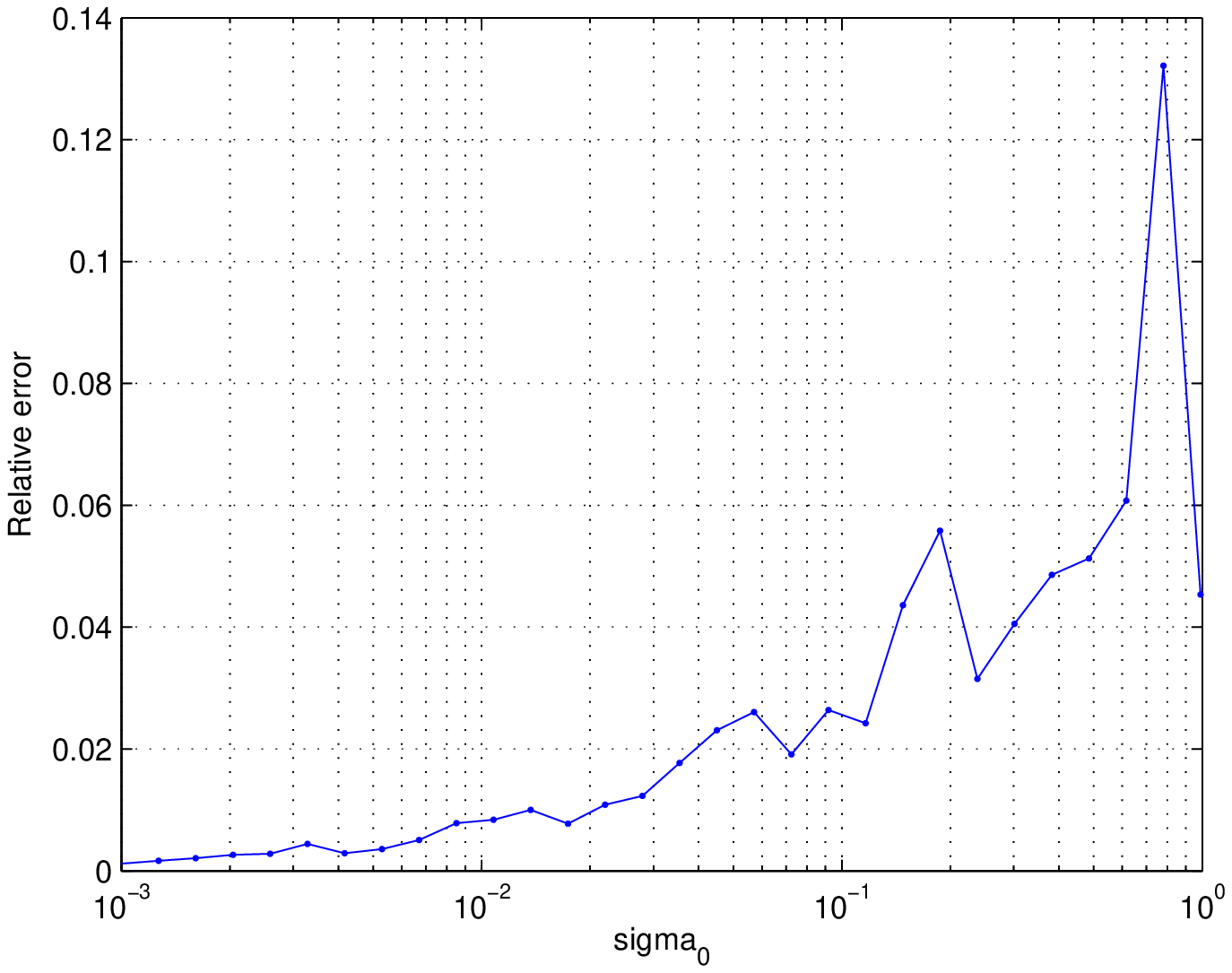}}
  \caption{The resolving order $m_0$, for $\sigma_0\in[10^{-3},1], \tau_0=10^{-1},$ and
    the relative error of the reconstruction within this order. As in
    Figure~\ref{fig:err_rec_CGPT}, we repeat the experiment 100
    times and the reconstruction is taken as their mean value. The
    large variations of the relative error in (b) for $\sigma_0>10^{-1}$ indicate the
    instability of the reconstruction for very noisy data.}
  \label{fig:rslv_ord_rel_err_CGPT}
\end{figure}

\subsection{Dictionary matching}
We are now ready to present the results of the dictionary matching
algorithms discussed in the sections~\ref{sec:cgpt-matching} and
\ref{sec:transf-invar-shape}. Unless specified, in the following
we suppose that the unknown shape of the target $D_0$ is an exact
copy of some element from the dictionary, up to a rigid transform
and dilatation. As examples, we consider a dictionary of flowers
and a dictionary of Roman letters. The aim is to identify the
target $D_0$ from imaging data if it belongs to one of the
dictionaries.

\subsubsection{Matching on a dictionary of flowers}\label{sec:match-dico-flower}
We start by considering a simple dictionary of rotation invariant
``flowers'', on which the shape identification algorithm can be
greatly simplified. The boundary of the $p$-th flower $B_p$ is
defined as a small perturbation of the standard disk:
\begin{align}
  \partial B_p(\xi) = x(\xi)(1+\eta \cos(p\xi)), \mbox{ } x(\xi)=
  \begin{pmatrix}
    \cos\xi\\ \sin\xi
  \end{pmatrix} ,
  \label{eq:flower}
\end{align}
where $p\geq 2$ is the number of petals and $\eta>0$ is a small
constant. According to Proposition \ref{prop:CGPT_rotsymm_struct},
$\No_{mn}(B_p)$ is zero if $p$ does not divide $m+n$.  For an
unknown shape $D=T_zsR_\theta B_p$, the translation parameter is
given by $z=\frac{\Nt_{12}(D)}{2\Nt_{11}(D)}$. Moreover, simple
calculations show that $\Dcrpo(D)$ and $\No(B_p)$ have exactly the
same zero patterns.

Therefore, we can find the true number of petals by searching the
first nonzero anti-diagonal entry in $\Dcrpo(D)$.

We fix $\eta=0.3$ (the amplitude of the perturbation introduced in
(\ref{eq:flower})) and $\delta/R=0.5$. The unknown shape $D_0$ is
obtained by applying the transform parameters $z=[16.3, -46.7],
s=7.5, \theta=2.69$ on $B_p$, and the reference point for data
acquisition is $z_0=[15, -45.5]$. The results for two flowers of 5
and 7 petals are shown in Figure~\ref{fig:matching_flower}, where
we plot the mean absolute value of the anti-diagonal entries $mn$,
for $m+n =l, l=2, \ldots, 11,$ in $\Dcrpo(D)$ by varying the noise
level $\sigma_0$. One can clearly distinguish the peak which
indicates the true number of petals for $\sigma_0$ up to
$10^{-2}$.

\begin{figure}[htp]
  \centering
  \subfigure[$p=5$]{\includegraphics[width=7.5cm]{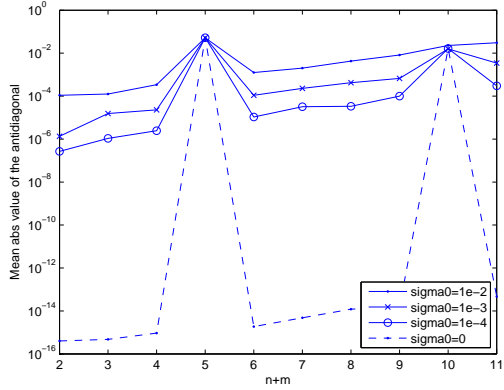}}
  \subfigure[$p=7$]{\includegraphics[width=7.5cm]{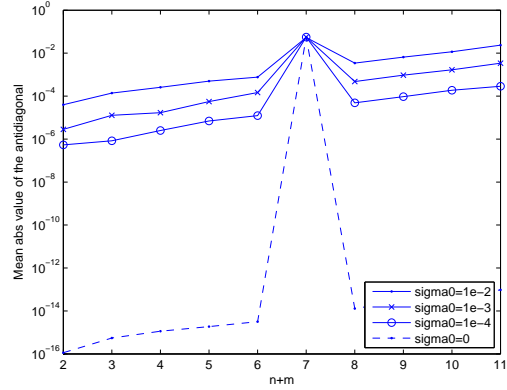}}
  \caption{Mean values of the anti-diagonal entries of $\Dcrpo$ for the flowers of
    5 and 7 petals at different noise levels.}
  \label{fig:matching_flower}
\end{figure}

\paragraph{Stability.}
Let us consider now the model \eqref{eq:flower} with a general
$\mathcal{C}^1$ function $h(\xi)$ in place of $\cos(p\xi)$. It was
proven in \cite{AGKLY11} that:
\begin{align}
  \No_{mn}(B_p) = 2\pi\eta\frac{mn}{\lambda^2}\hat
  h_{m+n} + O(\eta^2).
  \label{eq:flower_CGPT_Fourier}
\end{align}
Therefore as long as the perturbation $h(\xi)$ is close to
$\cos(p\xi)$, the significant nonzero coefficients in $\Dcrpo(D)$
will concentrate on the same anti-diagonals. We confirm this
observation by applying the same procedure above on a flower with
one damaged petal:
\begin{align}
  \label{eq:damaged_flower}
  \partial B_p(\xi) =
  \begin{cases}
    x(\xi)f(\xi, t) \ &\mbox{ for } \xi\in[0,2\pi/p),\\
    \nm
    x(\xi)(1+\eta \cos(p\xi)) \ &\mbox{ for } \xi\in[2\pi/p,
    2\pi).
  \end{cases}
\end{align}
Here, $f(\cdot, t):\R\mapsto\R$ is a polynomial of order 6,
constructed such that $\partial B_p$ is $\mathcal{C}^2$-smooth,
and $t \in (0,1)$ is the percentage of the damage; see
Figure~\ref{fig:imperfect_flower}. In
Figure~\ref{fig:matching_imperfect_flower} we plot the mean value
of the anti-diagonal entries at different noise levels. Compared
to Figure~\ref{fig:matching_flower}, we see that the effect of the
damage in the petal dominates the measurement noise. Nonetheless,
the peak indicating the true number of petals is still visible.

\begin{figure}[htp]
  \centering
  \subfigure[$\tau_0=0.5$]{\includegraphics[width=7.5cm]{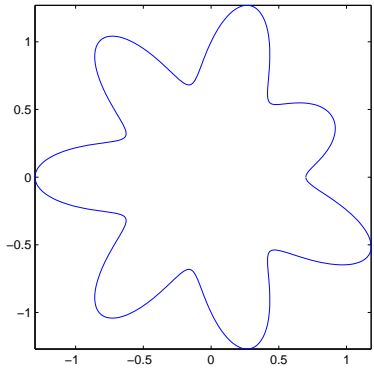}}
  \subfigure[$\tau_0=0.8$]{\includegraphics[width=7.5cm]{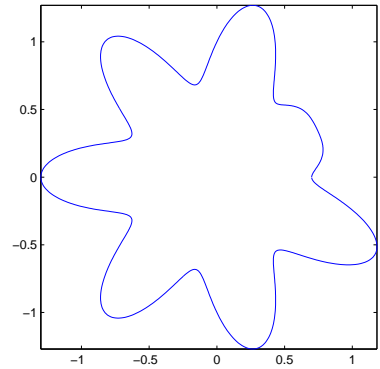}}
  \caption{Flowers with one damaged petal. The following parameters
    are used in \eqref{eq:damaged_flower}: $p=7$, $\eta=0.3$,
    $t=0.5$ for (a) and $t=0.8$ for (b).}
  \label{fig:imperfect_flower}
\end{figure}

\begin{figure}[htp]
  \centering
  \subfigure[$\tau_0=0.5$]{\includegraphics[width=7.5cm]{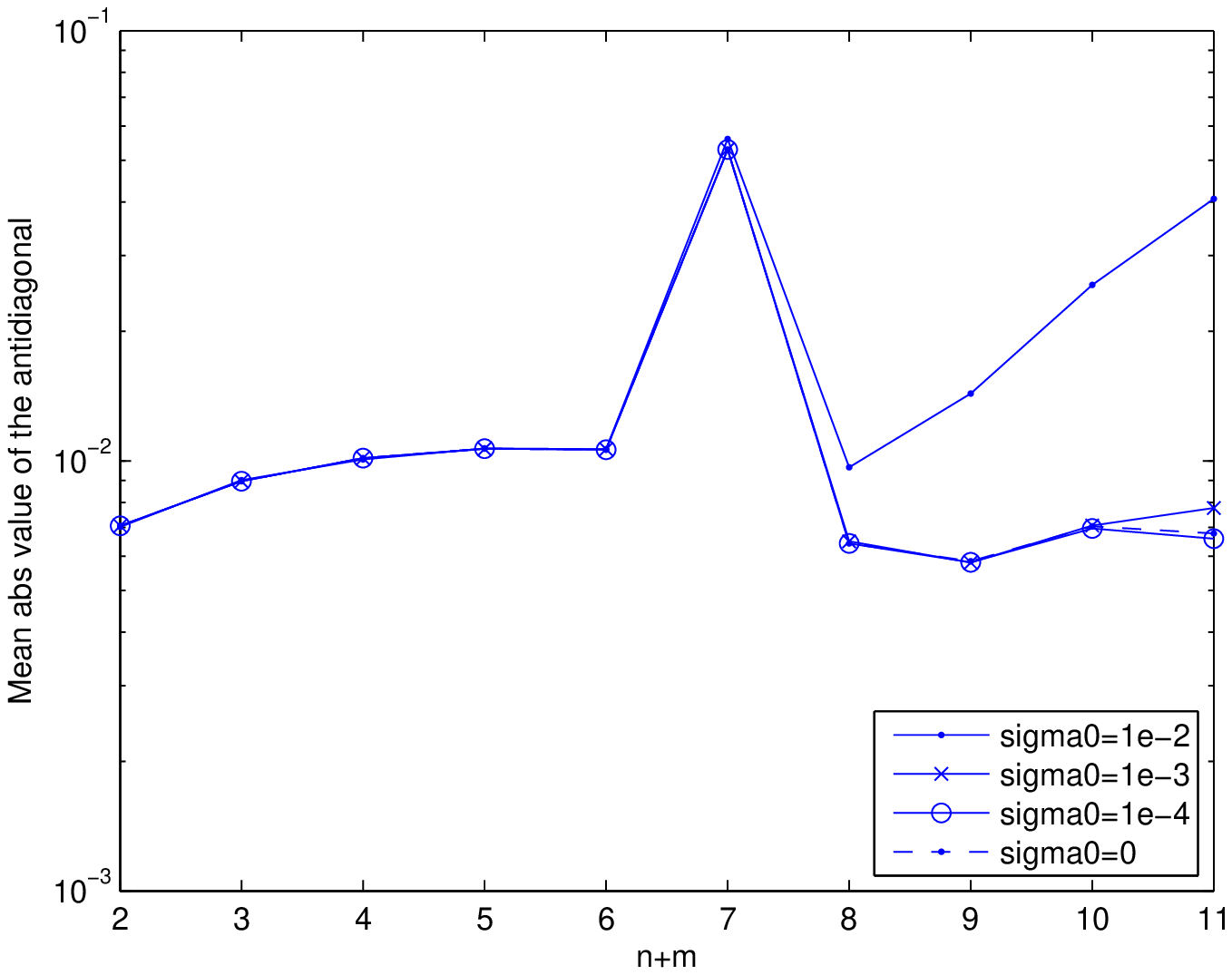}}
  \subfigure[$\tau_0=0.8$]{\includegraphics[width=7.5cm]{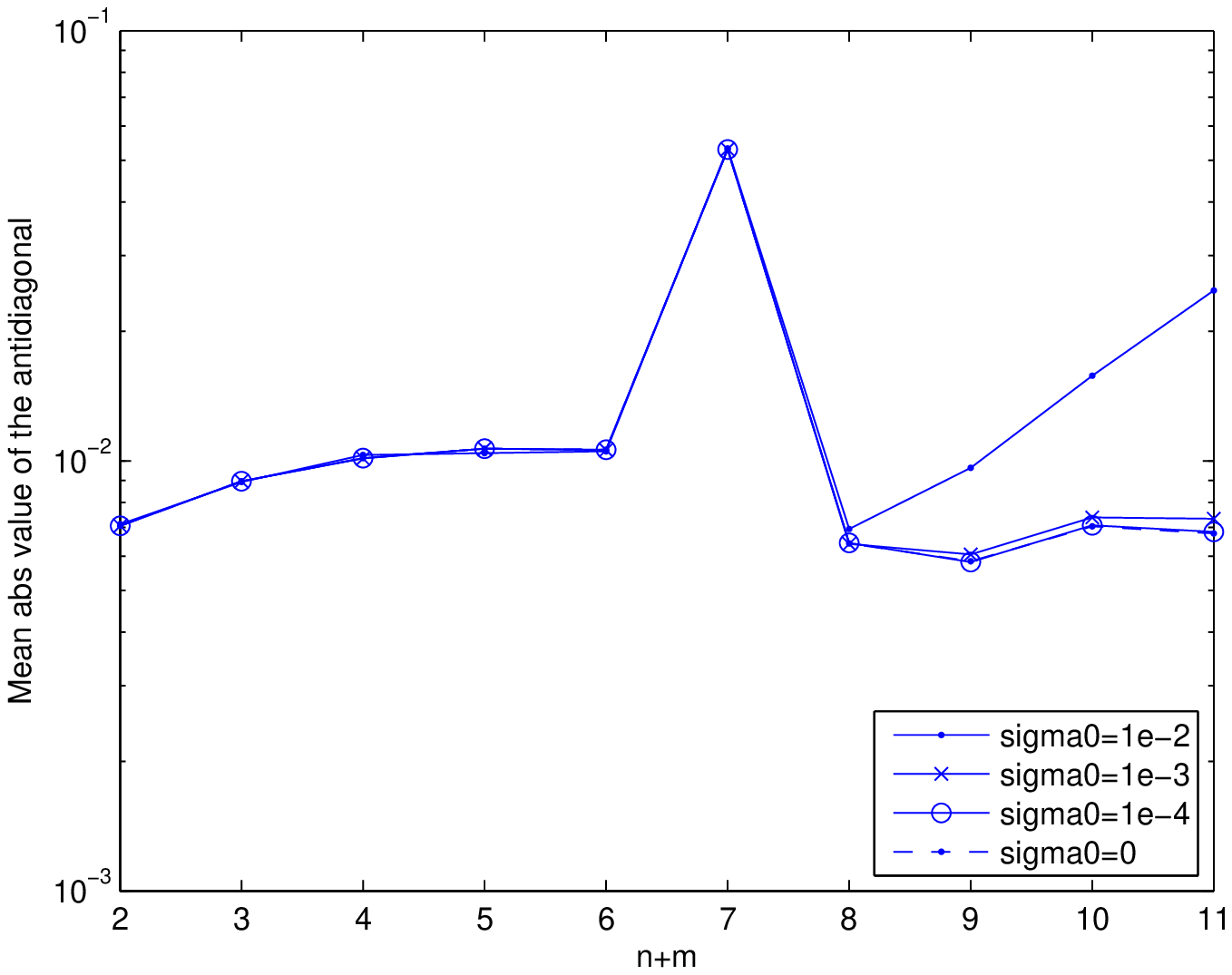}}
  \caption{Mean value of the anti-diagonal entries of $\Dcrpo$ for the flowers of
    Figure~\ref{fig:imperfect_flower} at different noise levels. The peaks indicate the number of petals.}
  \label{fig:matching_imperfect_flower}
\end{figure}

\subsubsection{Dictionary of letters}
Next we consider here a dictionary consisting of 26 Roman capital
letters without rotational symmetry. The shapes are defined in
such a way that the holes inside the letters are filled, see
Figure~\ref{fig:all_letters_A_Z}. We set $\delta/R=0.5$, $s=
2.4762, \theta = 6.0827, z= [33.3505, 73.8395]$ and the center of
mass of the target at $[33.4042, 73.8627]$.

\paragraph{Performance of Algorithm 1.}

First we test Algorithm \ref{algo:shape-ident-cgpt} on the letter
``P''. For the noiseless case ($\sigma_0=0$), the values of $e_n$
defined in Algorithm~\ref{algo:shape-ident-cgpt} are plotted in
Figure~\ref{fig:matching_letter_p} (a) and (b). These results
suggest that the high order CGPTs can better distinguish similar
shapes such as ``P'' and ``R'', since they contain more high
frequency information \cite{AGKLY11}. Nonetheless, the advantage
of using high order CGPTs drops quickly when the data are
contaminated by noise, and low order CGPTs provide more stable
results in this situation, see Figure~\ref{fig:matching_letter_p}
(c) and (d).

\begin{figure}[htp]
  \centering
  \subfigure[$\sigma_0=0$, order $\leq 2$]{\includegraphics[width=7.5cm]{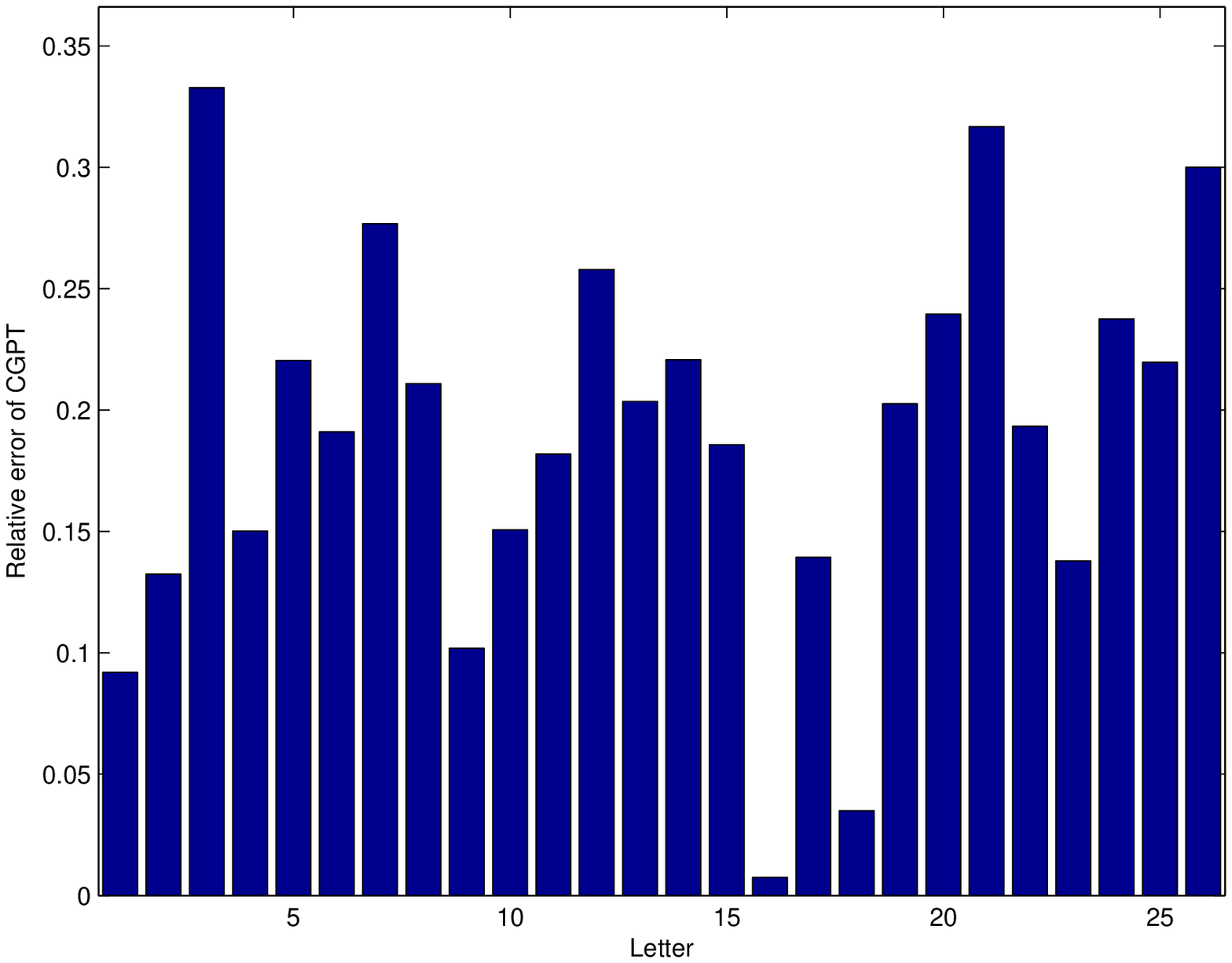}}
  \subfigure[$\sigma_0=0$, order $\leq 5$]{\includegraphics[width=7.5cm]{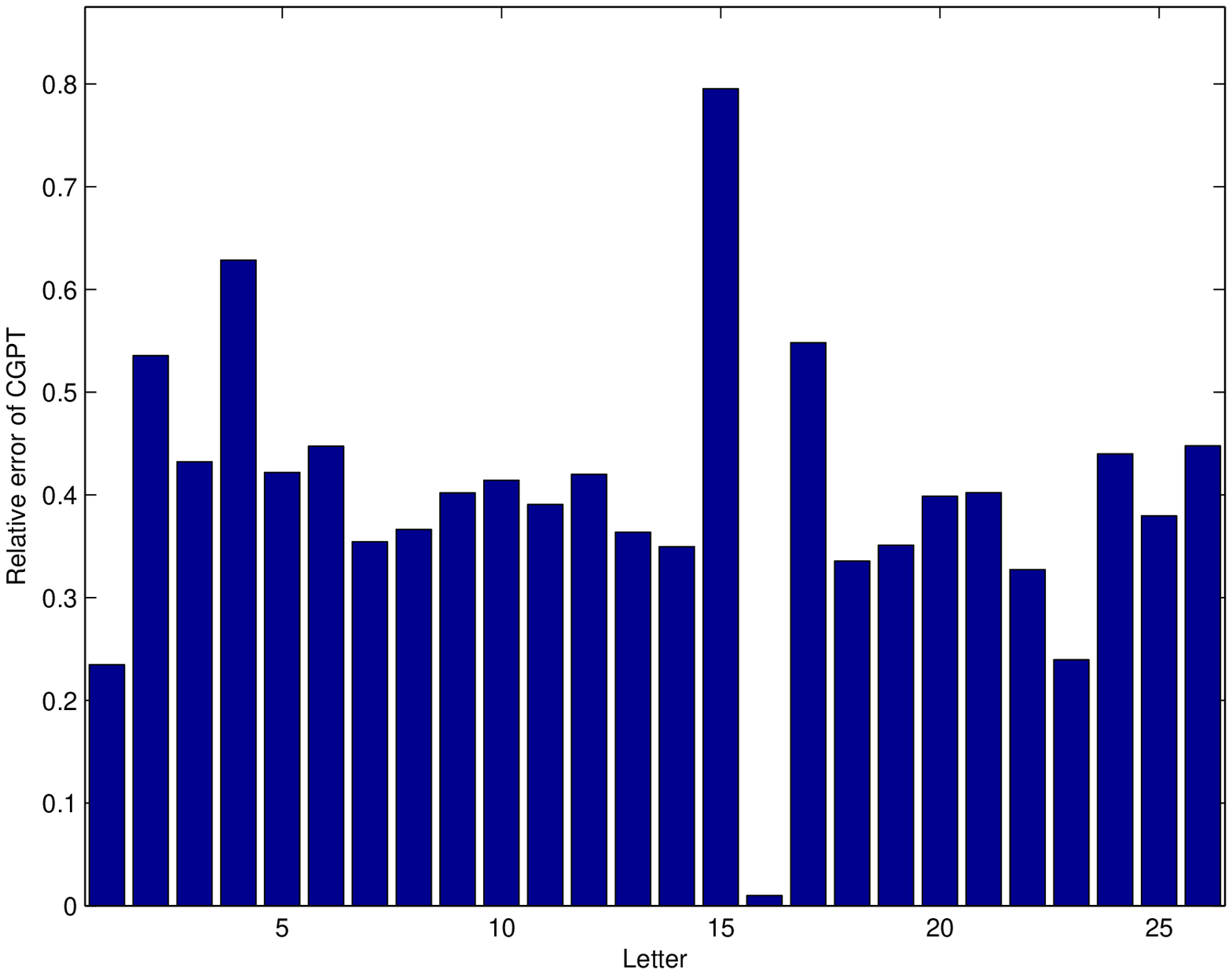}}
  \subfigure[$\sigma_0=0.1$, order $\leq 2$]{\includegraphics[width=7.5cm]{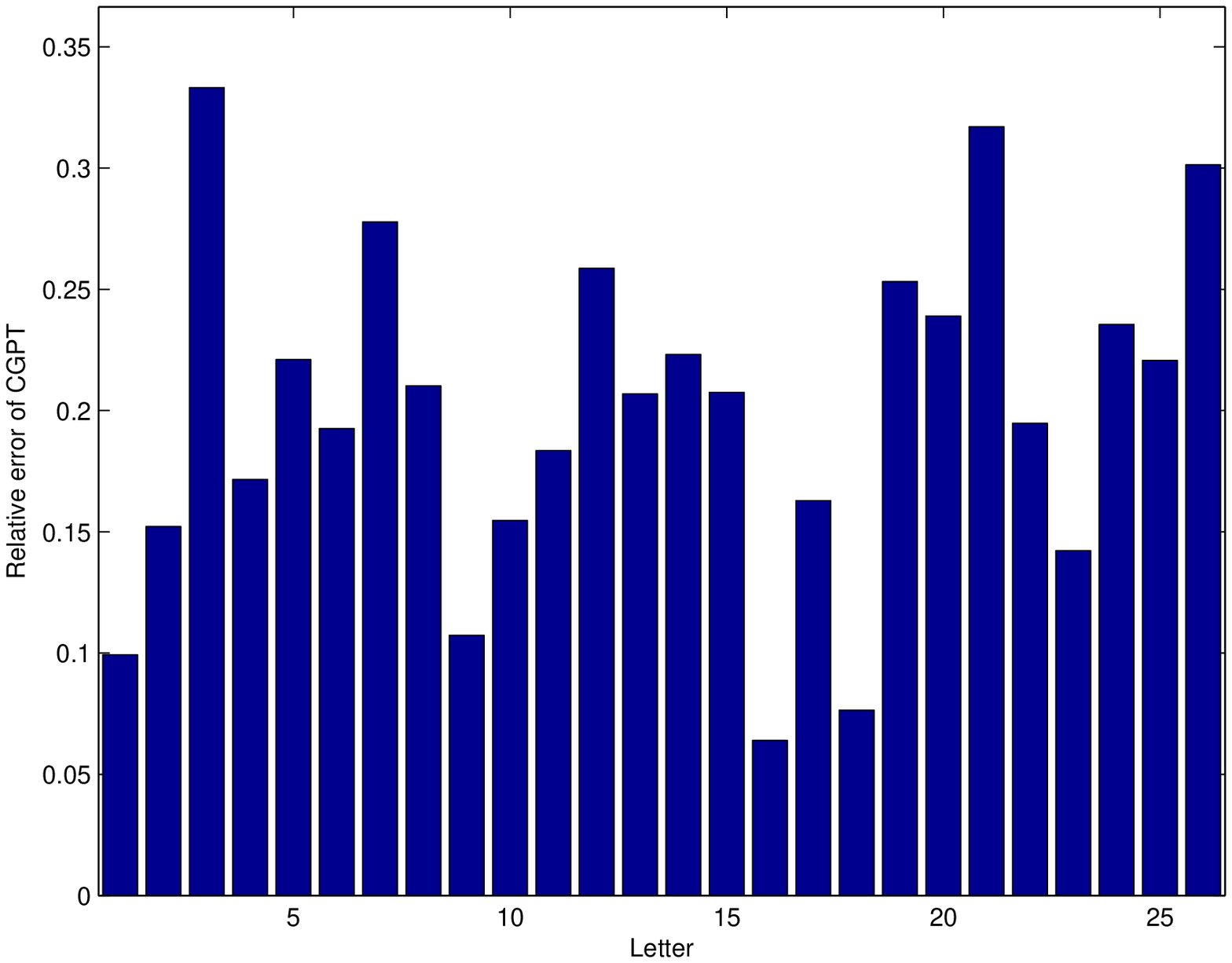}}
  \subfigure[$\sigma_0=0.1$, order $\leq 5$]{\includegraphics[width=7.5cm]{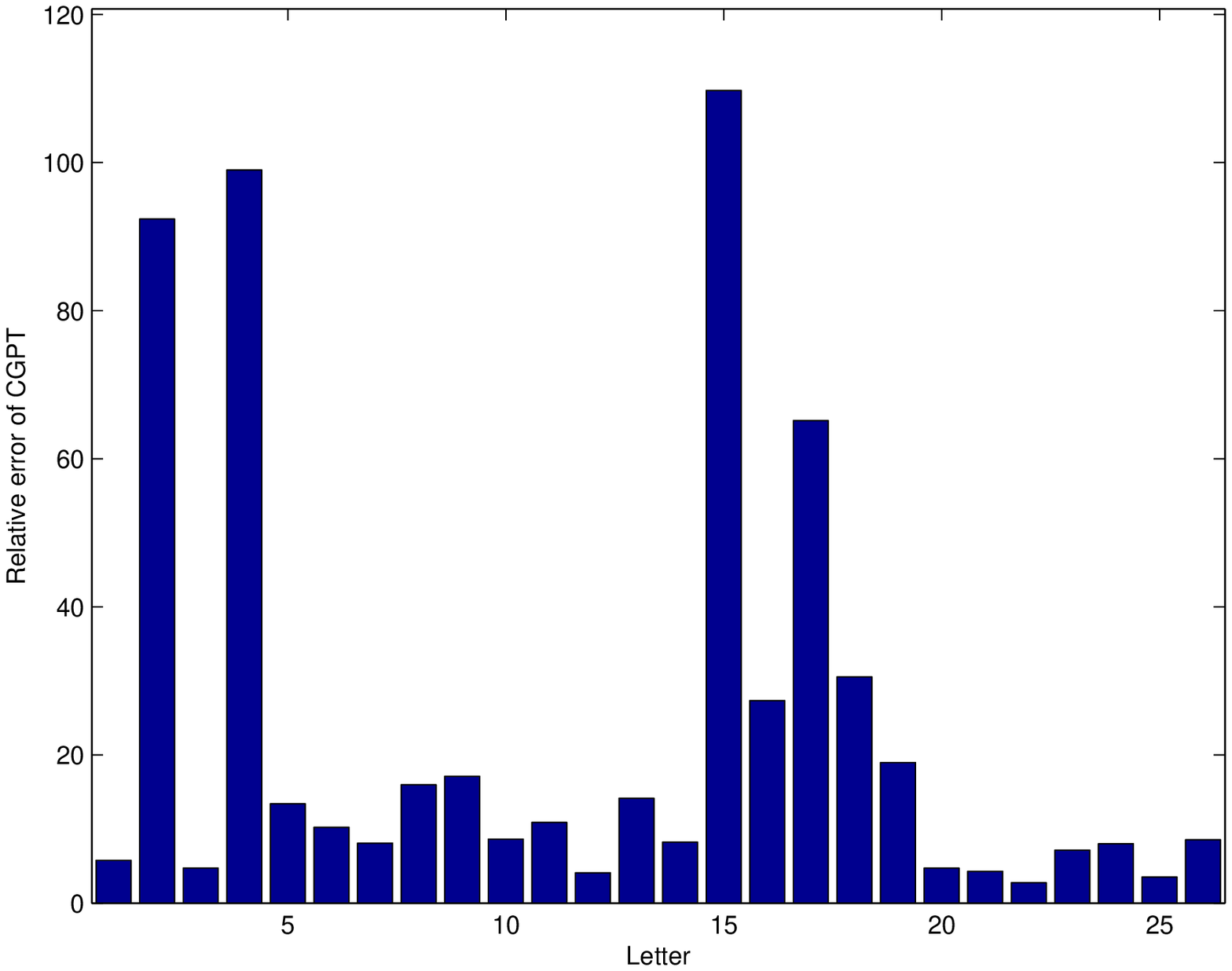}}
  \caption{The identification of the letter ``P'' using the first 2,
    and 5 orders CGPTs at noise levels $\sigma_0=0$ and
    $\sigma_0=0.1$. The bar represents the relative error $e_n$
    between the CGPTs of the $n$-th letter and that of the data, as
    defined in Algorithm~\ref{algo:shape-ident-cgpt}, and the shortest
    one in each figure corresponds to the identified letter. For (c)
    and (d), the experiment has been repeated for 100 times, using
    independent draws of white noise, and the results are the mean
    values of all experiments.}
  \label{fig:matching_letter_p}
\end{figure}

By repeating the same procedure as above, we apply
Algorithm~\ref{algo:shape-ident-cgpt} on all letters at noise
levels $\sigma_0=0$ and $\sigma_0=0.1$, and show the result in
Figure~\ref{fig:matching_all_letters} (a) and (c). At the
coordinate $(m,n)$, the unknown shape is the $m$-th letter and the
color represents the relative error (in logarithmic scale) of the
CGPTs when compared with the $n$-th standard letter of the
dictionary.


\begin{figure}[htp]
  \centering
  \subfigure[$\sigma_0=0$, order $\leq 5$, Standard letters]{\includegraphics[width=7.5cm]{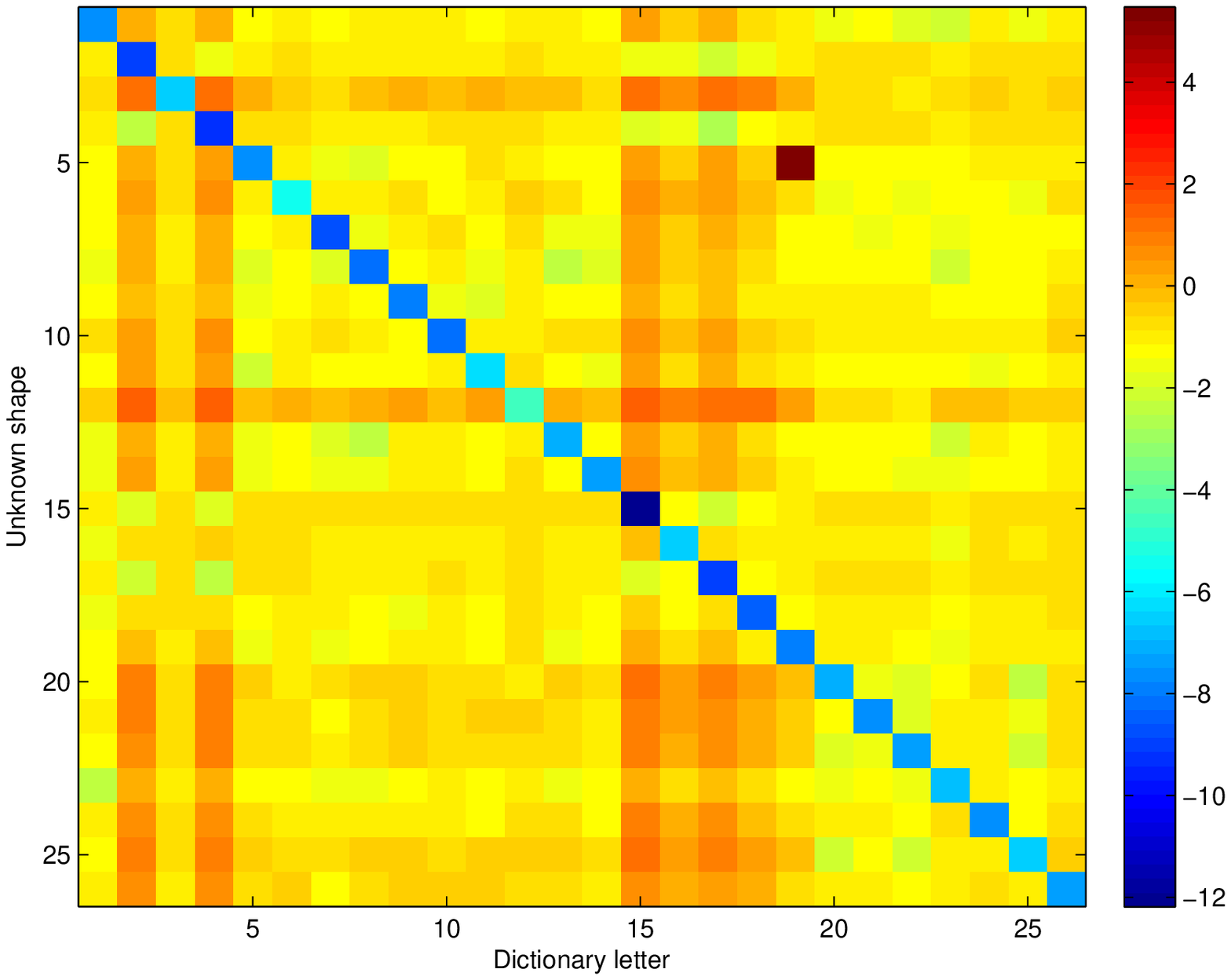}}
  \subfigure[$\sigma_0=0$, order $\leq 5$, Perturbed letters]{\includegraphics[width=7.5cm]{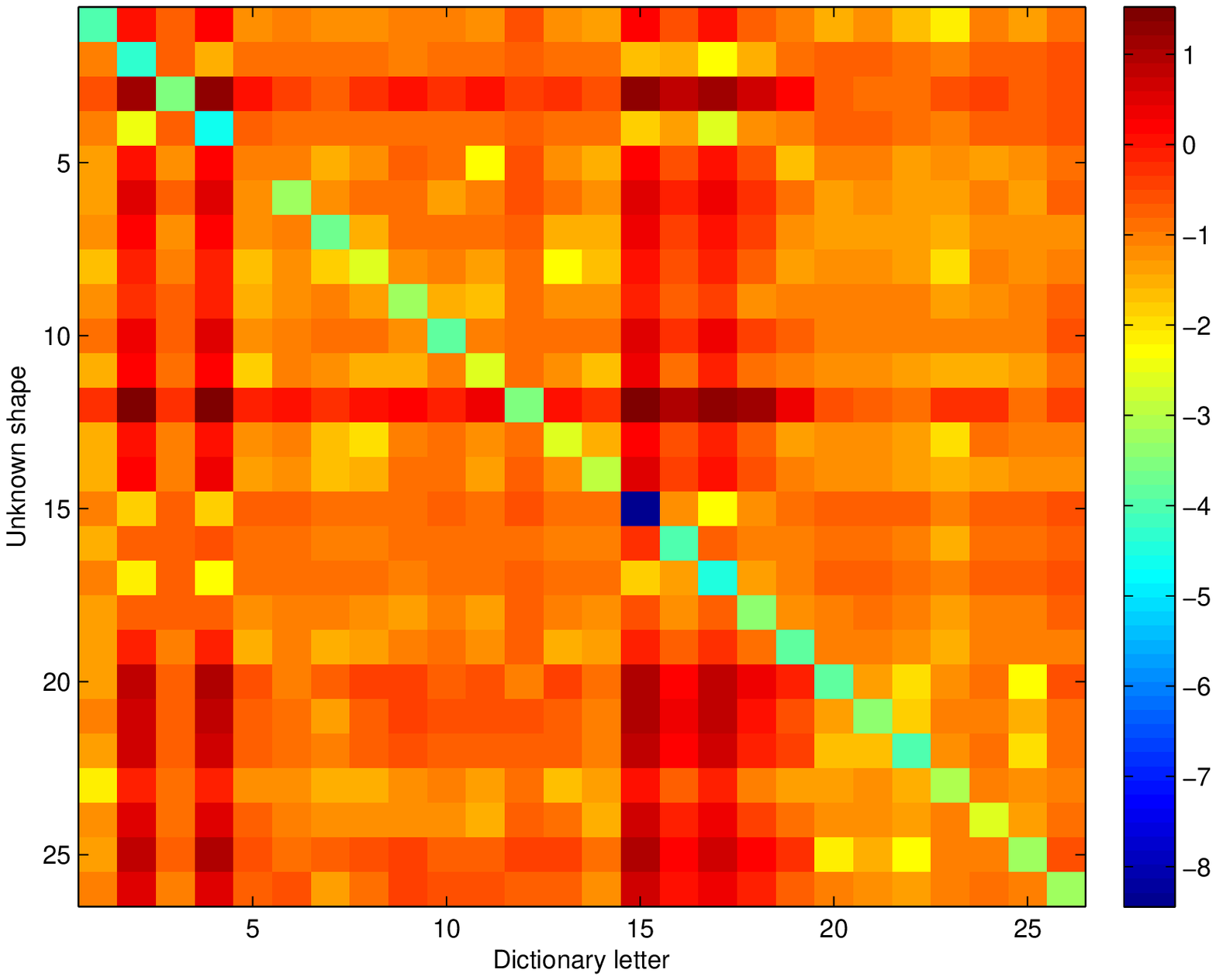}}
   \subfigure[$\sigma_0=0.1$, order $= 1$, Standard letters]{\includegraphics[width=7.5cm]{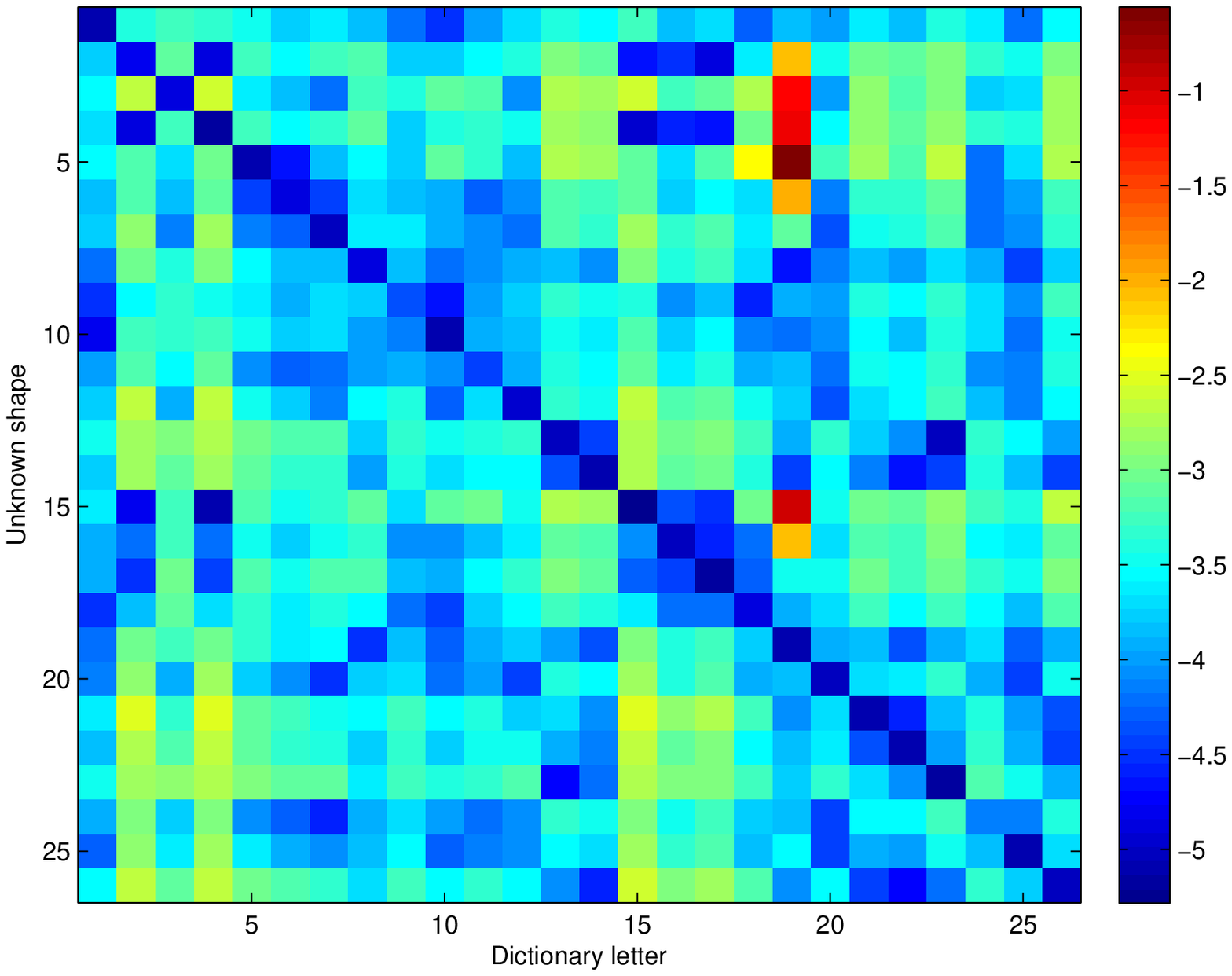}}
  \subfigure[$\sigma_0=0.1$, order $= 1$, Perturbed letters]{\includegraphics[width=7.5cm]{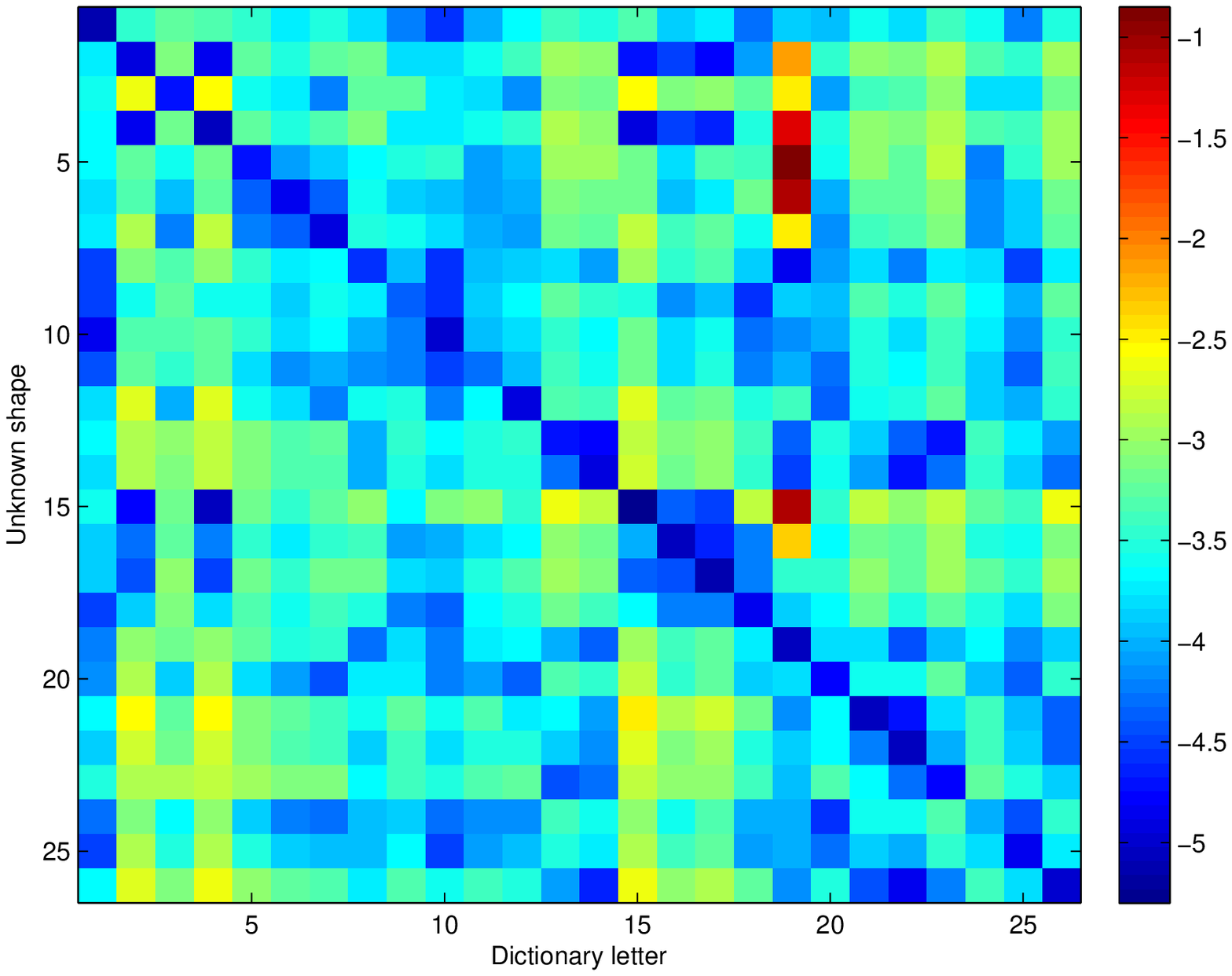}}
  \caption{Algorithm~\ref{algo:shape-ident-cgpt} applied on the all 26
    letters using the standard dictionary
    (Figure~\ref{fig:all_letters_A_Z}) at noise level $\sigma_0=0$
    (first column) and $\sigma_0=0.1$ (second column), with the color
    indicating the relative error $e_n$ in logarithmic scale. The
    unknown shapes in the first row are exact copies of the standard
    dictionary, and in the second row are those of
    Figure~\ref{fig:all_ptb_letters_A_Z}. In (a) all letters are
    correctly identified, while in (b) letters 'E' is identified as
    'H'. For the noisy case, the experiment has been repeated 100
    times, using independent draws of white noise, and the results in
    (c) and (d) are the mean values of all experiments, where only the
    first order CGPT is taken into account. 22 and 21 letters
    are correctly identified in (c) and (d), respectively.}
  \label{fig:matching_all_letters}
\end{figure}

\paragraph{Stability.}
In real world applications we would like to have
Algorithm~\ref{algo:shape-ident-cgpt} work also on letters which
are not exact copies of the dictionary, such as handwriting
letters. Figure~\ref{fig:all_ptb_letters_A_Z} shows the letters
obtained by perturbing and smoothing the dictionary elements. With
these letters as unknown shapes, we repeat the experiment of
Figure~\ref{fig:matching_all_letters} (a) and (c) by applying
Algorithm \ref{algo:shape-ident-cgpt} on the standard dictionary
and show the results in Figure~\ref{fig:matching_all_letters} (b)
and (d). Comparing with the results of
Figure~\ref{fig:matching_all_letters} (a) and (c), we see that
Algorithm~\ref{algo:shape-ident-cgpt} remains quite stable,
despite of some slight degradations.

\paragraph{Performance of Algorithm 2.}
In the case of noiseless data,
Algorithm~\ref{algo:shape-ident-inv} provides correct results with
low computational cost. Here we repeat the experiment in
Figure~\ref{fig:matching_letter_p} (a) and (c) using
Algorithm~\ref{algo:shape-ident-inv}, and plot the error $e_n$
defined in Algorithm~\ref{algo:shape-ident-inv} in
Figure~\ref{fig:matching_all_letters_inv}. Nonetheless, when data
are noisy, Algorithm~\ref{algo:shape-ident-cgpt} performs
significantly better than Algorithm~\ref{algo:shape-ident-inv}, as
shown by Figure~\ref{fig:algorithm_cgpt_vs_inv} where we compare
the two algorithms for identifying letter ``P'' at various noise
levels. Thanks to the debiasing step \eqref{eq:tsr_lst_debiasing},
Algorithm~\ref{algo:shape-ident-cgpt} is much more robust with
respect to noise than Algorithm~\ref{algo:shape-ident-inv}, in
which there is no debiasing and the invariance of the shape
descriptors $\Dcrpo$ and $\Dcrpt$ may be severely affected by
noise (see Figure \ref{fig:algorithm_cgpt_vs_inv}).

\begin{figure}[htp]
  \centering
  \subfigure[$\sigma_0=0$, order $\leq 5$, Standard letters]{\includegraphics[width=7.5cm]{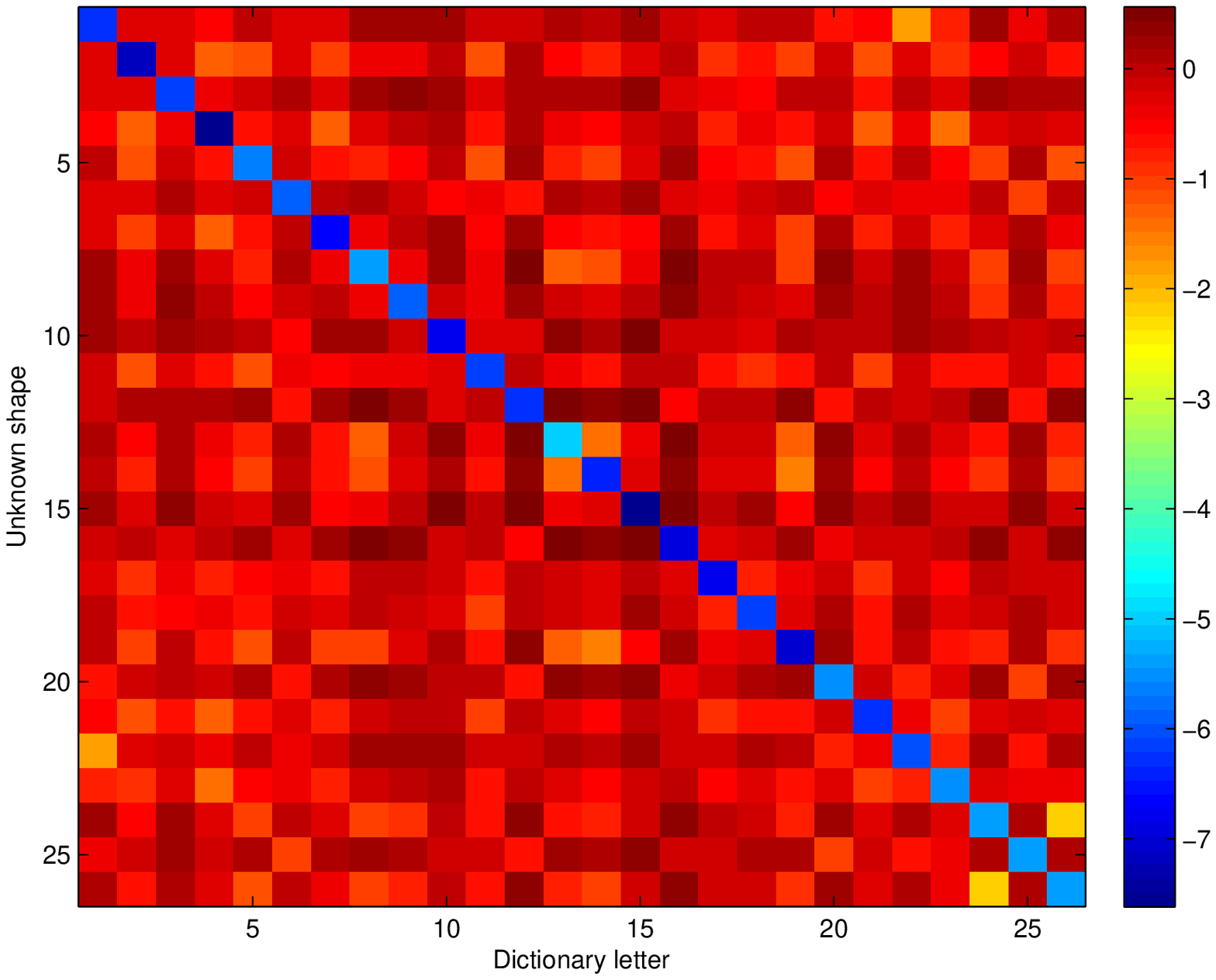}}
  \subfigure[$\sigma_0=0$, order $\leq 5$, Perturbed letters]{\includegraphics[width=7.5cm]{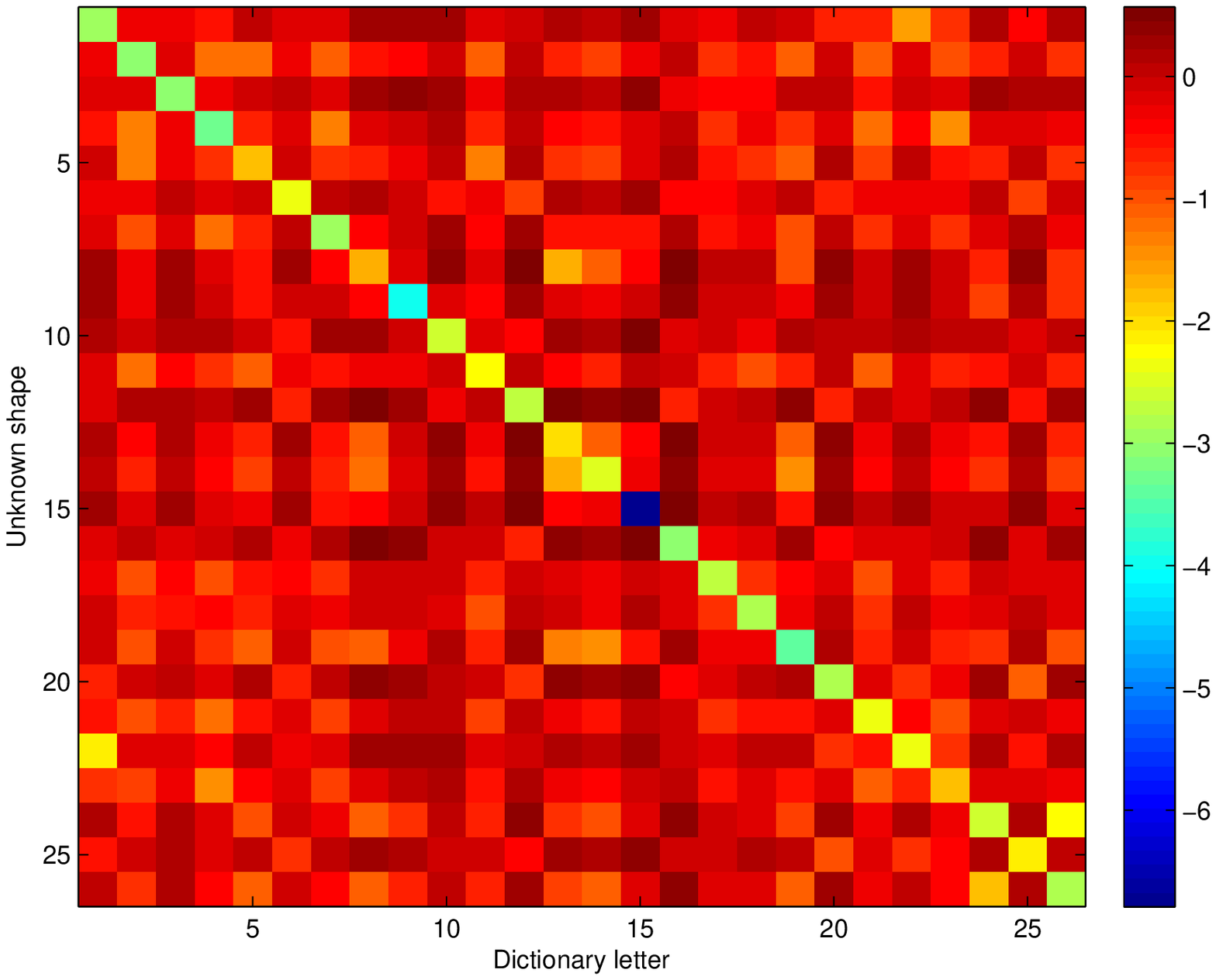}}
  \caption{Algorithm~\ref{algo:shape-ident-inv} applied on the all 26
    letters using the standard dictionary
    (Figure~\ref{fig:all_letters_A_Z}) at noise level
    $\sigma_0=0$. The unknown shapes in (a) are exact copies of the
    standard dictionary, while in (b) are those of
    Figure~\ref{fig:all_ptb_letters_A_Z}. The color indicates the
    error $e_n$ in logarithmic scale. All letters are correctly
    identified in both (a) and (b).}
  \label{fig:matching_all_letters_inv}
\end{figure}

\begin{figure}[htp]
  \centering
  \subfigure[order $\leq 2$]{\includegraphics[width=7.5cm]{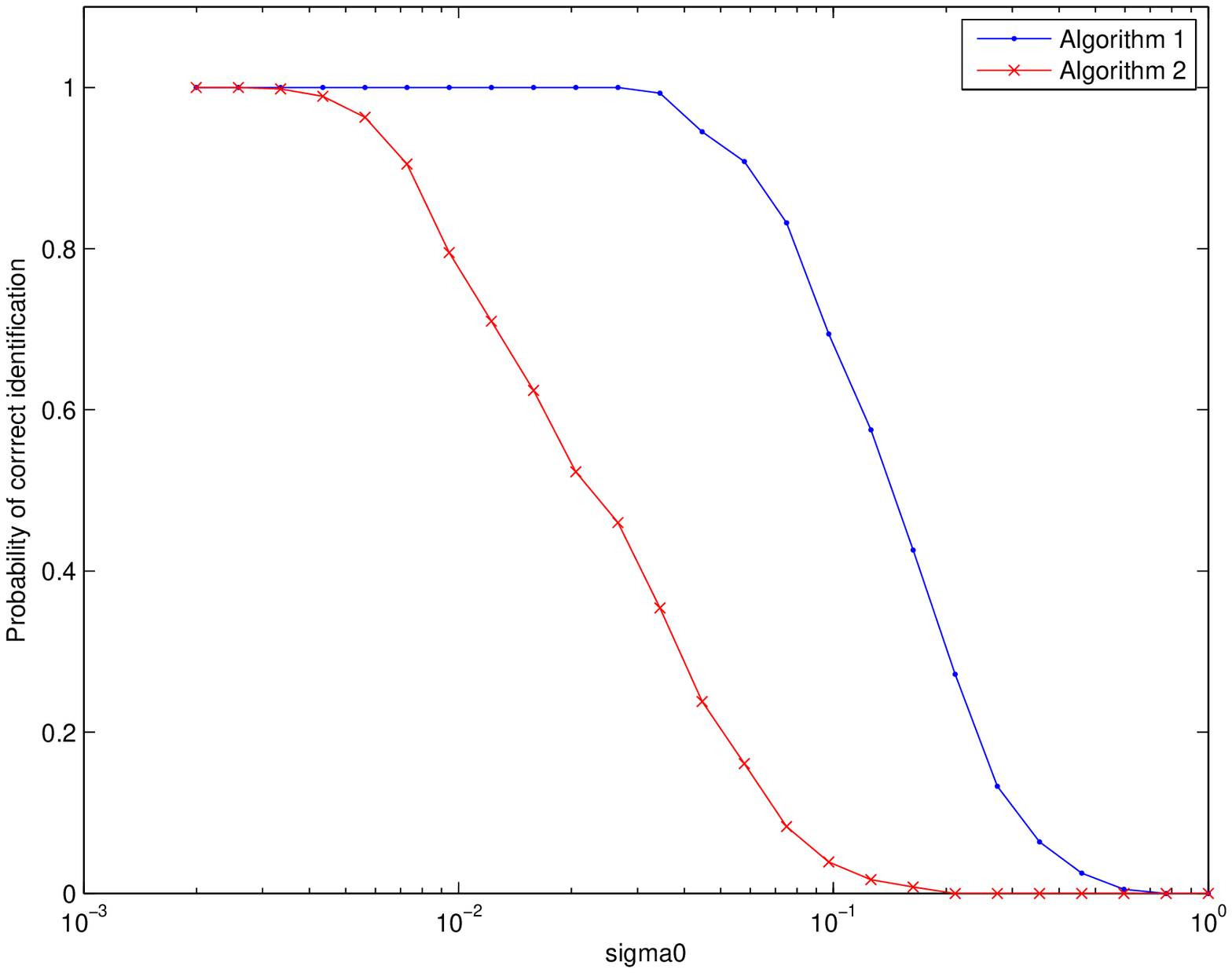}}
  \subfigure[order $\leq 3$]{\includegraphics[width=7.5cm]{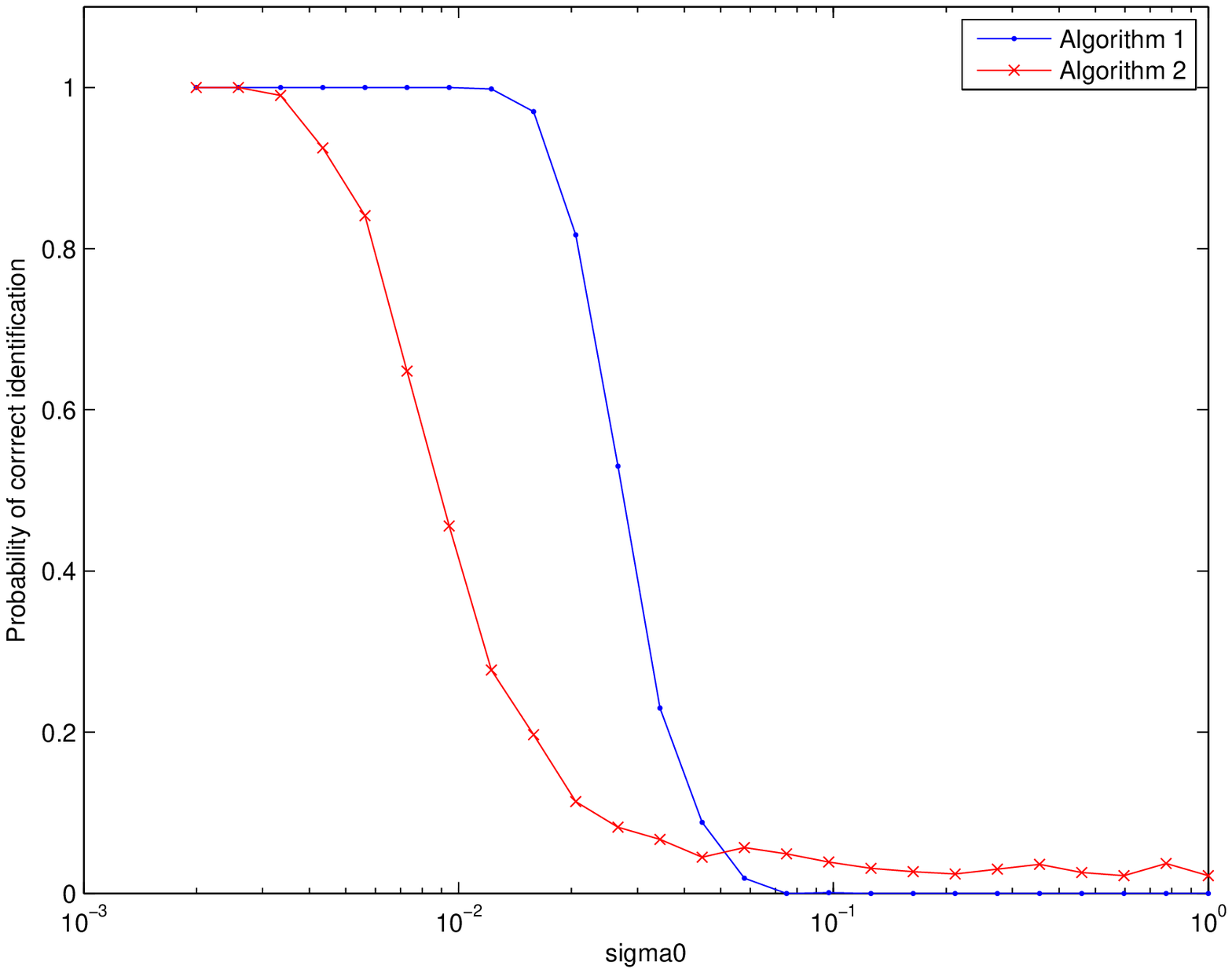}}
  \caption{Comparison of Algorithm~\ref{algo:shape-ident-inv} and
    Algorithm~\ref{algo:shape-ident-cgpt} on identification of the
    standard letter ``P''. At each noise level, the experiment has
    been repeated 1000 times, using independent draws of white
    noise. For each algorithm, the curve represents the percentage of
    experiments where the letter ``P'' is correctly identified. }
  \label{fig:algorithm_cgpt_vs_inv}
\end{figure}

\def\lettersize{2.7cm}

\begin{figure}[htp]
  \centering
  \subfigure[]{\includegraphics[width=\lettersize]{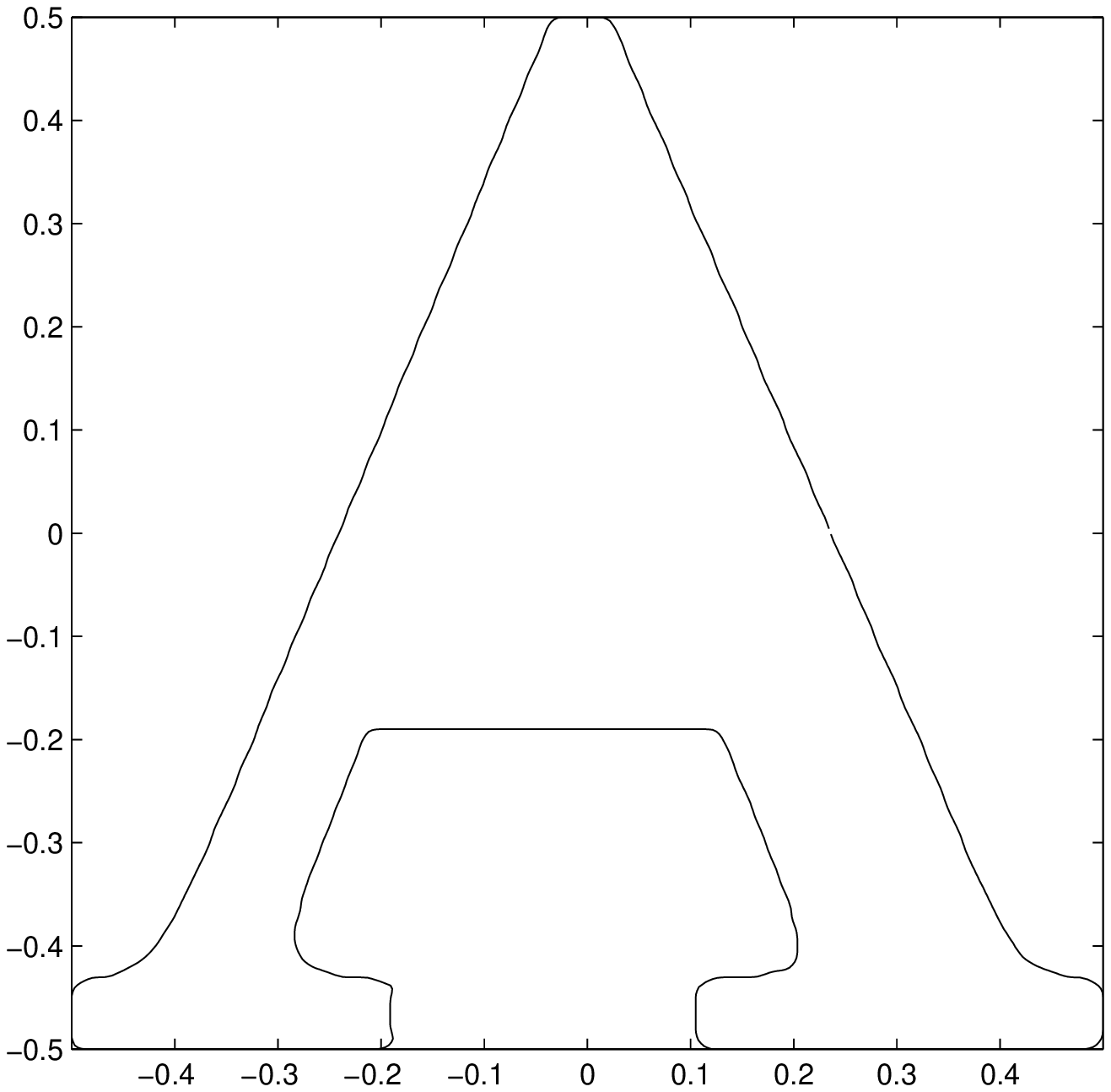}}
  \subfigure[]{\includegraphics[width=\lettersize]{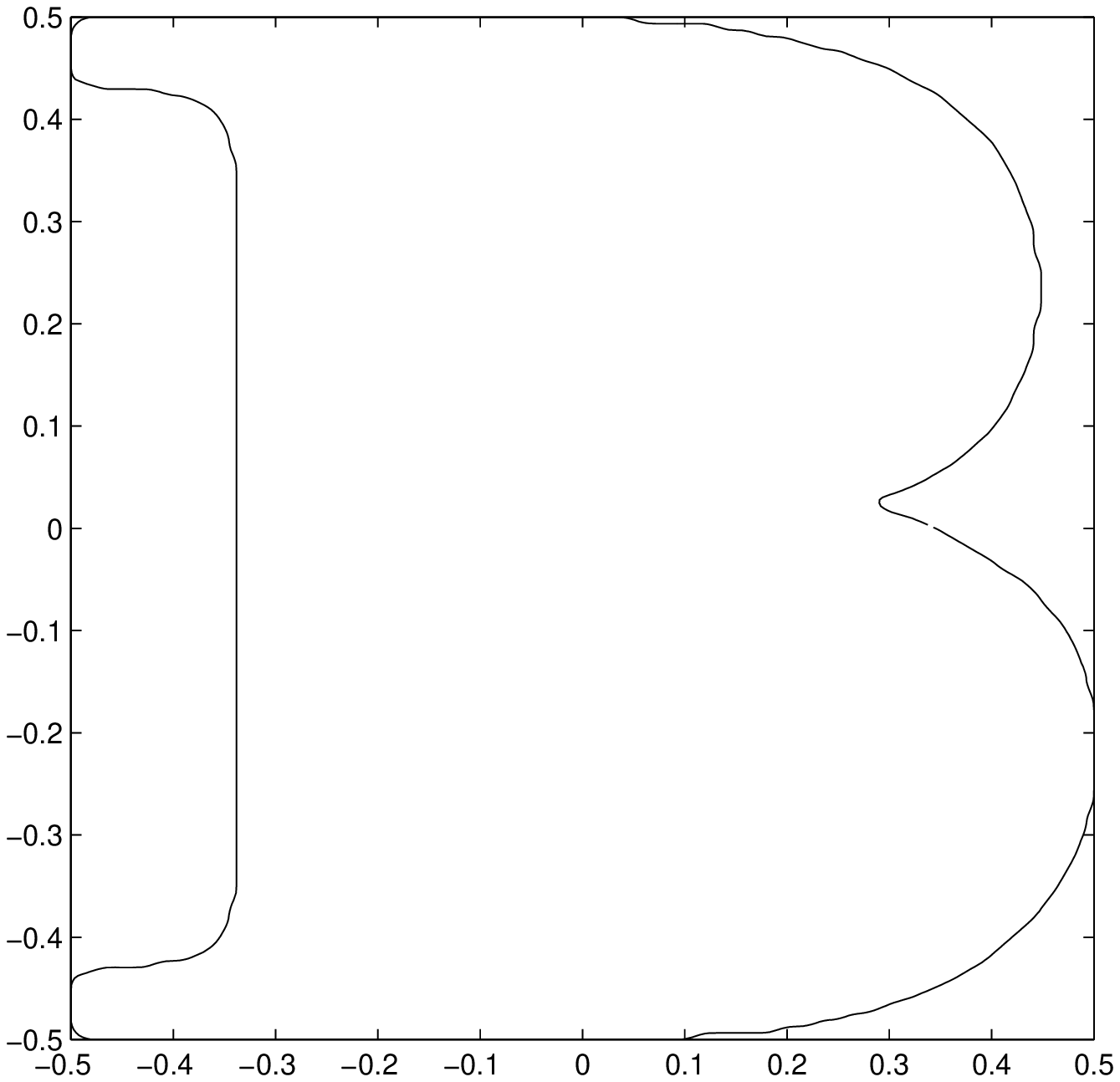}}
  \subfigure[]{\includegraphics[width=\lettersize]{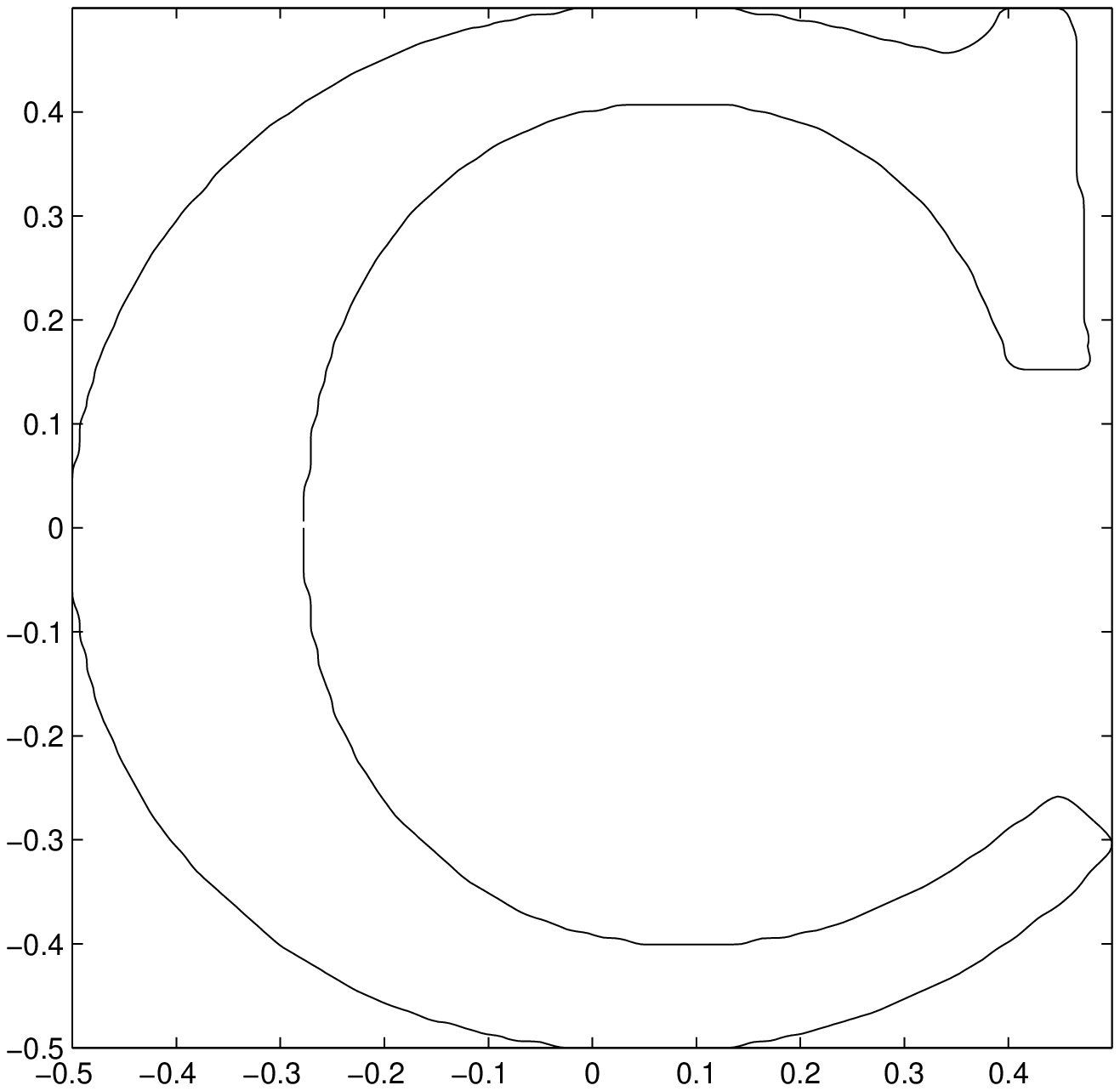}}
  \subfigure[]{\includegraphics[width=\lettersize]{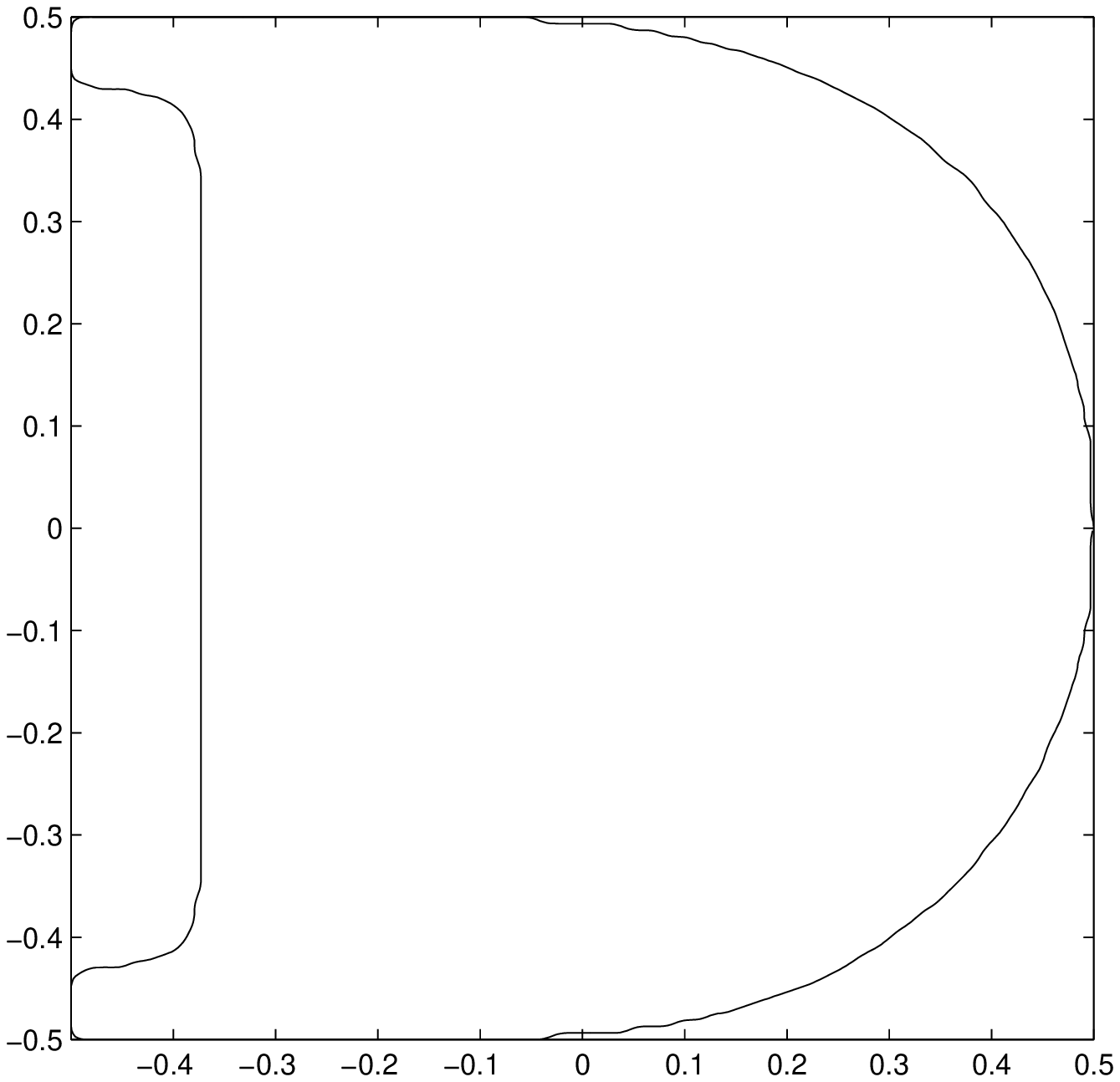}}
  \subfigure[]{\includegraphics[width=\lettersize]{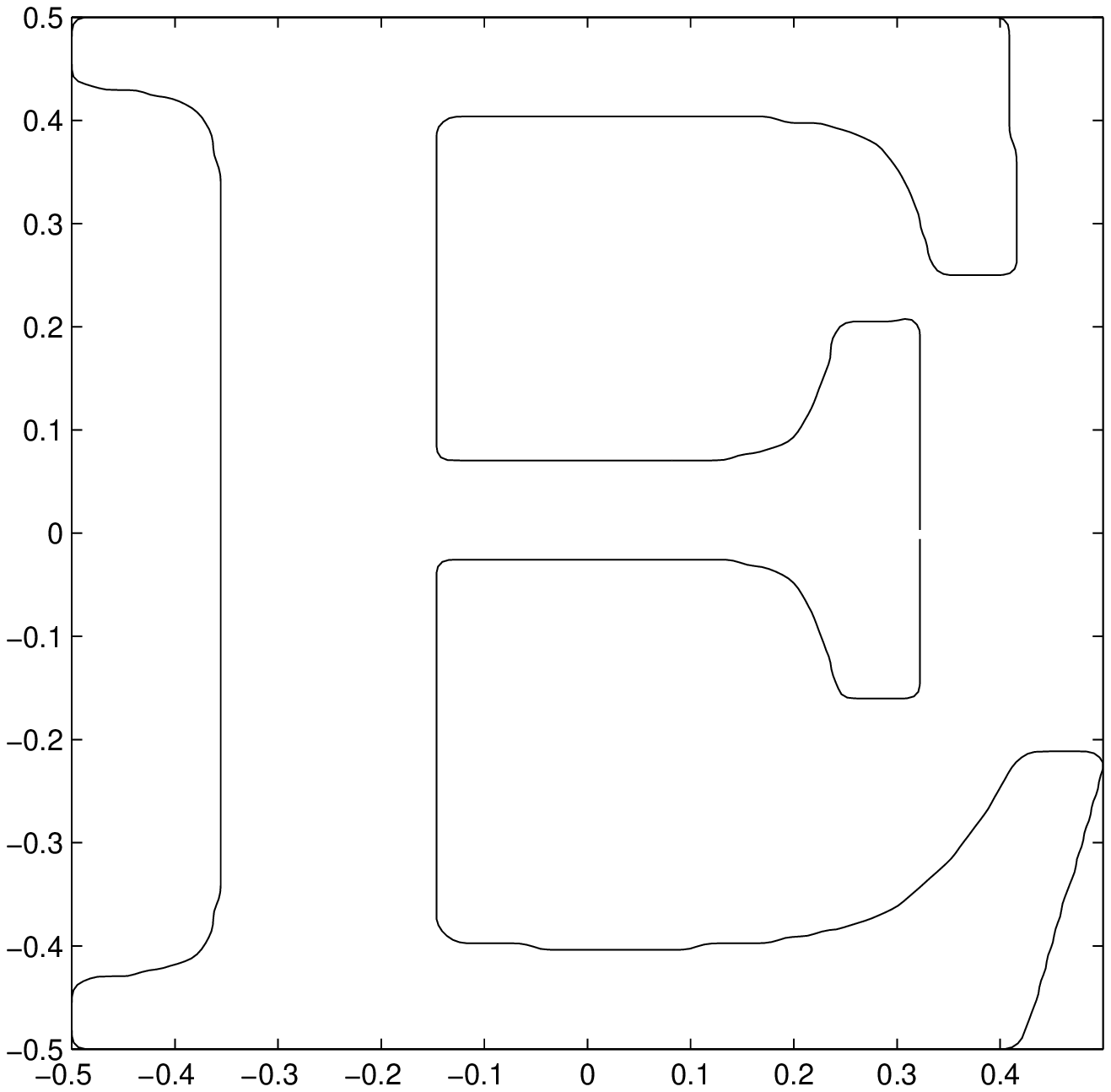}}
  \subfigure[]{\includegraphics[width=\lettersize]{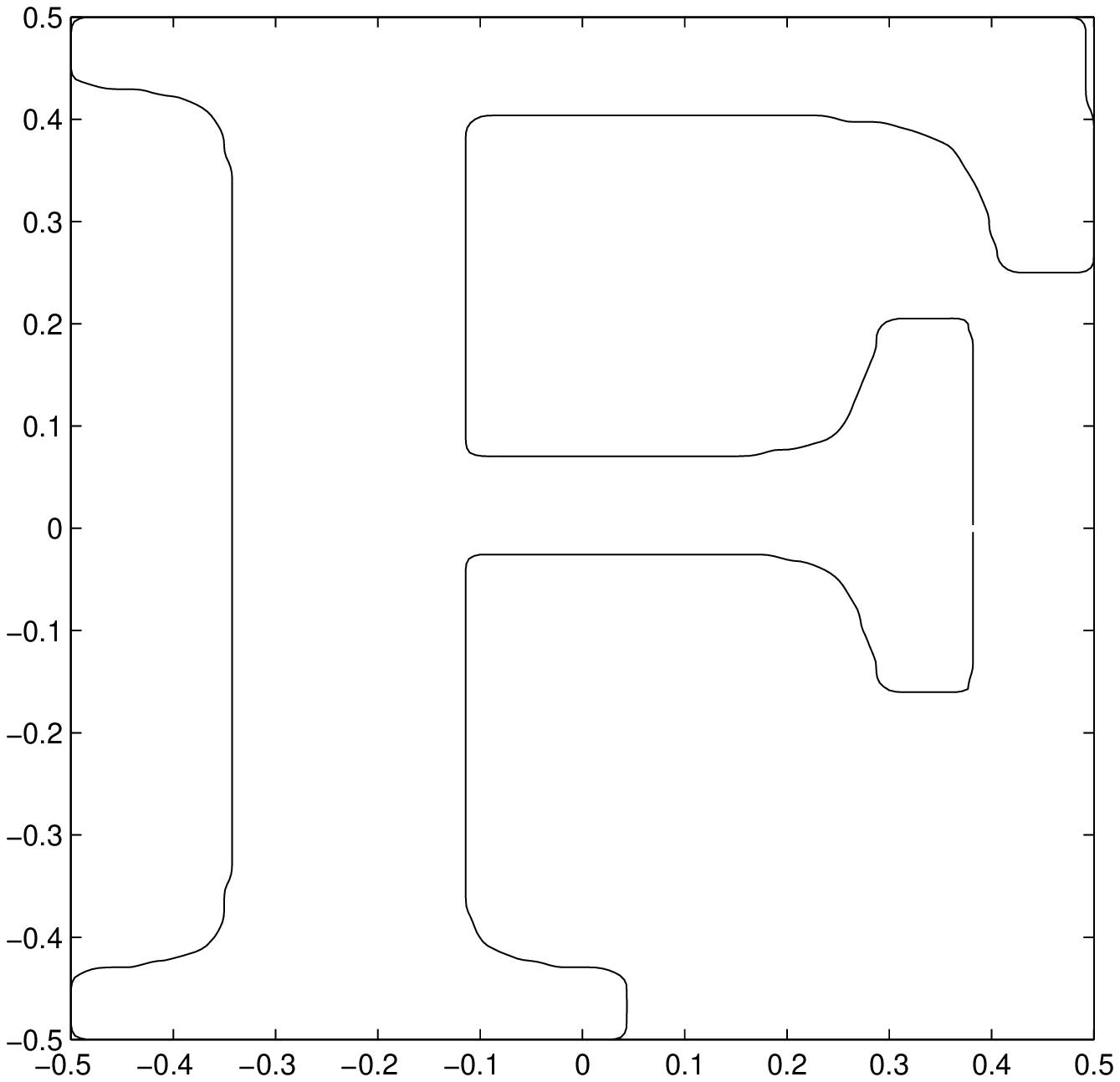}}
  \subfigure[]{\includegraphics[width=\lettersize]{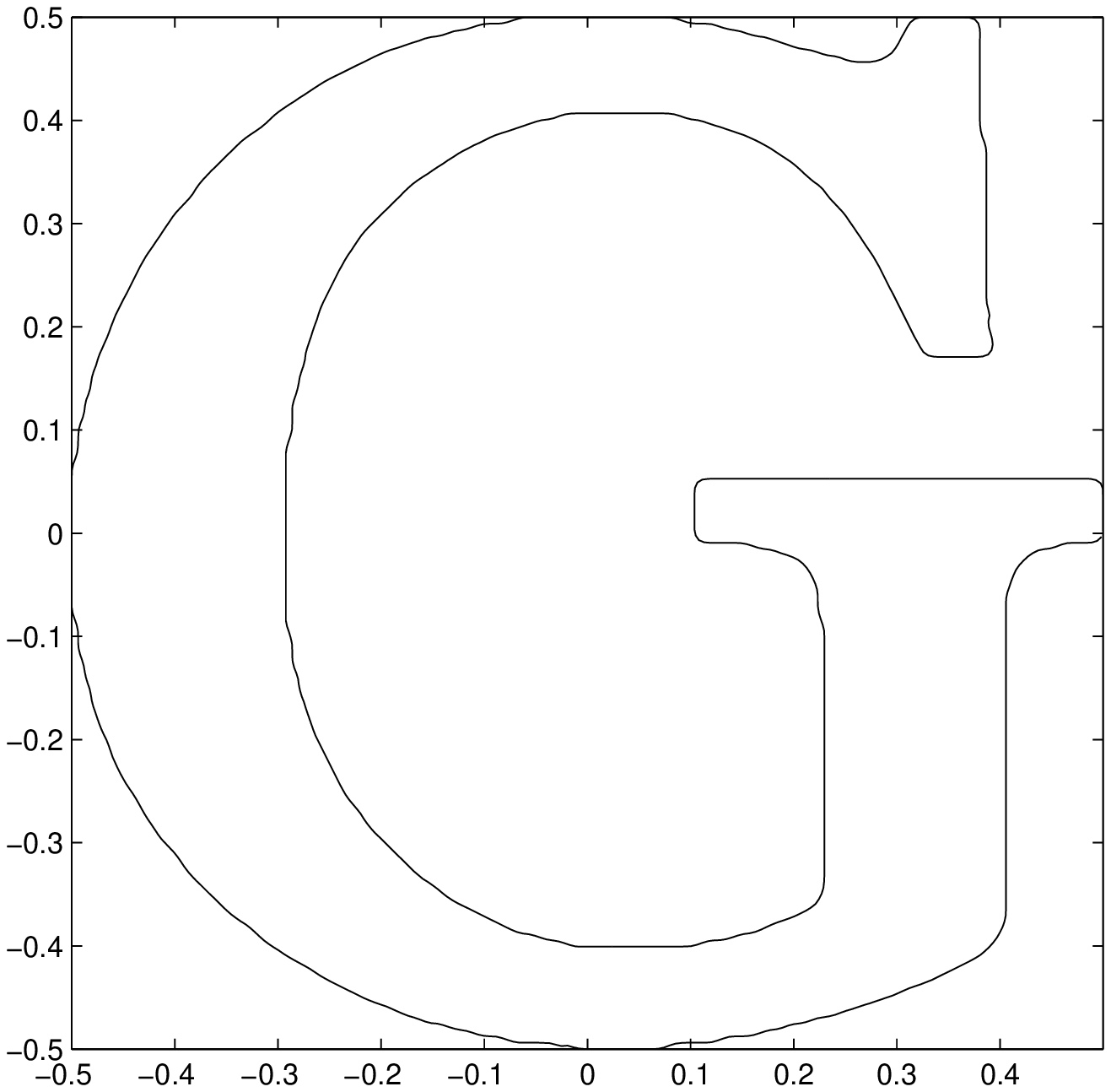}}
  \subfigure[]{\includegraphics[width=\lettersize]{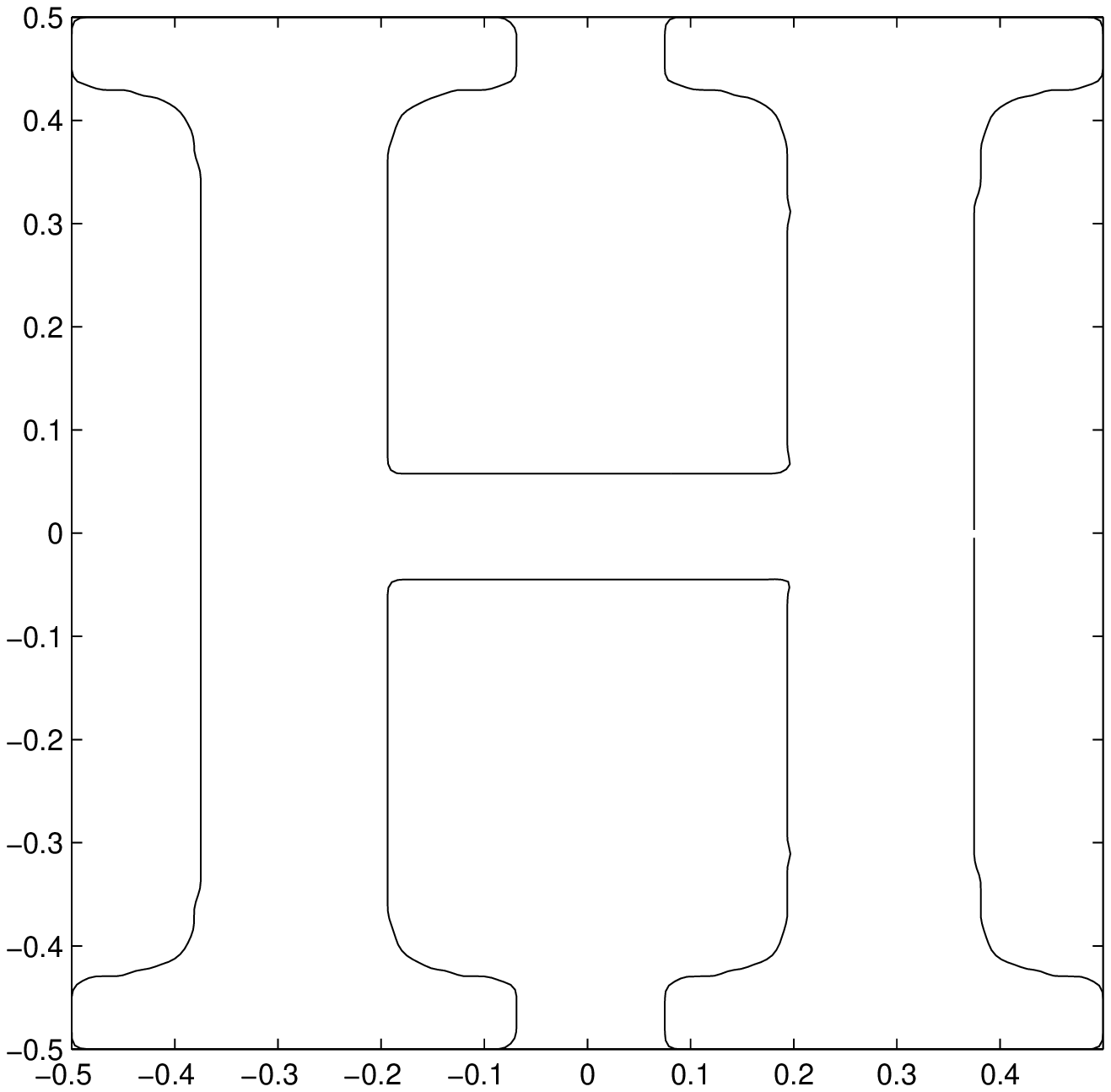}}
  \subfigure[]{\includegraphics[width=\lettersize]{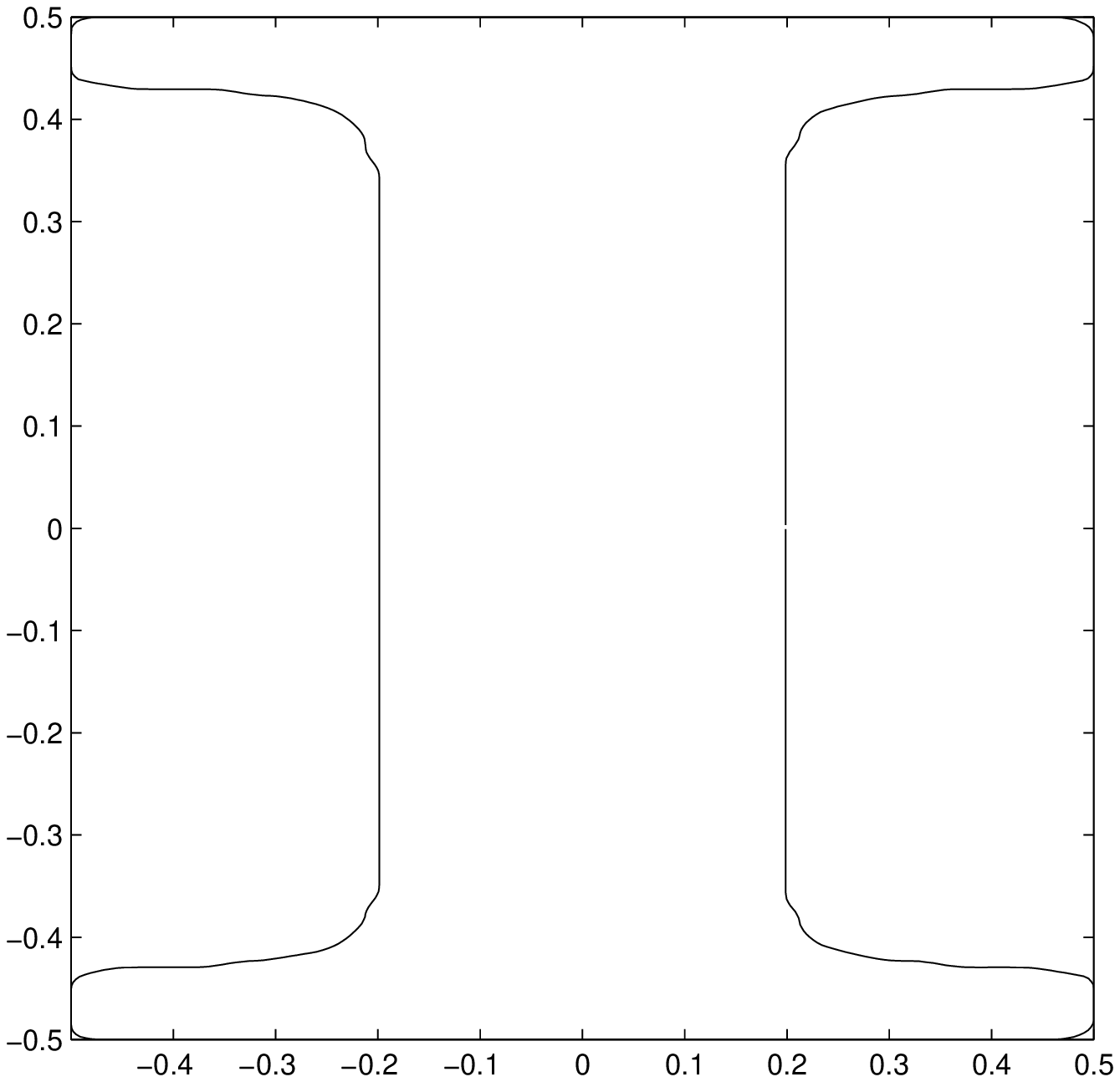}}
  \subfigure[]{\includegraphics[width=\lettersize]{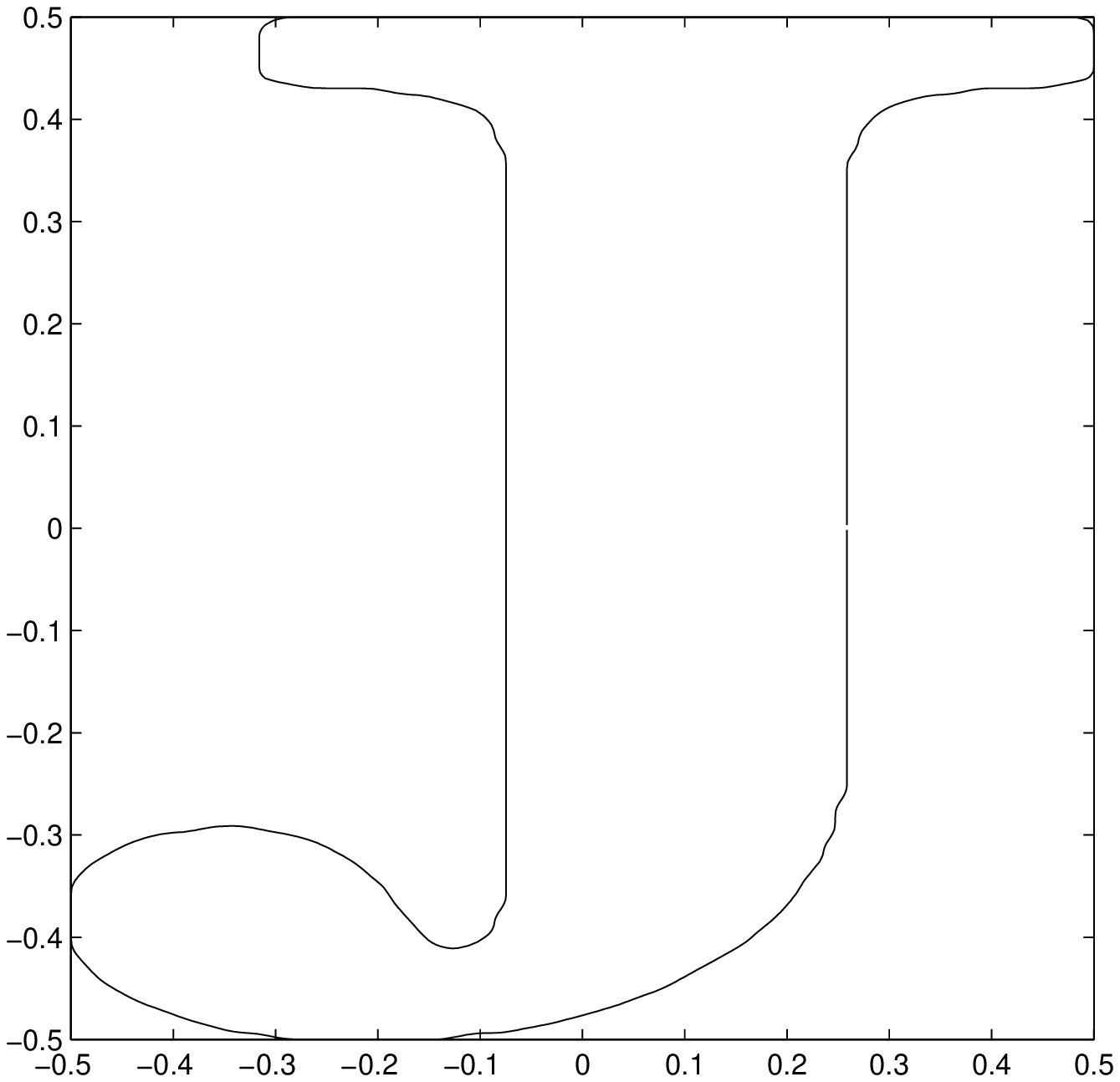}}
  \subfigure[]{\includegraphics[width=\lettersize]{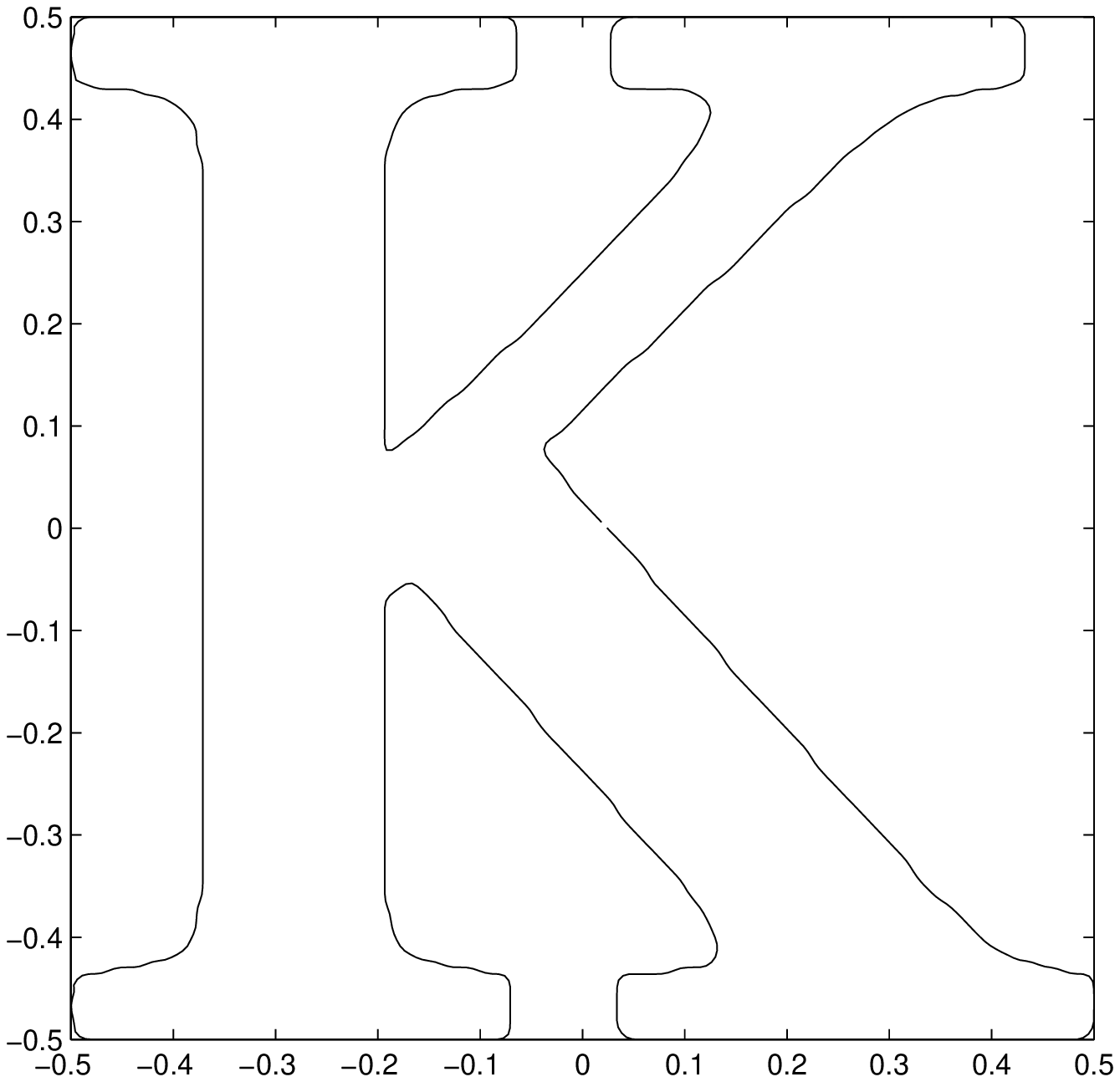}}
  \subfigure[]{\includegraphics[width=\lettersize]{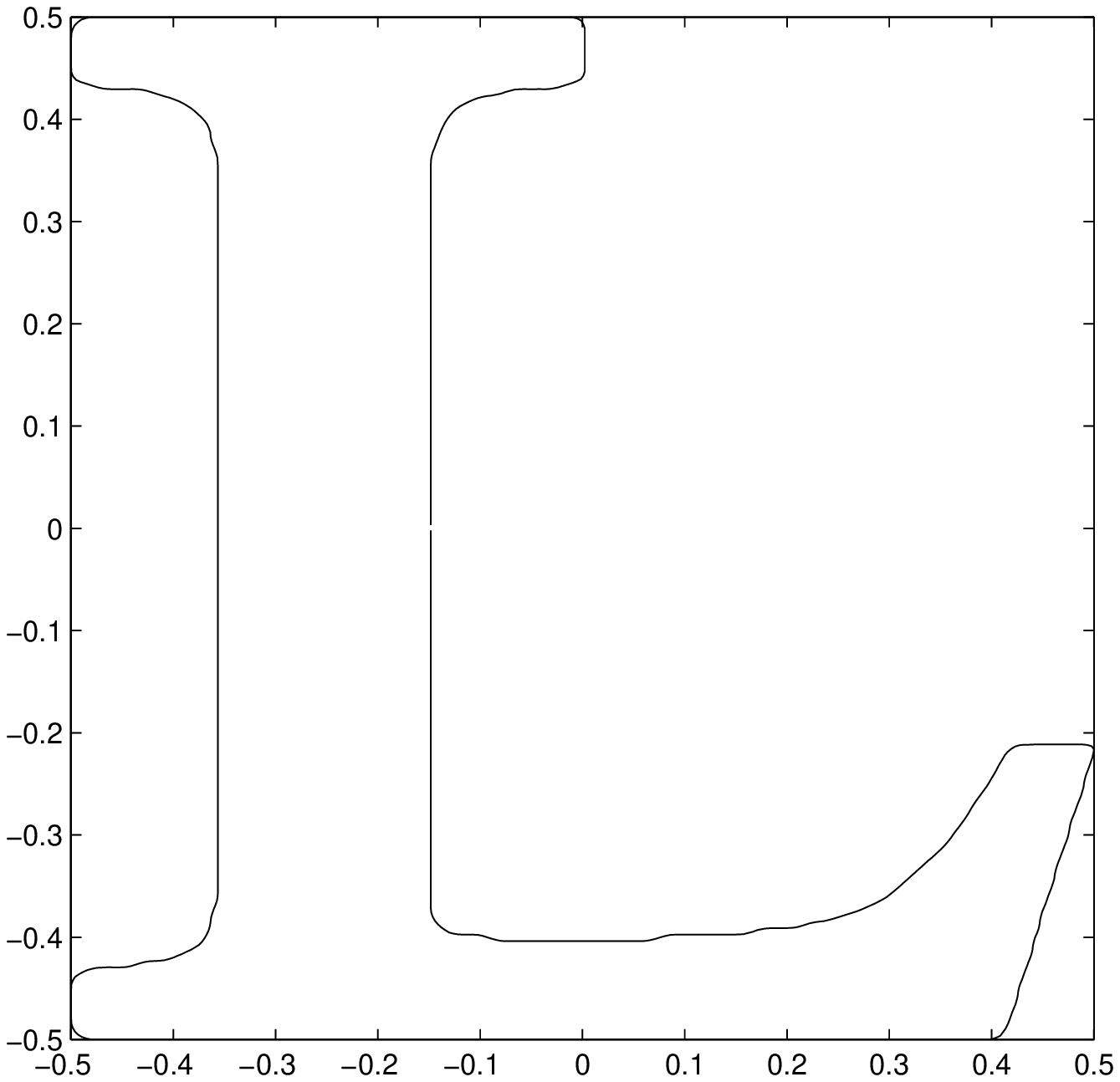}}
  \subfigure[]{\includegraphics[width=\lettersize]{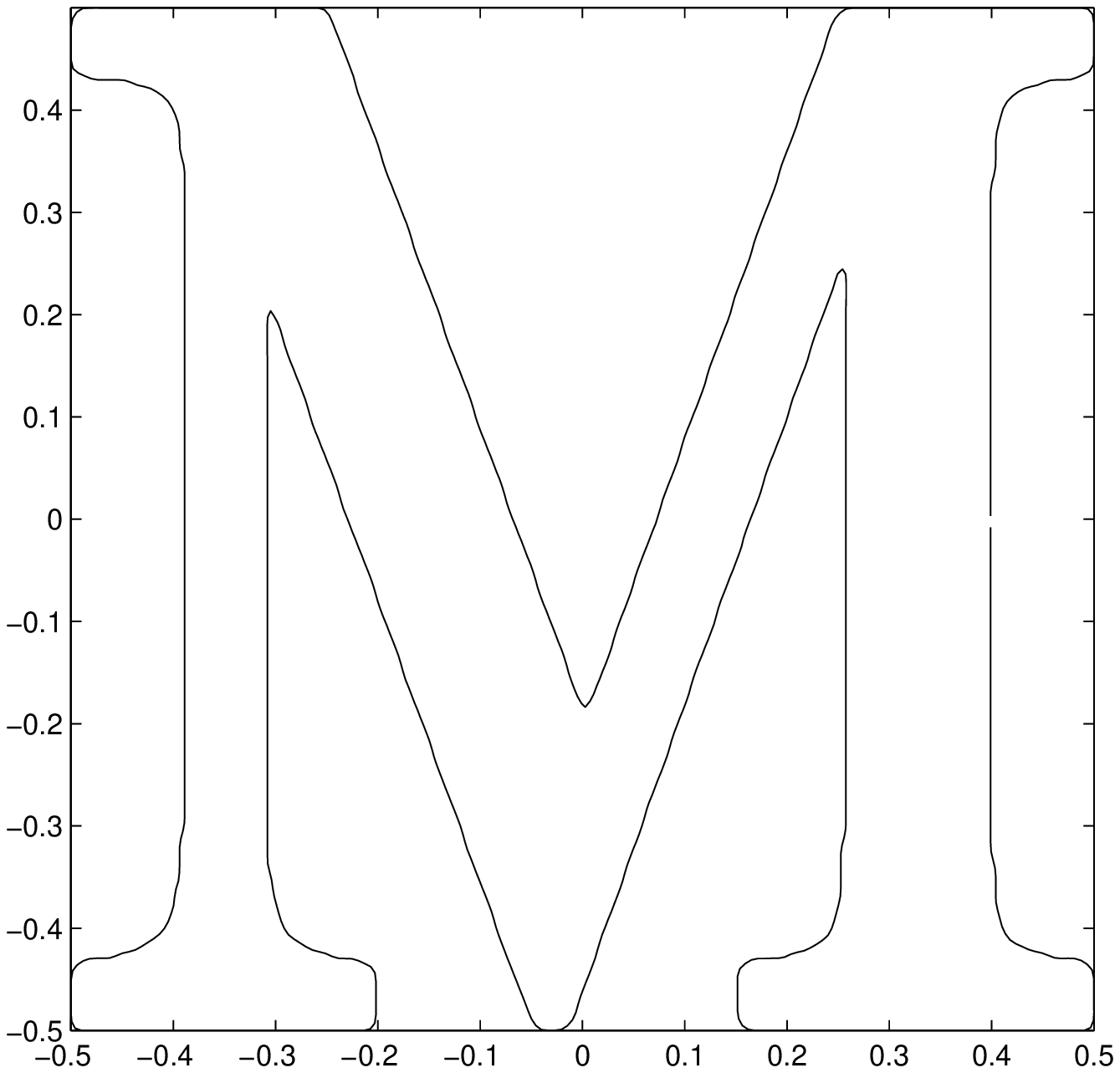}}
  \subfigure[]{\includegraphics[width=\lettersize]{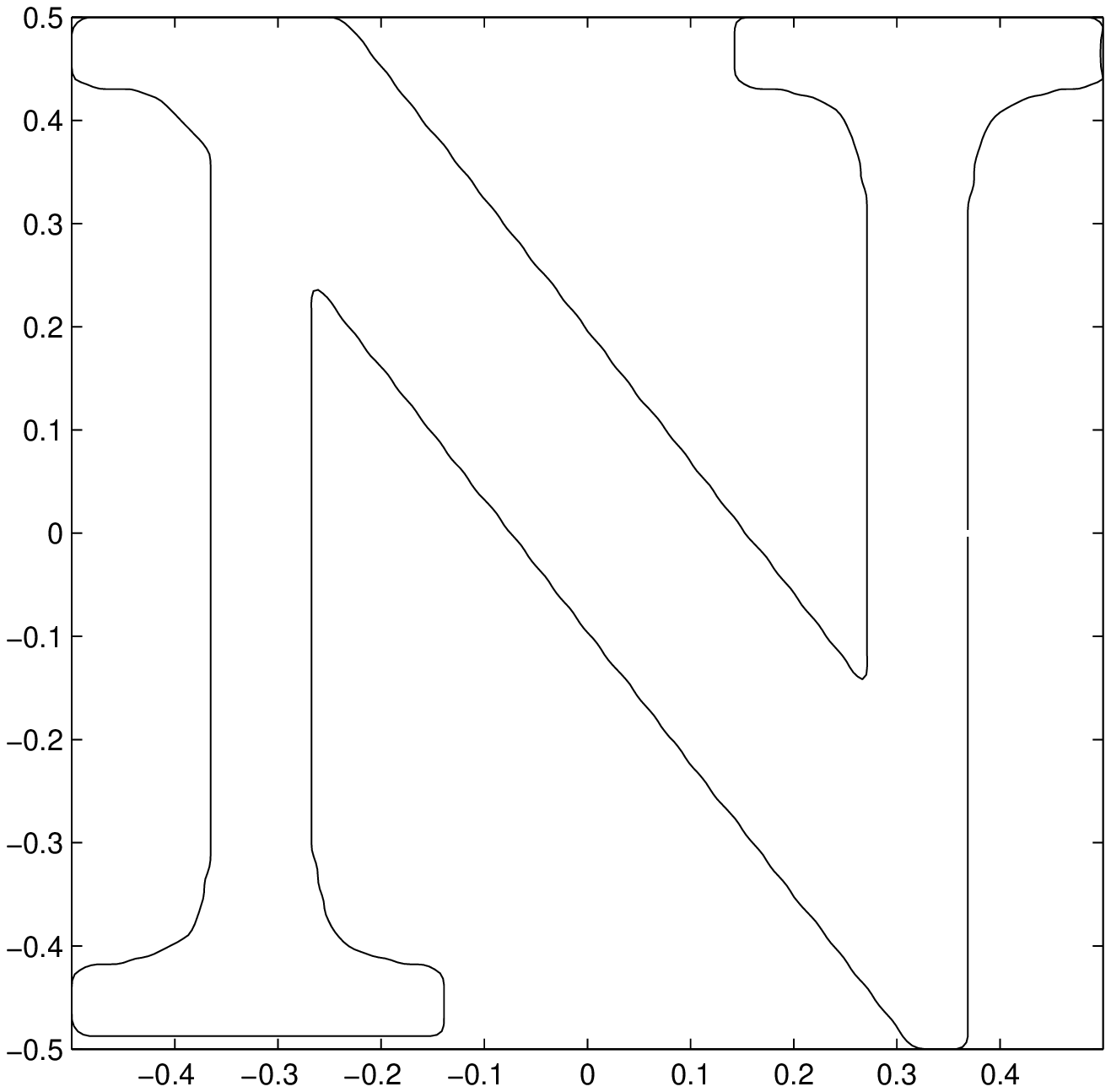}}
  \subfigure[]{\includegraphics[width=\lettersize]{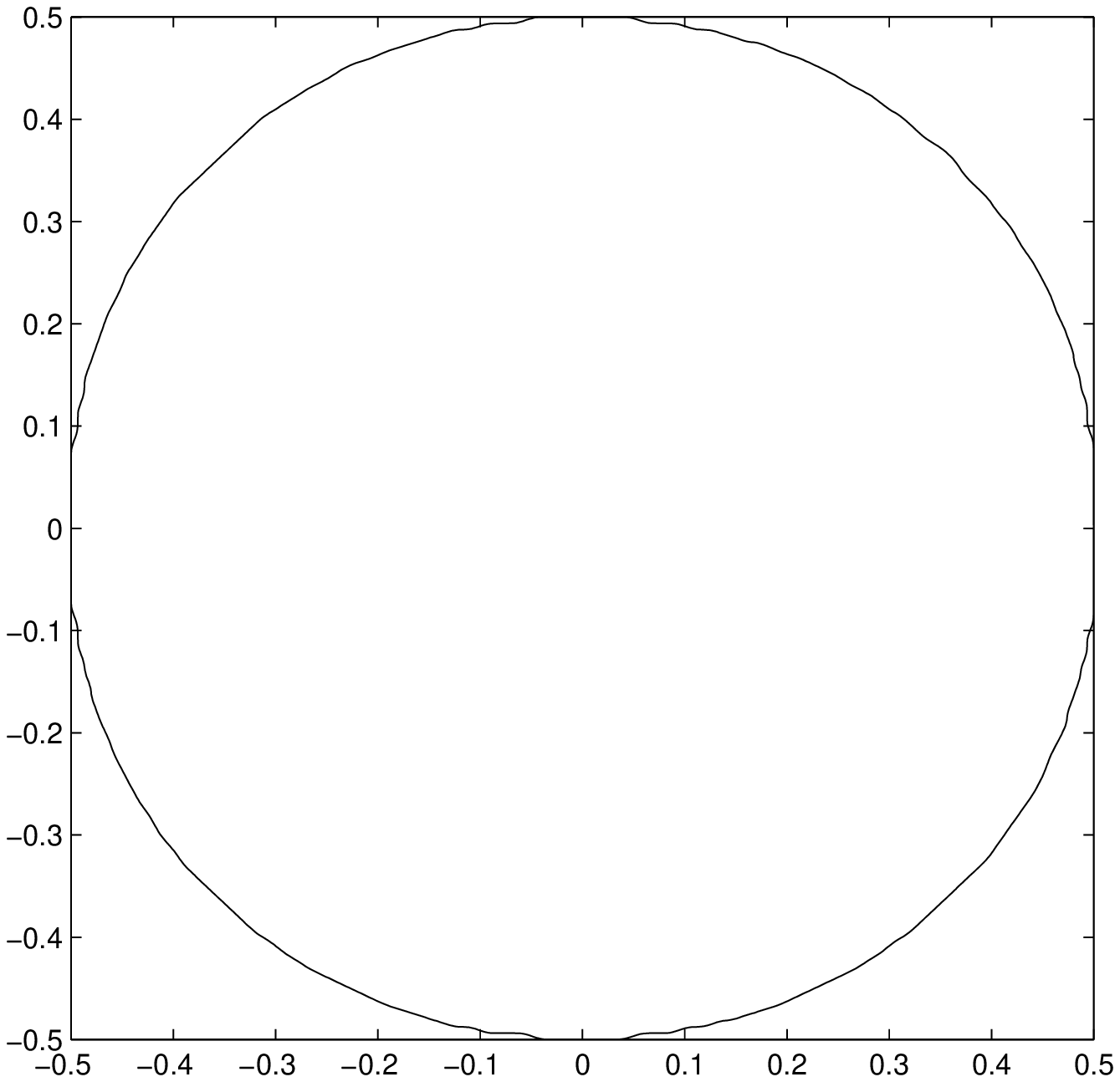}}
  \subfigure[]{\includegraphics[width=\lettersize]{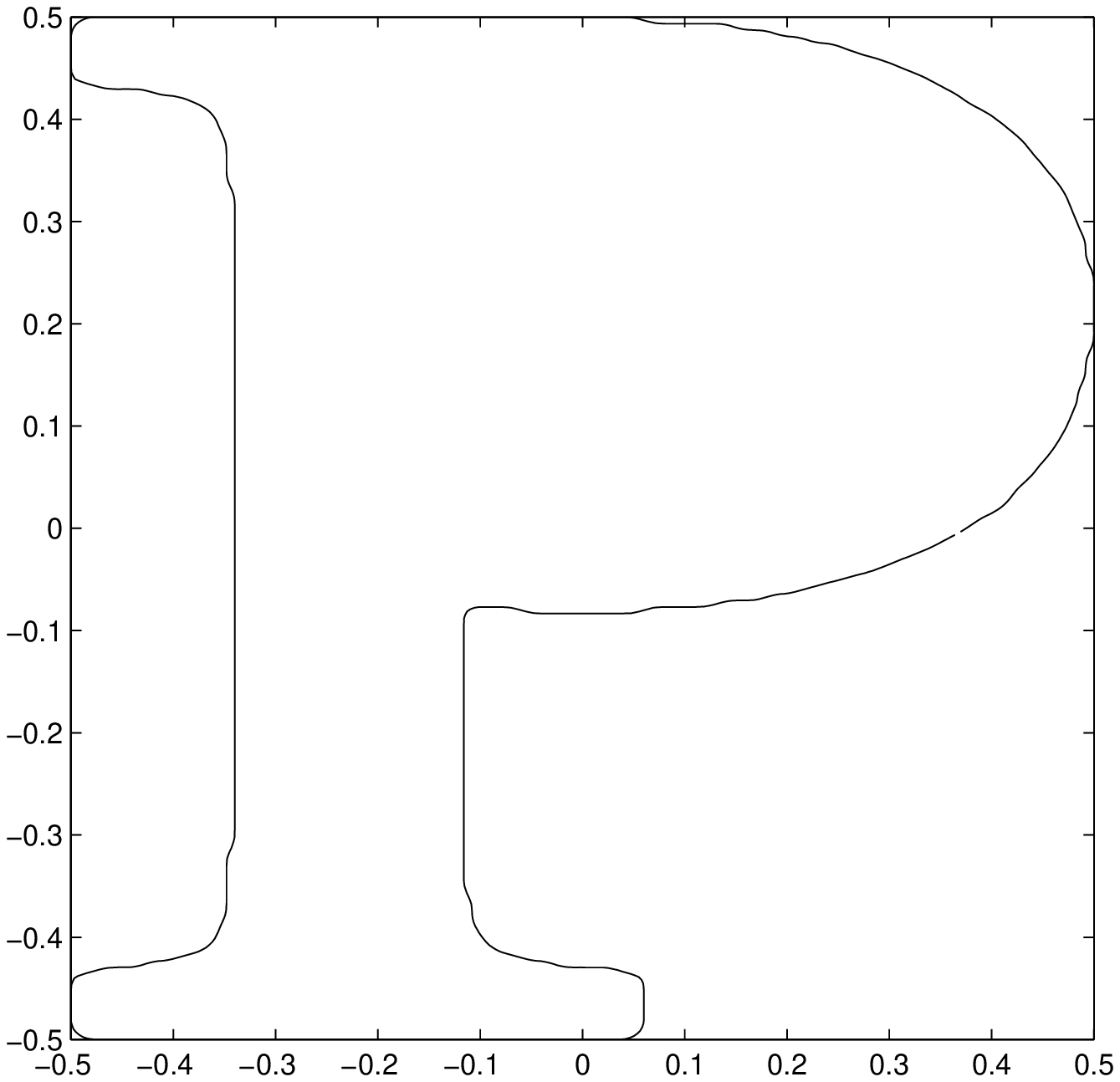}}
  \subfigure[]{\includegraphics[width=\lettersize]{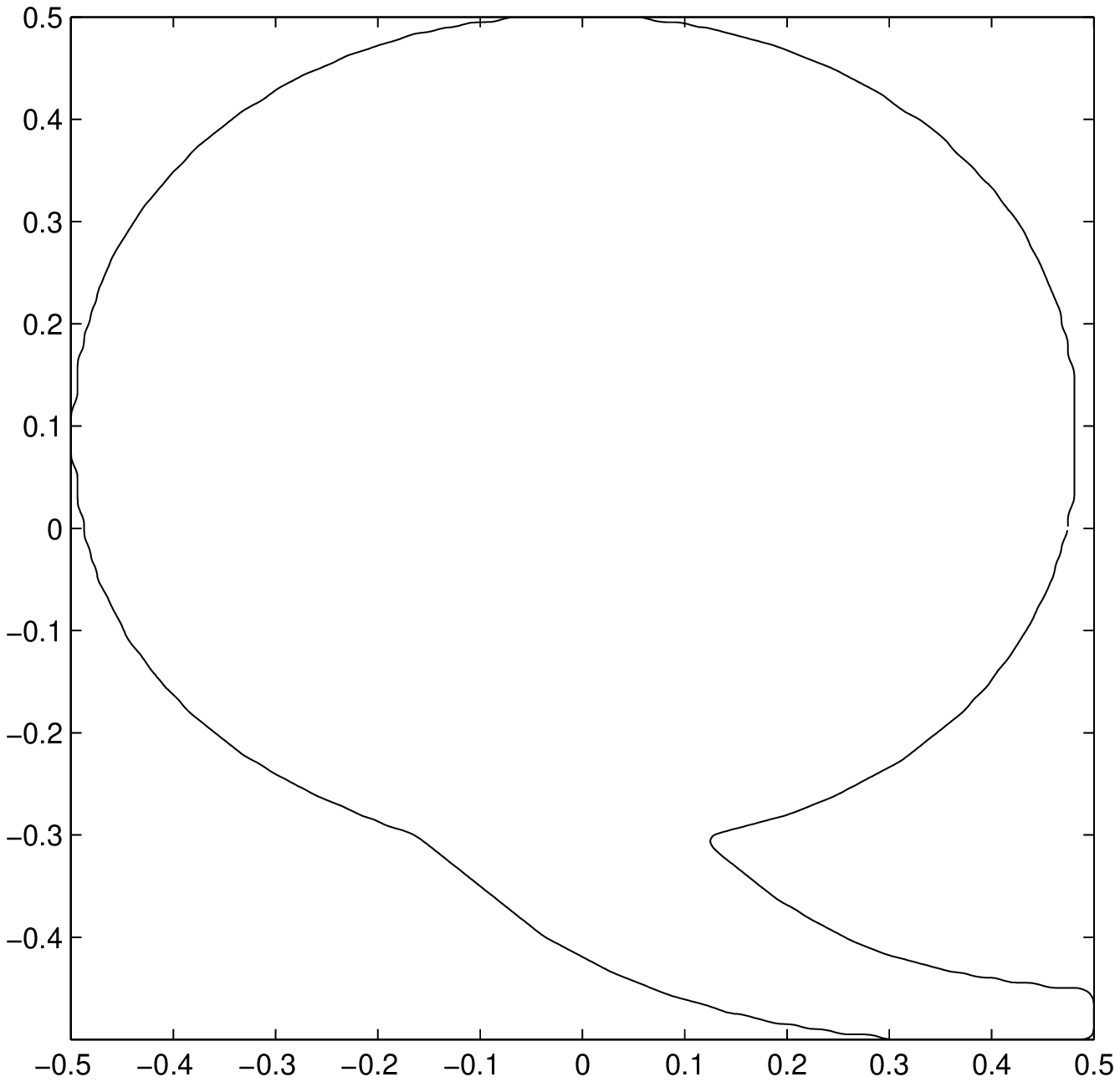}}
  \subfigure[]{\includegraphics[width=\lettersize]{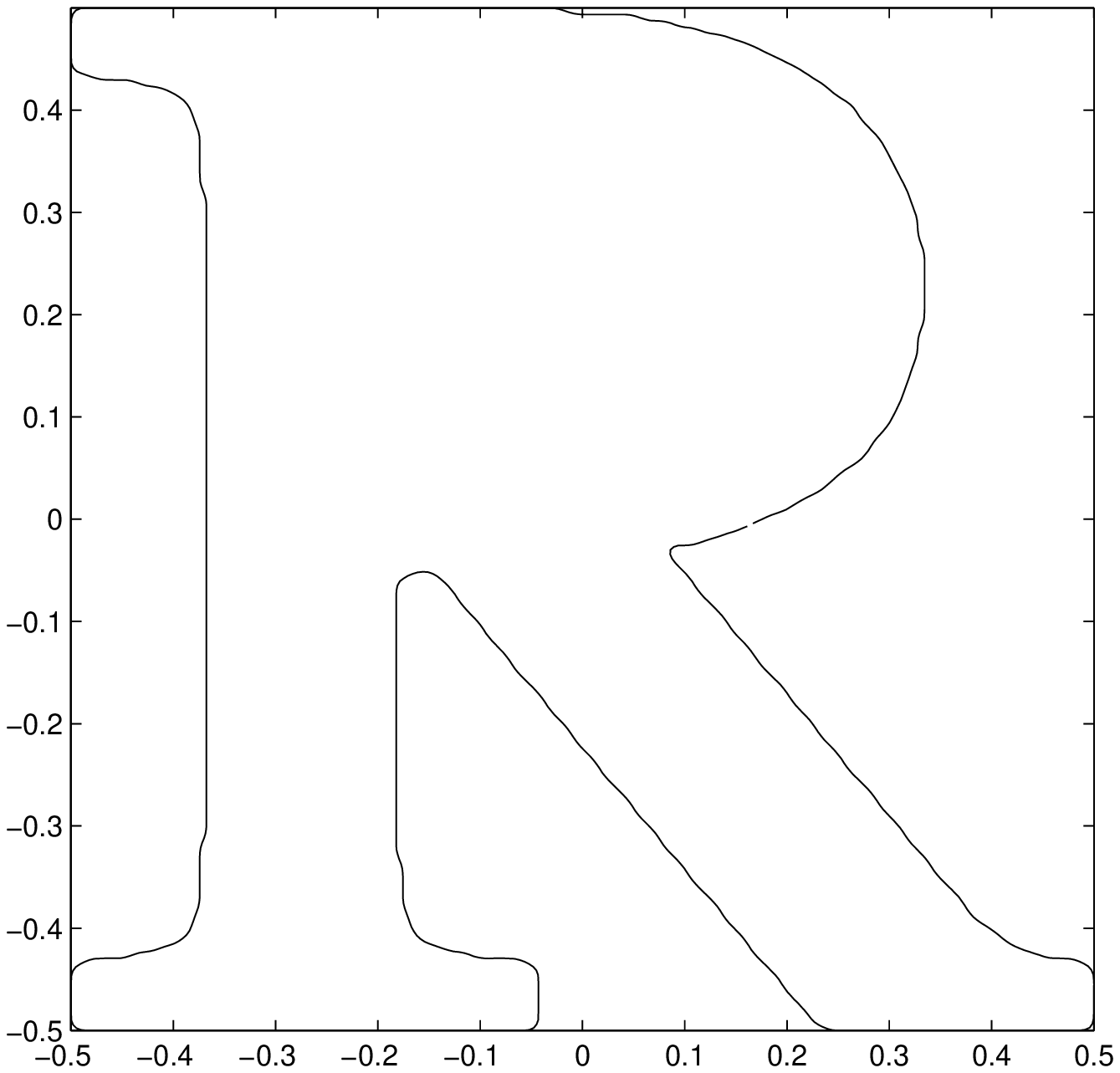}}
  \subfigure[]{\includegraphics[width=\lettersize]{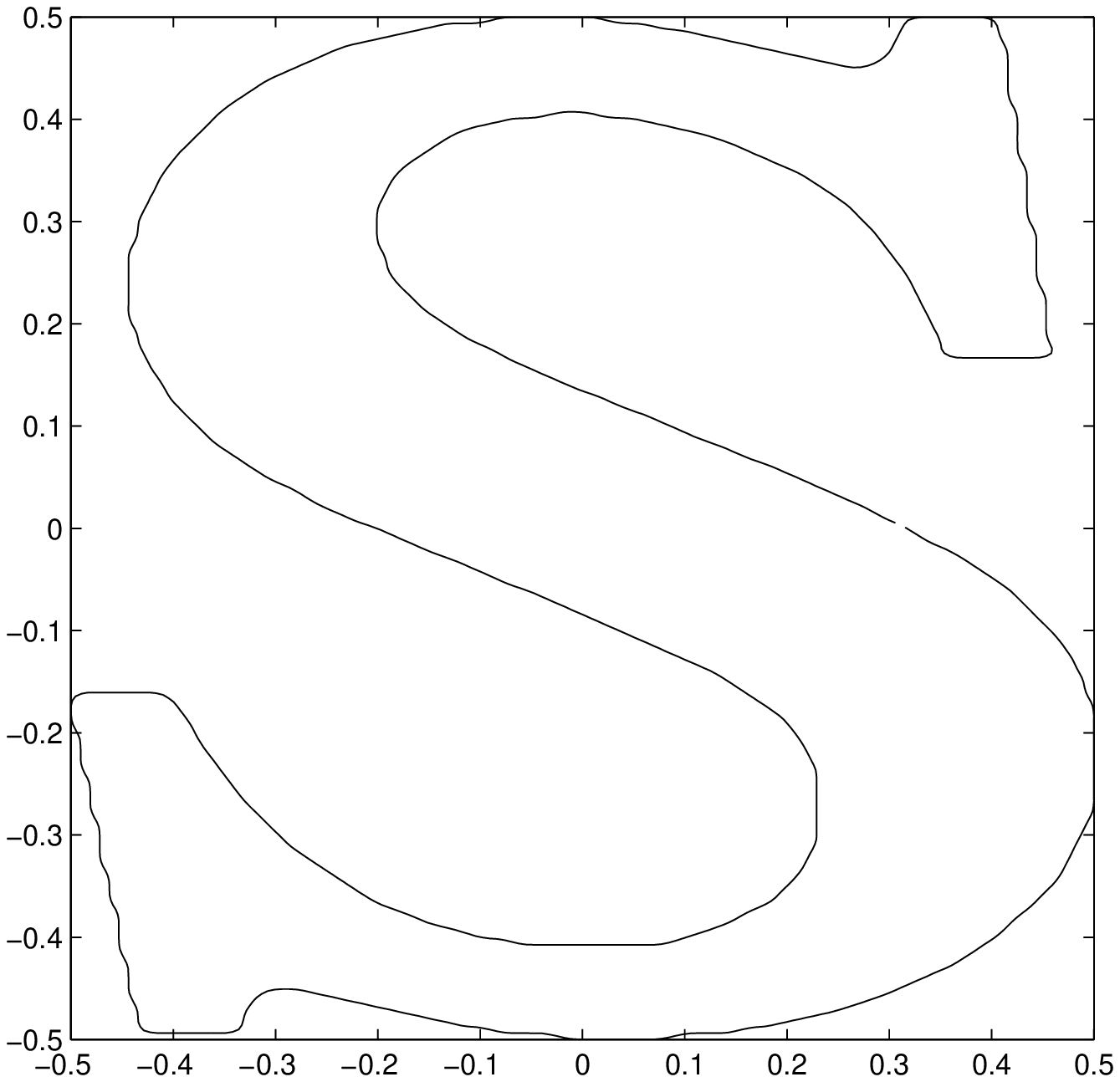}}
  \subfigure[]{\includegraphics[width=\lettersize]{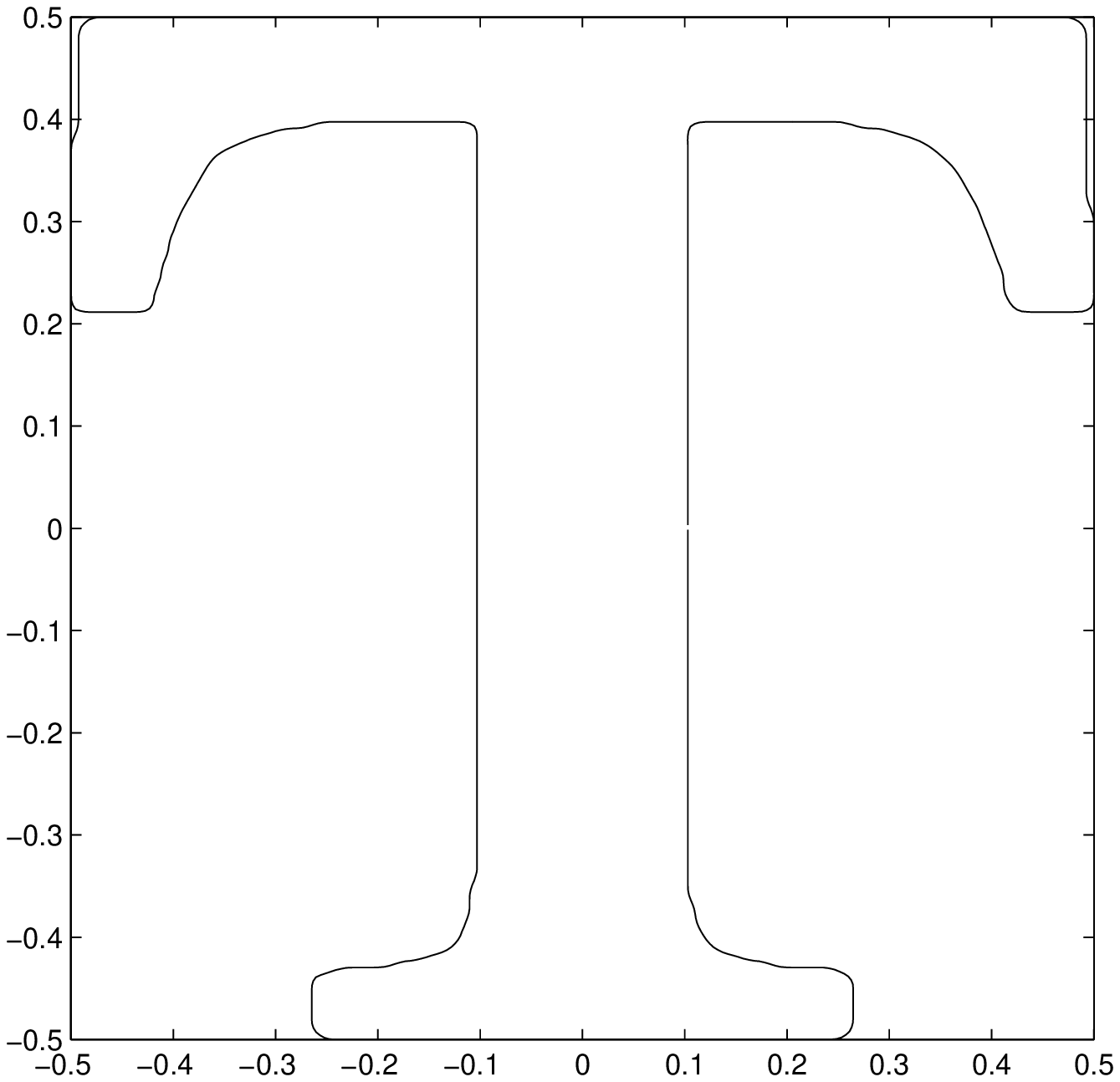}}
  \subfigure[]{\includegraphics[width=\lettersize]{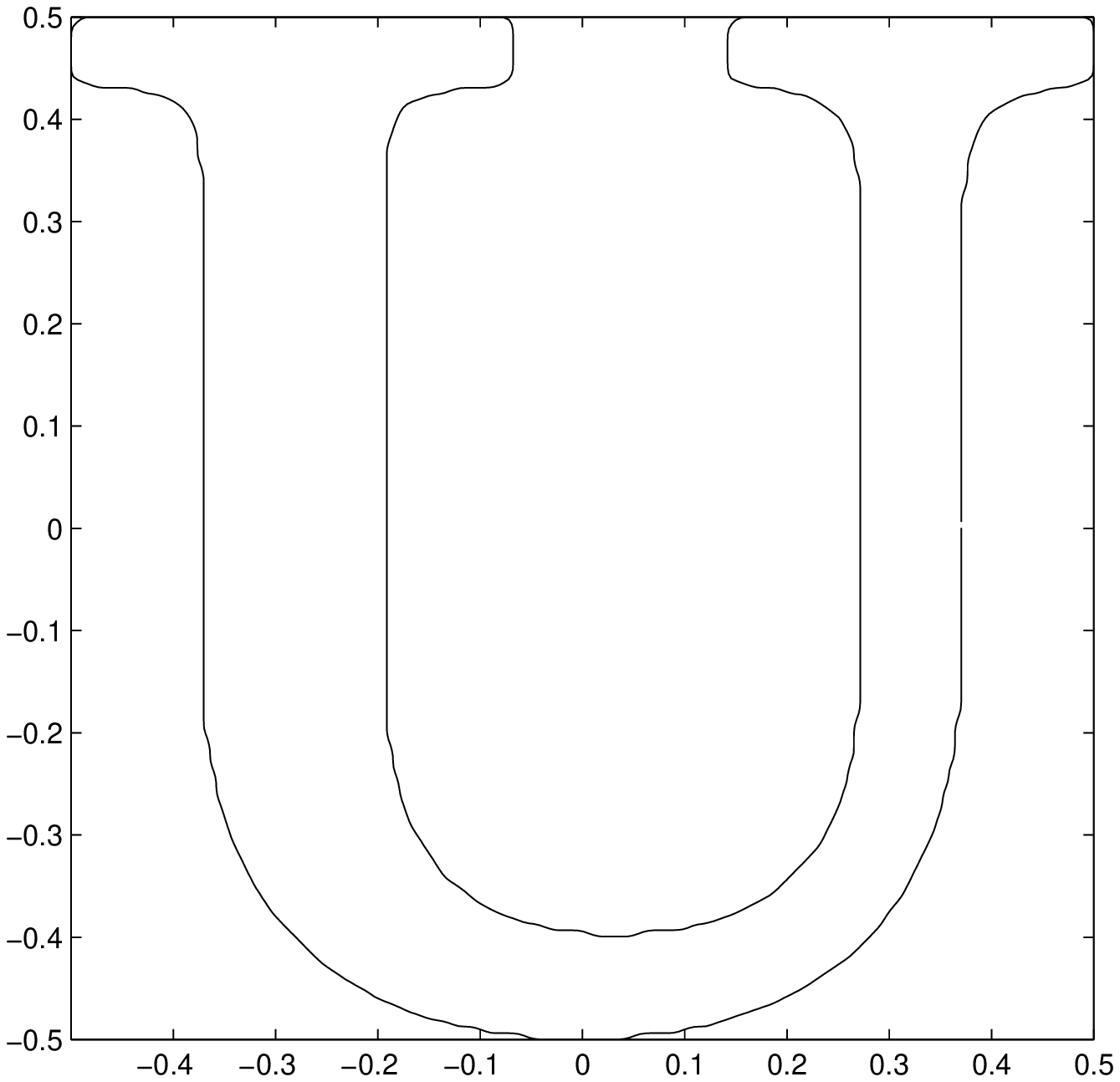}}
  \subfigure[]{\includegraphics[width=\lettersize]{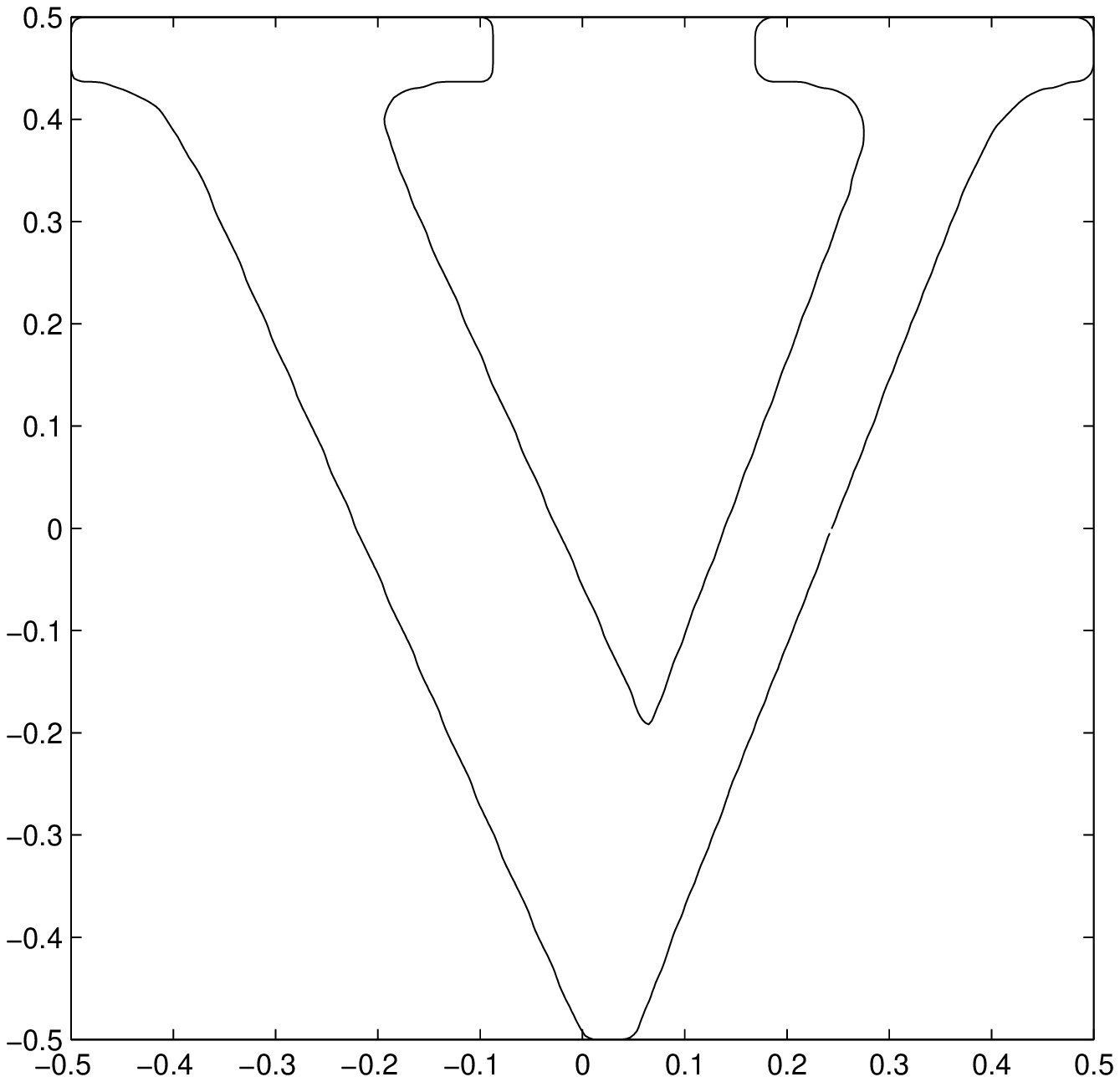}}
  \subfigure[]{\includegraphics[width=\lettersize]{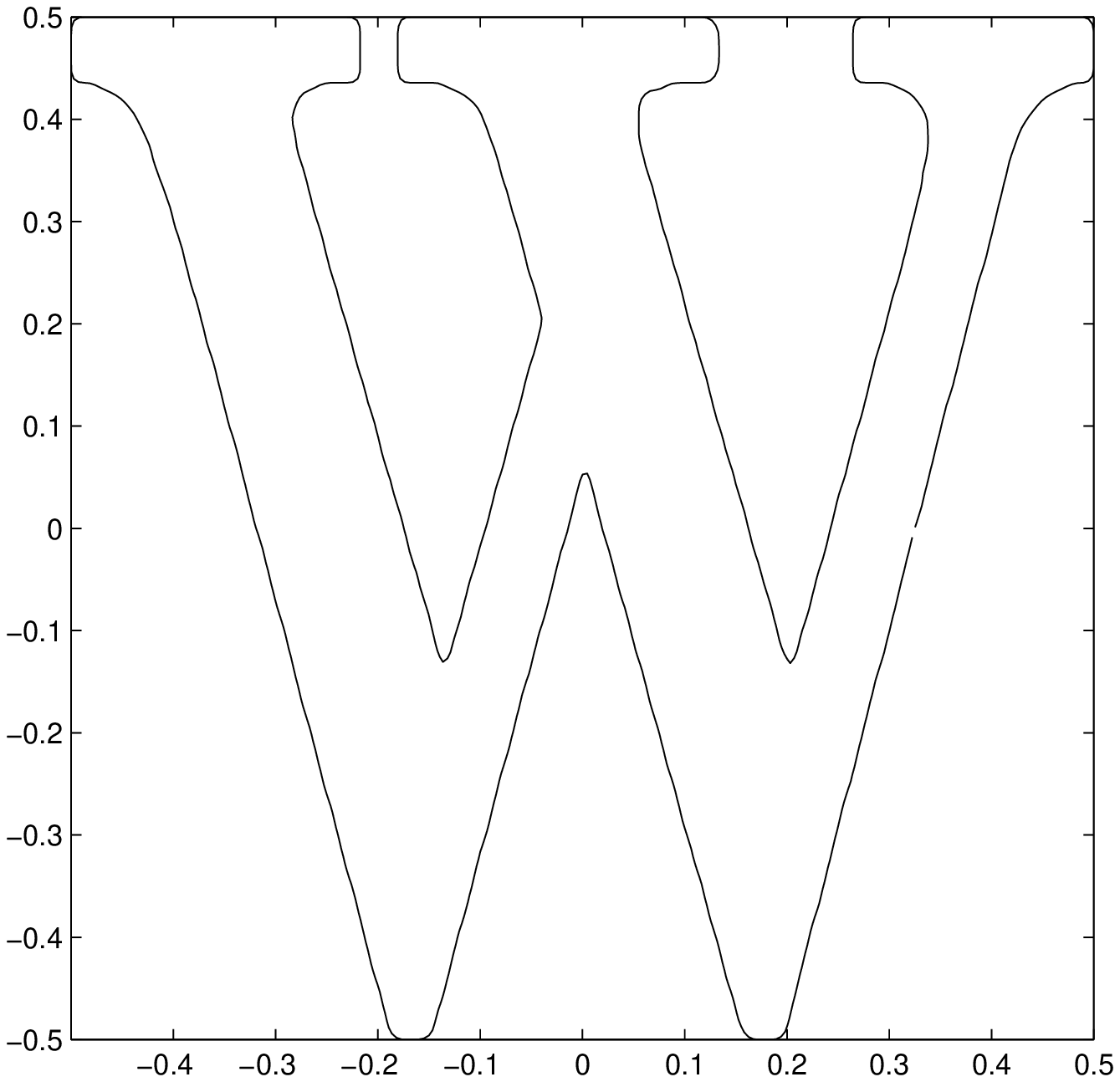}}
  \subfigure[]{\includegraphics[width=\lettersize]{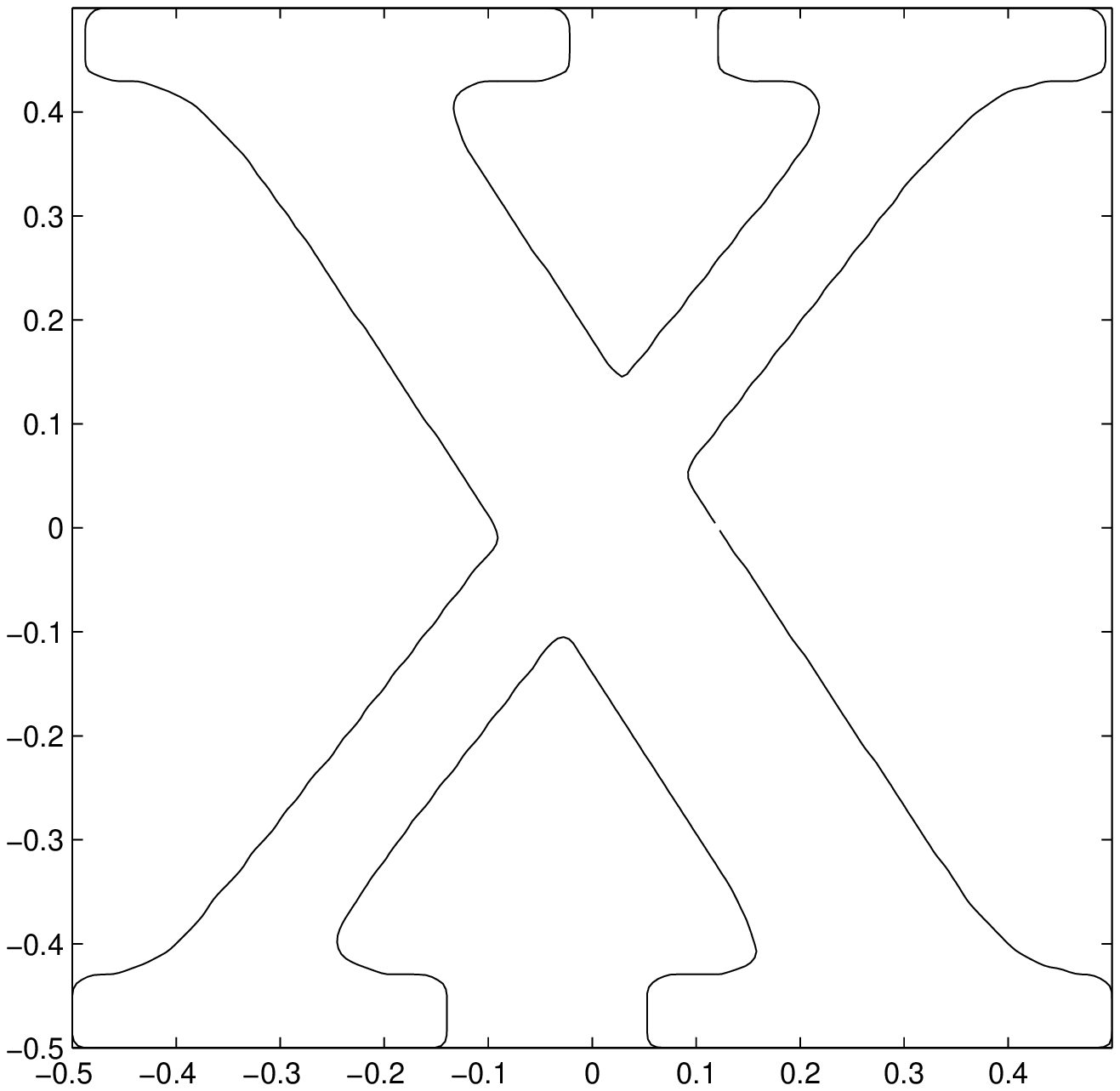}}
  \subfigure[]{\includegraphics[width=\lettersize]{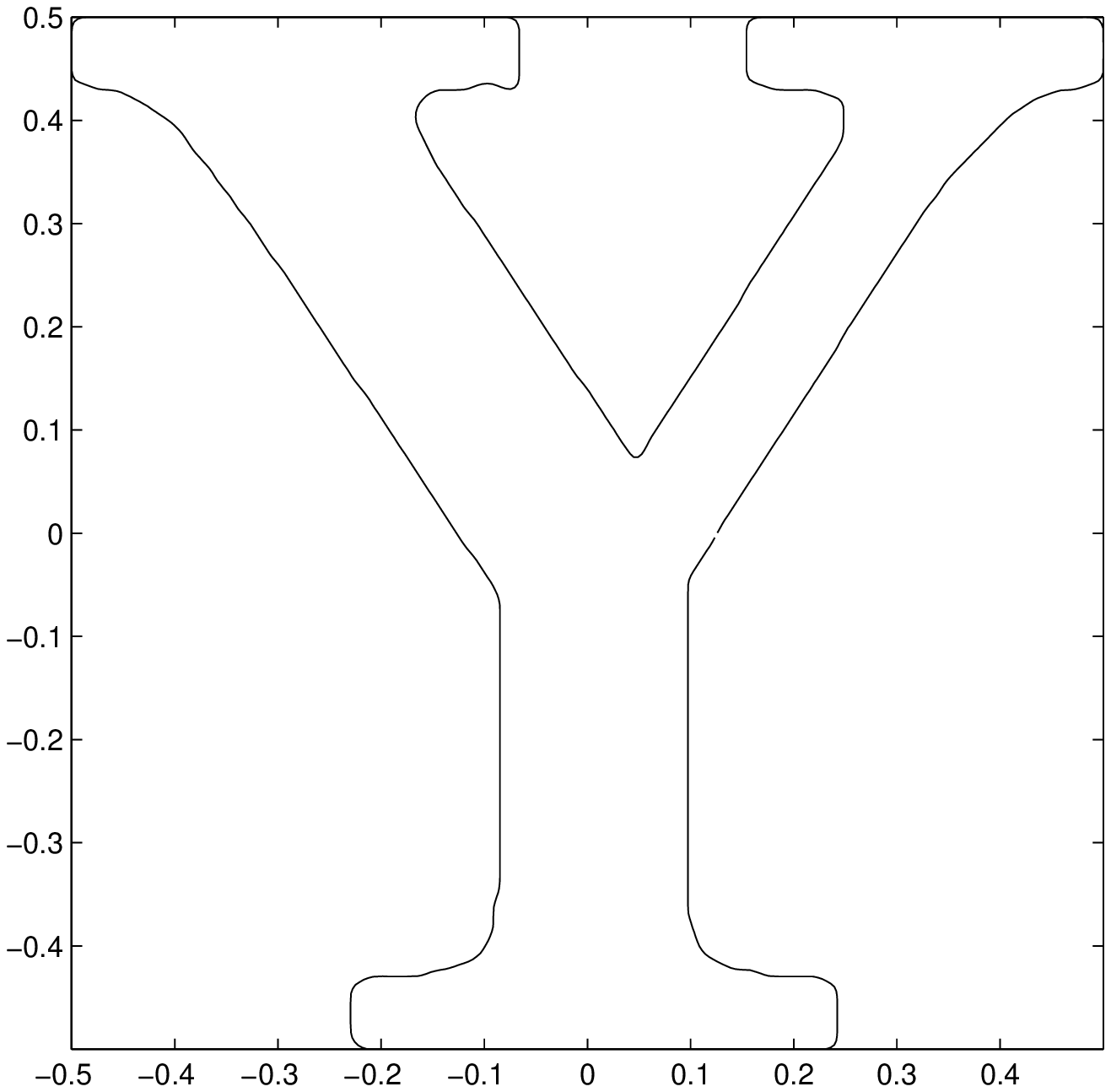}}
  \subfigure[]{\includegraphics[width=\lettersize]{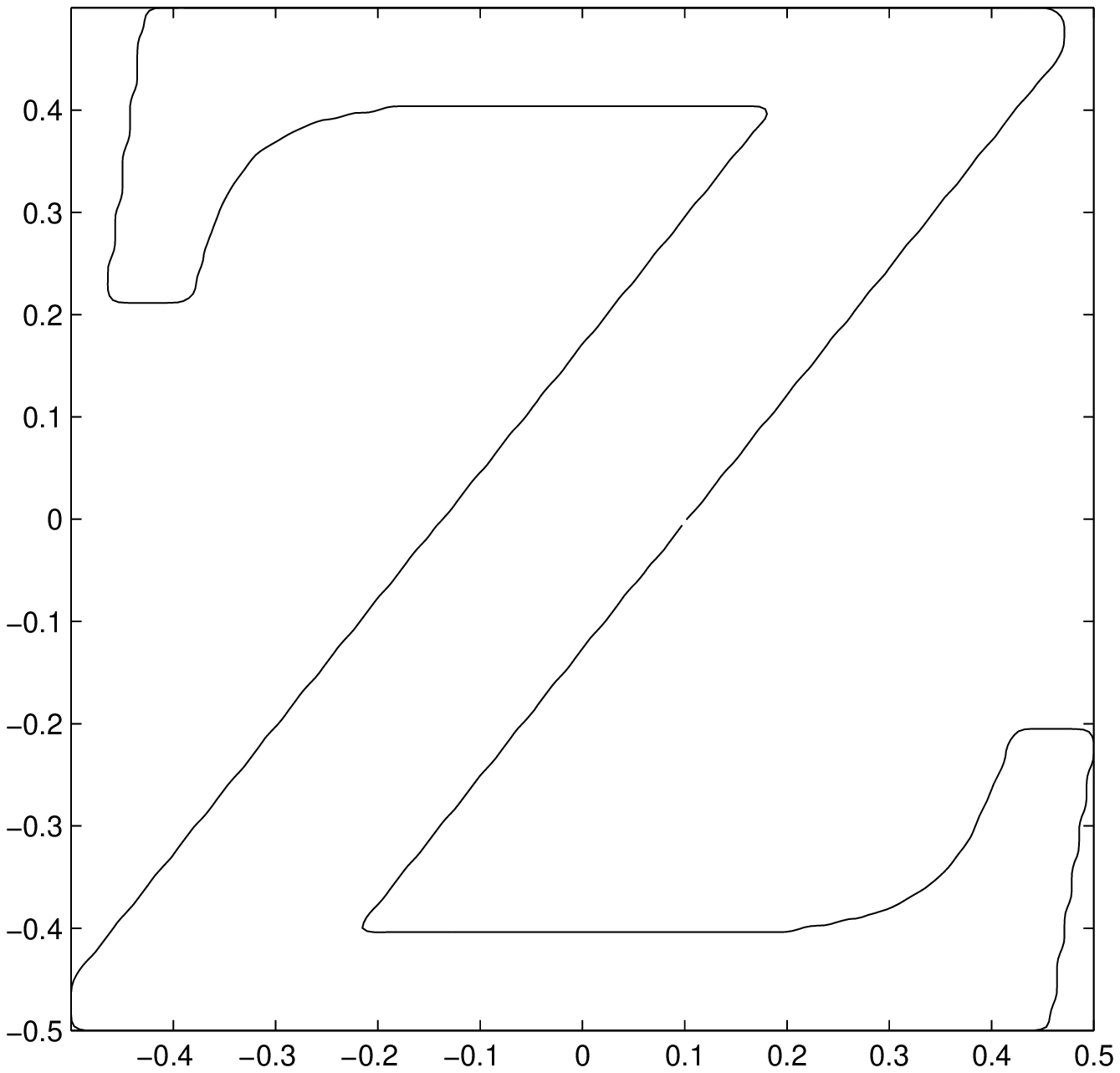}}
  \caption{Dictionary of standard letters.}
  \label{fig:all_letters_A_Z}
\end{figure}

\def\lettersize{2.7cm}
\begin{figure}[htp]
  \centering
  \subfigure[]{\includegraphics[width=\lettersize]{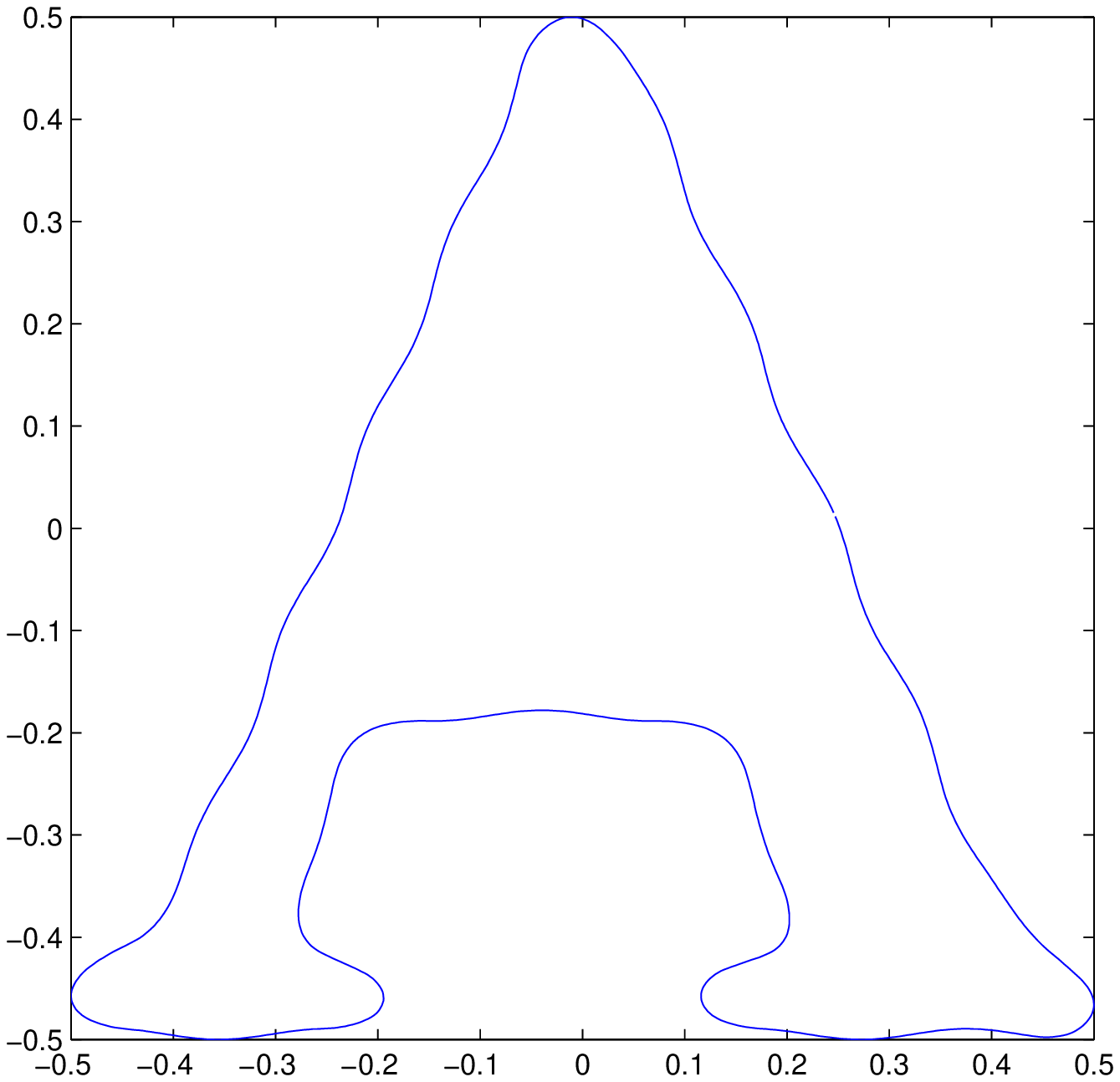}}
  \subfigure[]{\includegraphics[width=\lettersize]{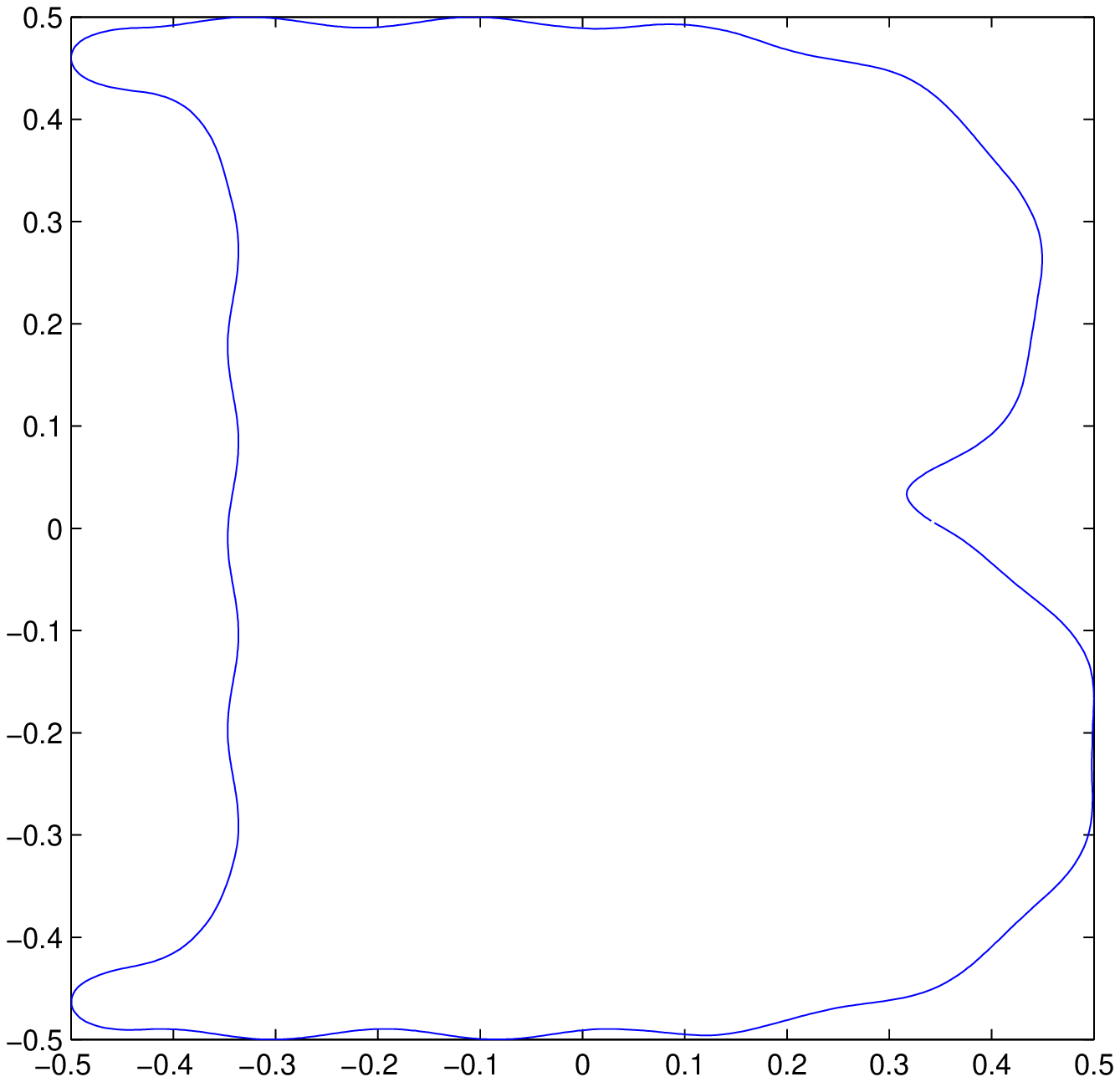}}
  \subfigure[]{\includegraphics[width=\lettersize]{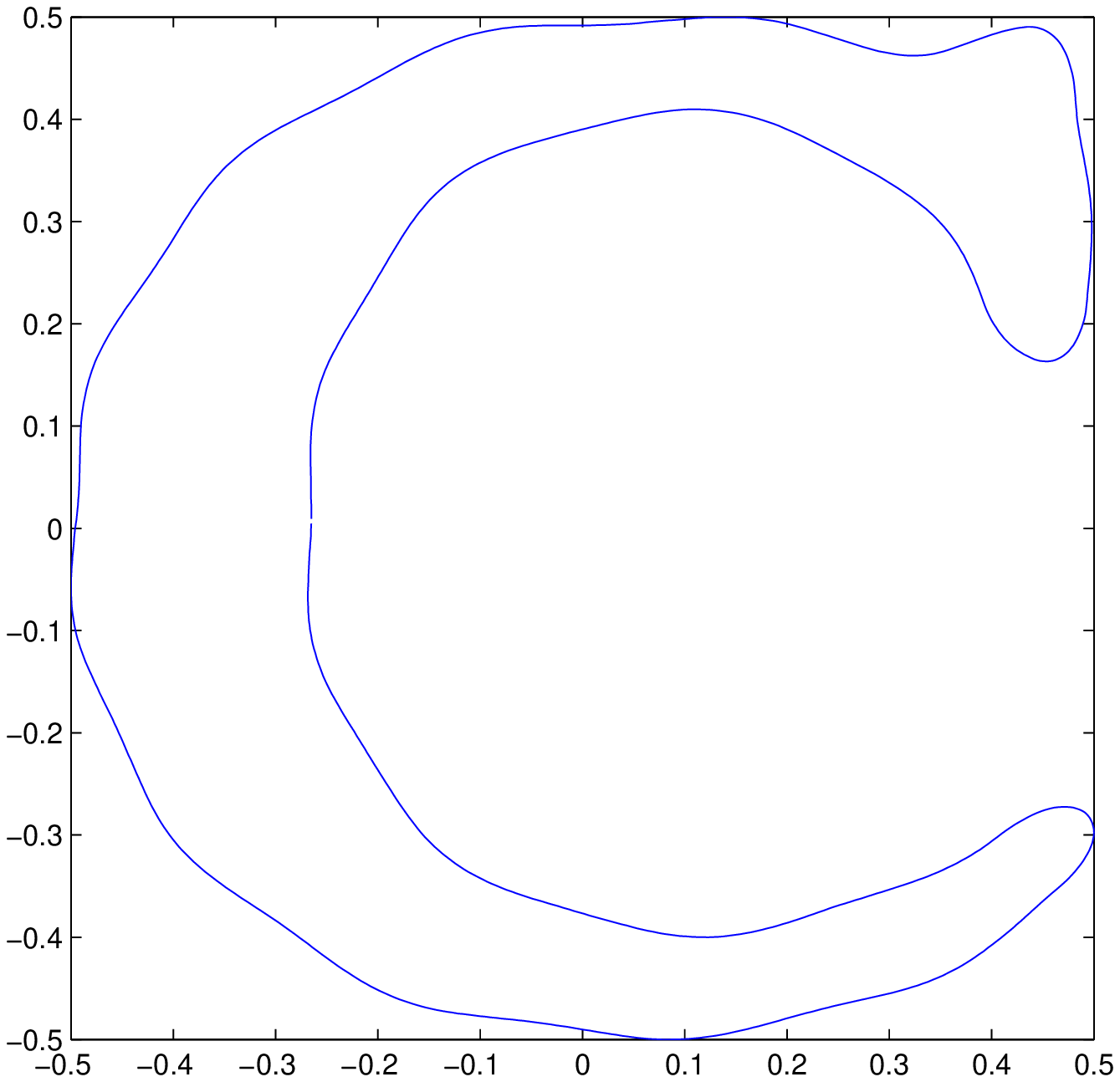}}
  \subfigure[]{\includegraphics[width=\lettersize]{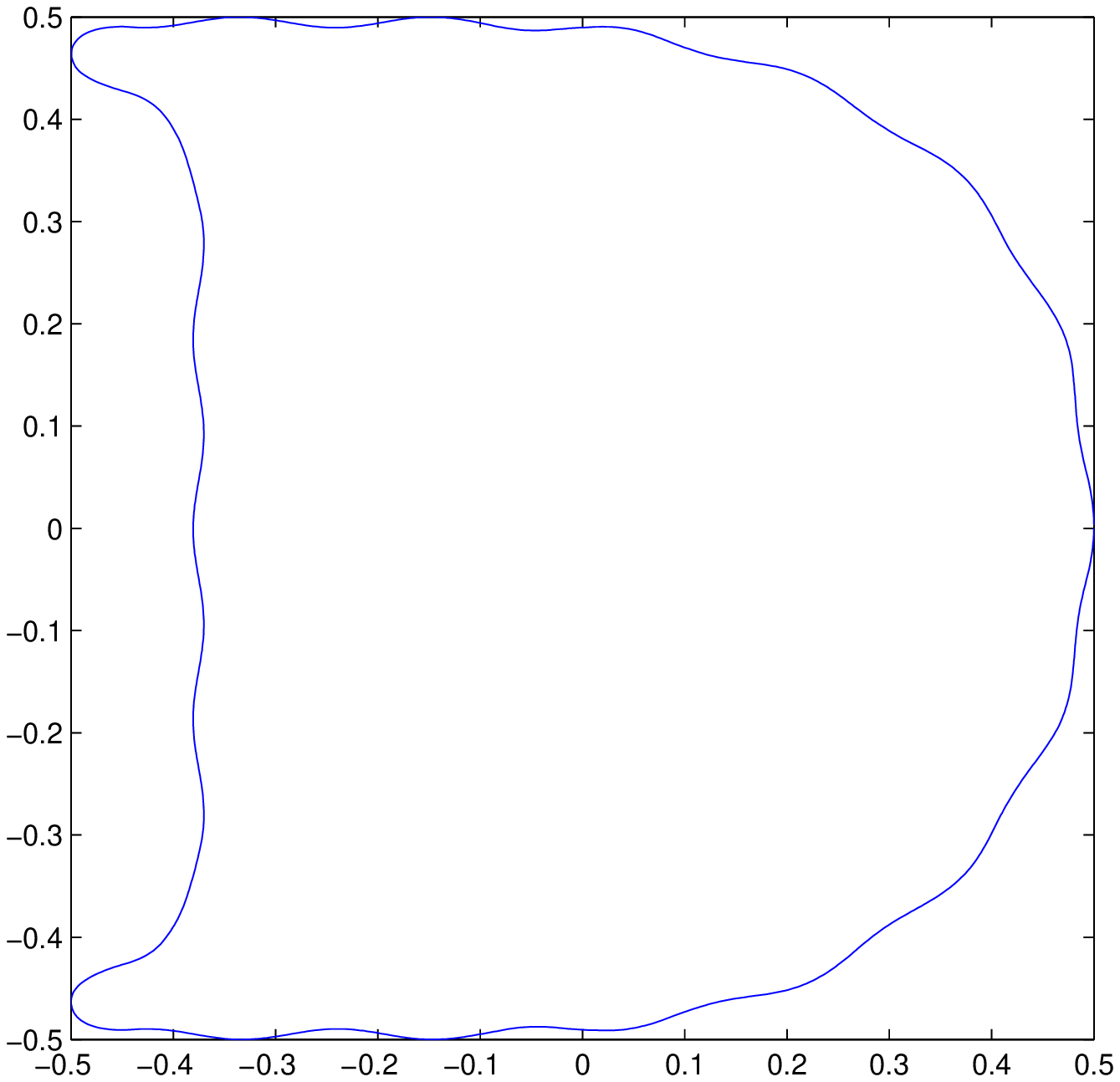}}
  \subfigure[]{\includegraphics[width=\lettersize]{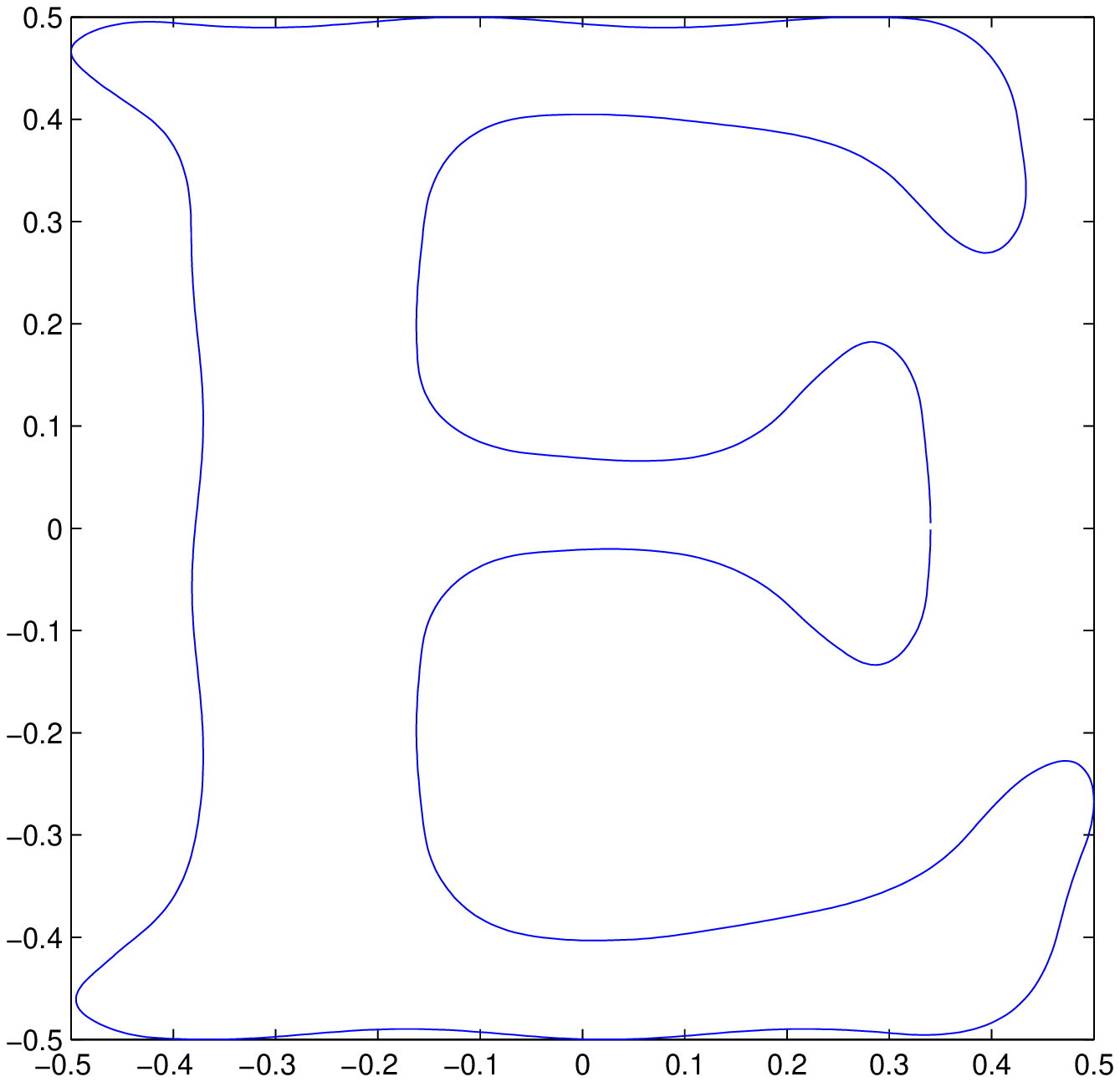}}
  \subfigure[]{\includegraphics[width=\lettersize]{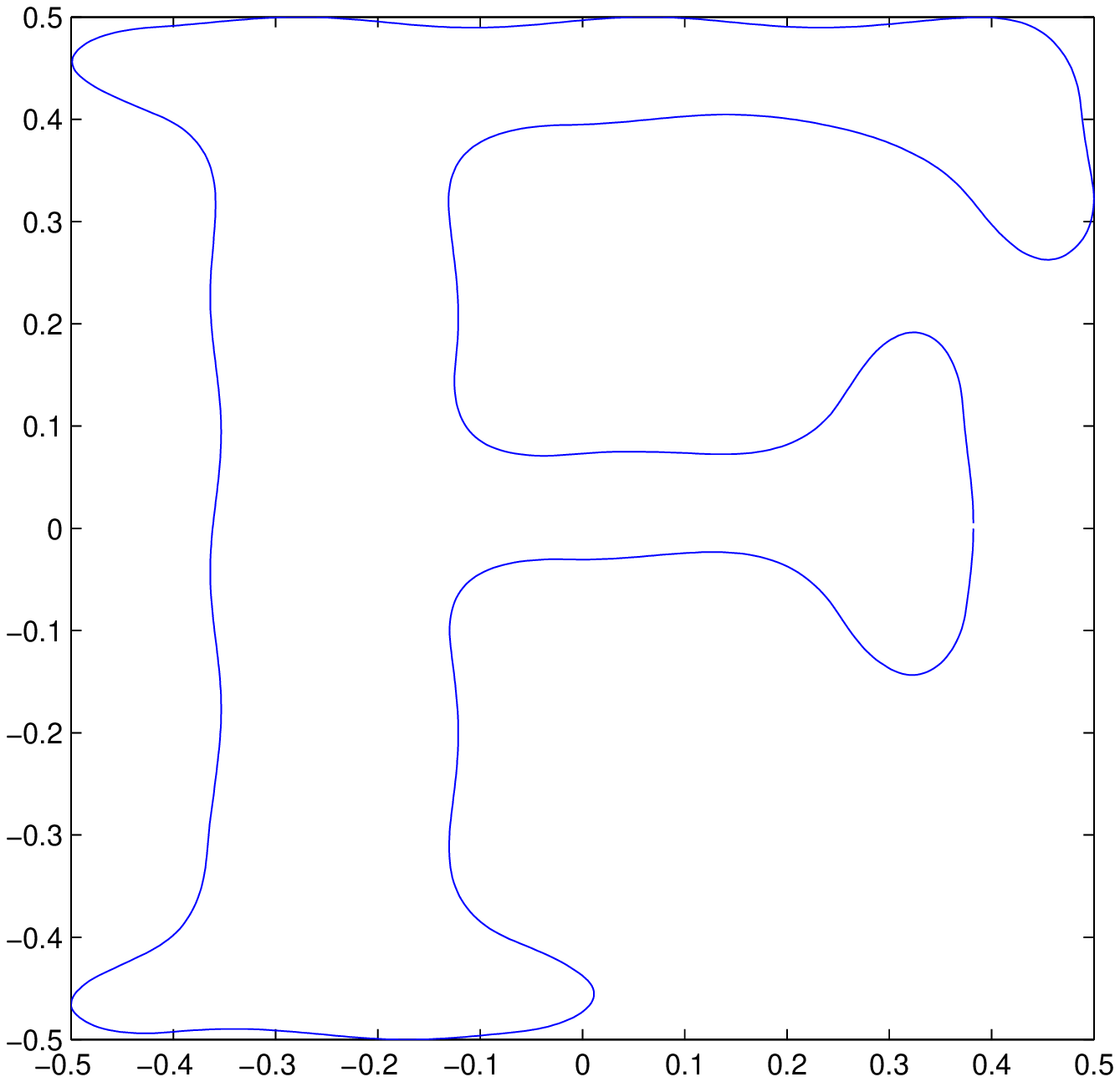}}
  \subfigure[]{\includegraphics[width=\lettersize]{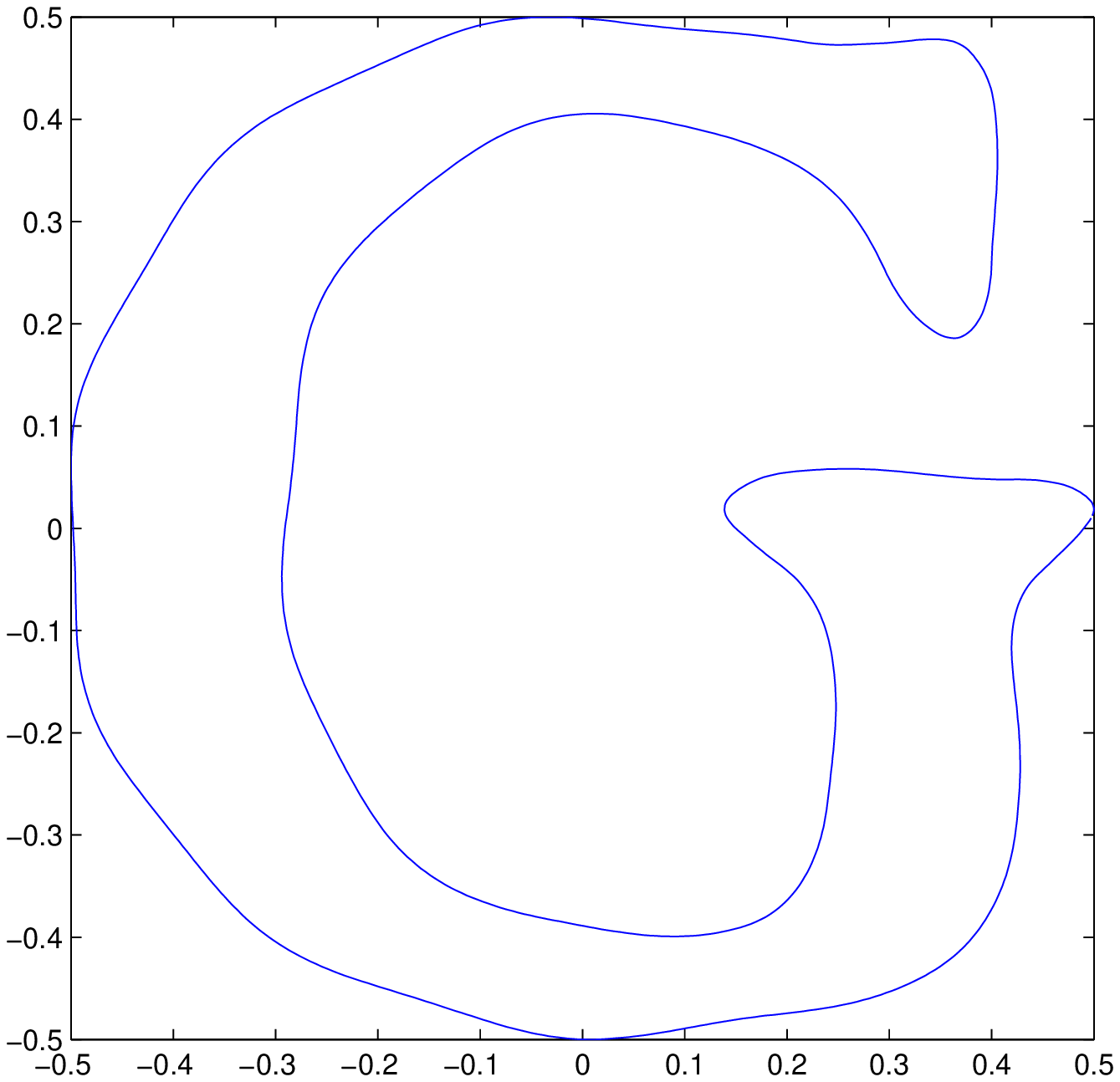}}
  \subfigure[]{\includegraphics[width=\lettersize]{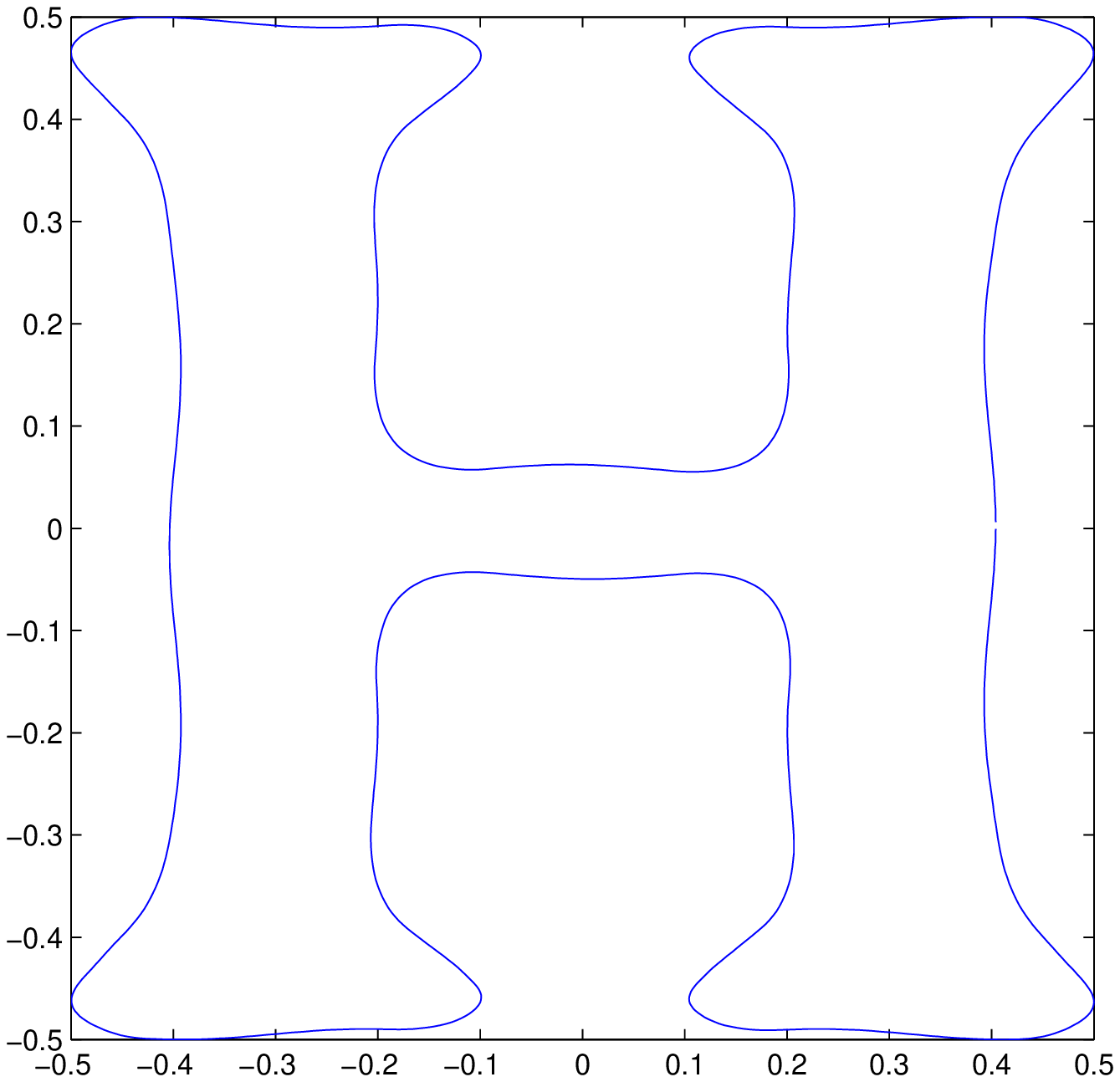}}
  \subfigure[]{\includegraphics[width=\lettersize]{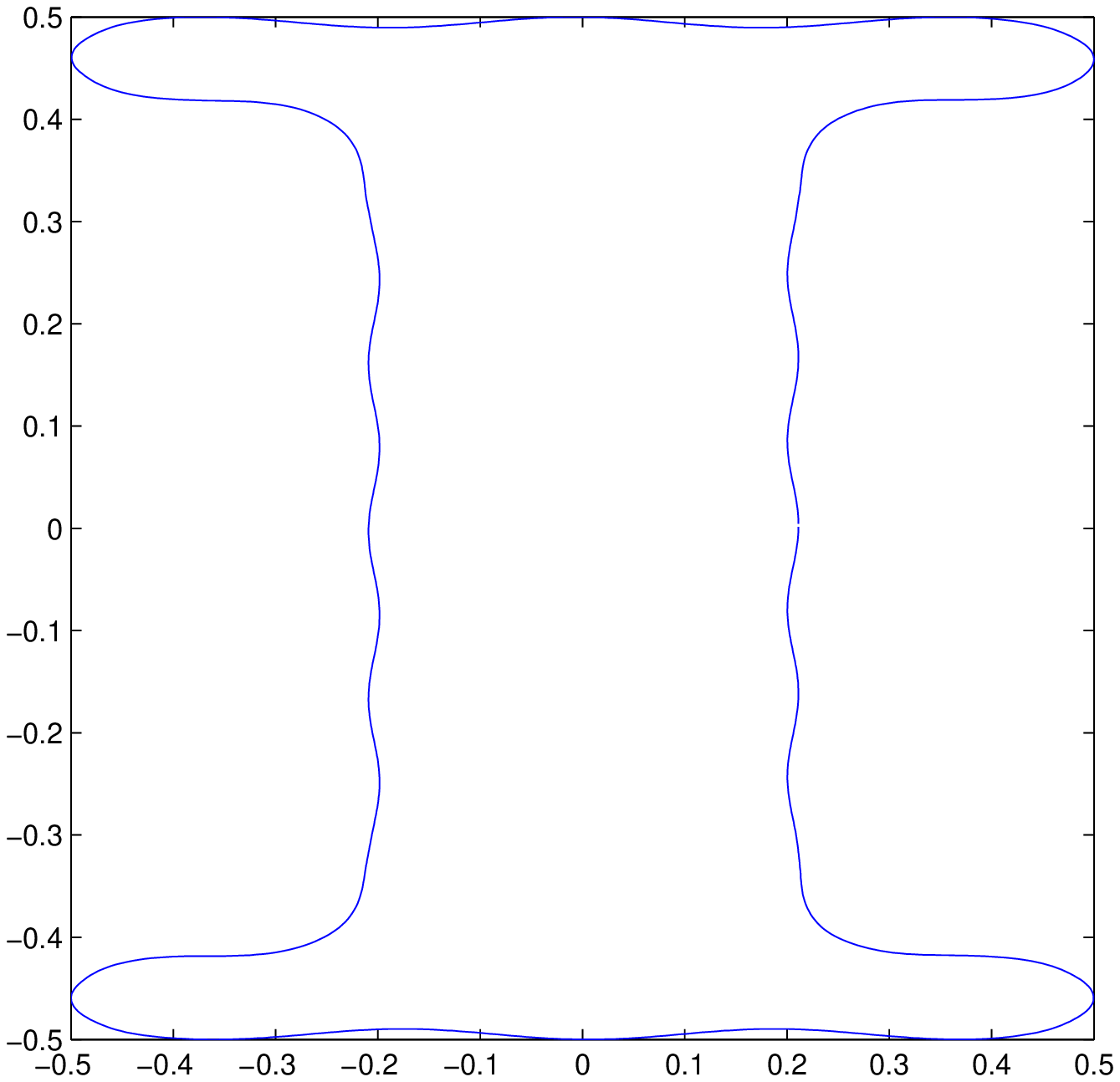}}
  \subfigure[]{\includegraphics[width=\lettersize]{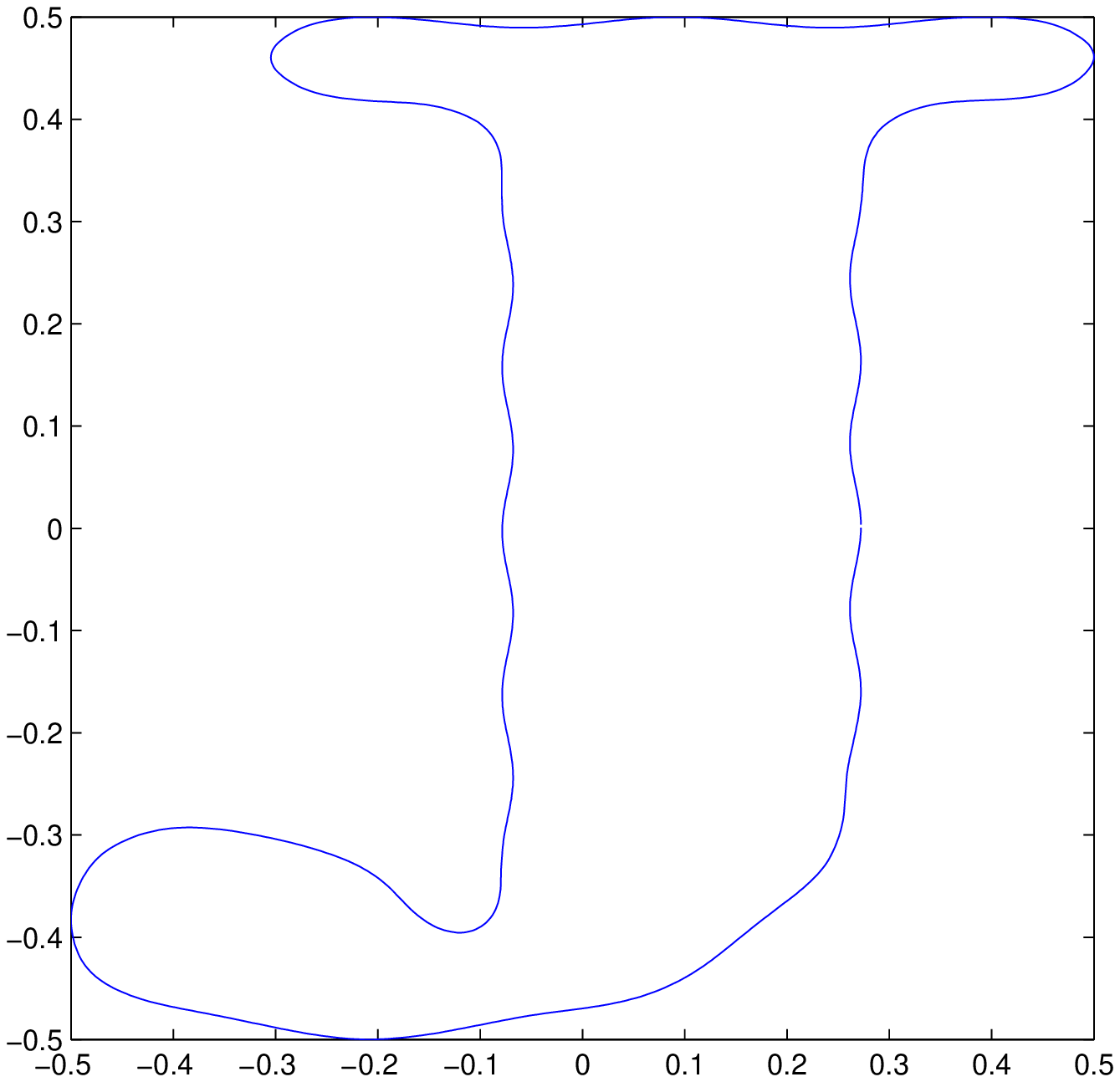}}
  \subfigure[]{\includegraphics[width=\lettersize]{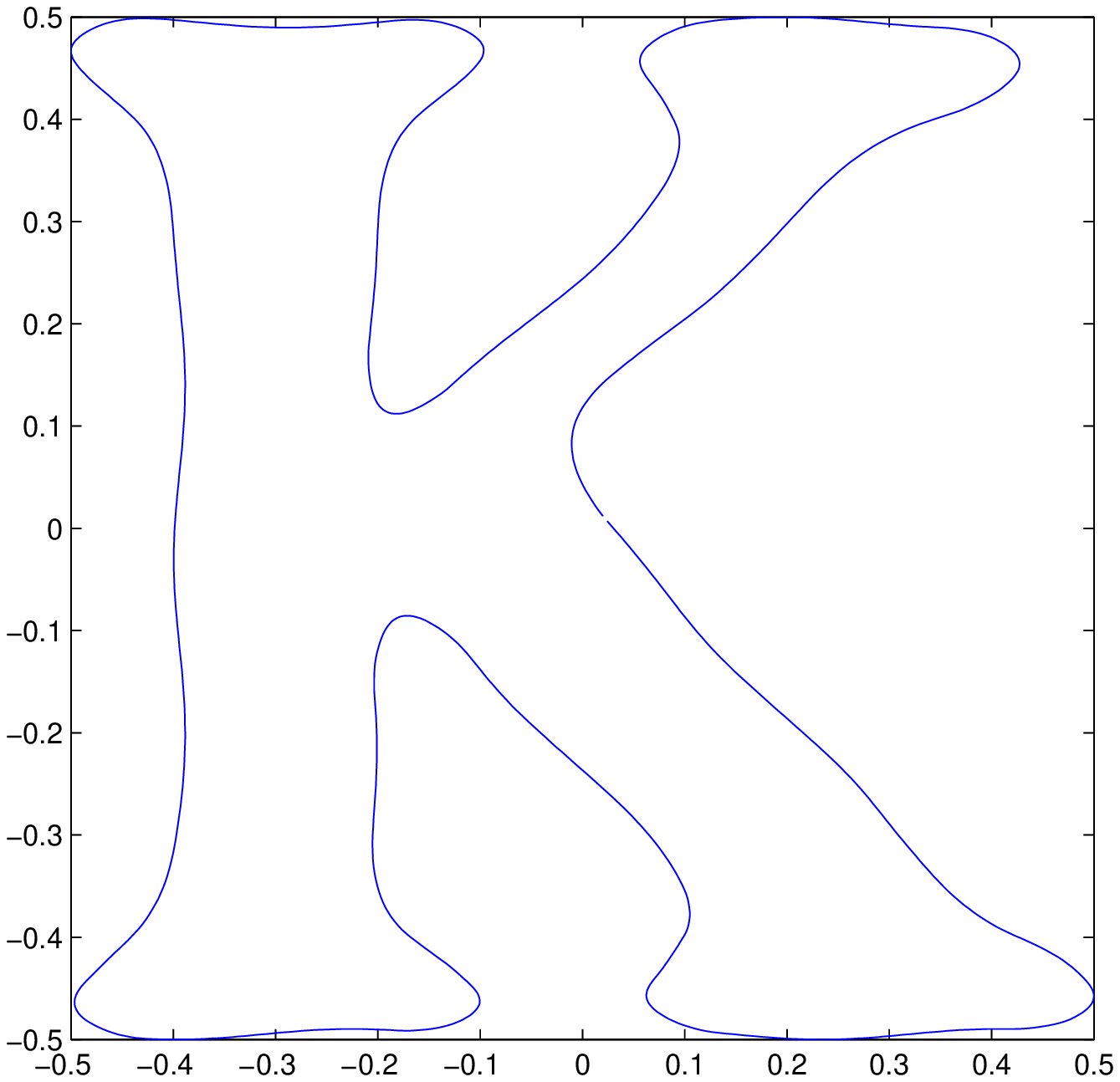}}
  \subfigure[]{\includegraphics[width=\lettersize]{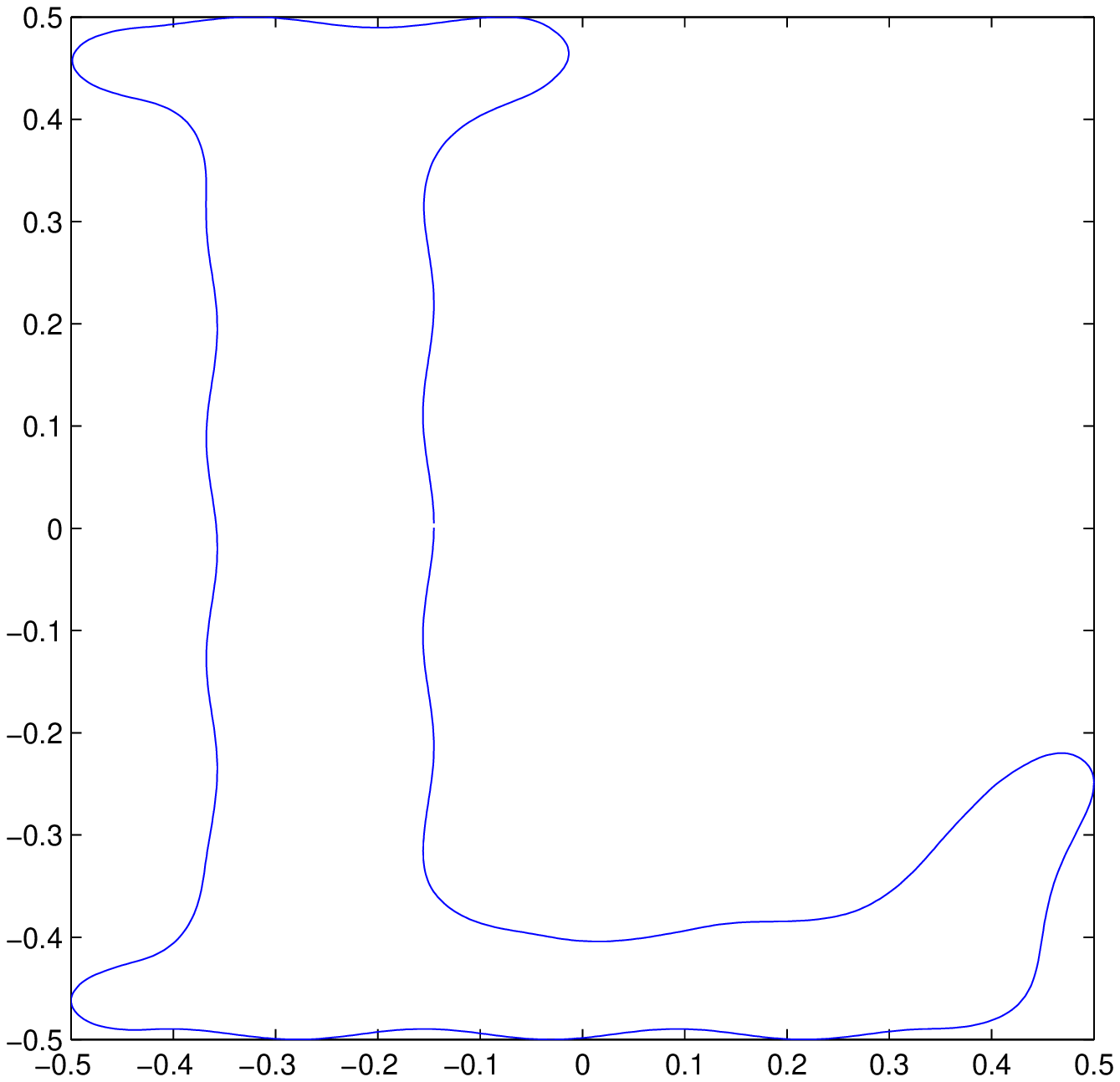}}
  \subfigure[]{\includegraphics[width=\lettersize]{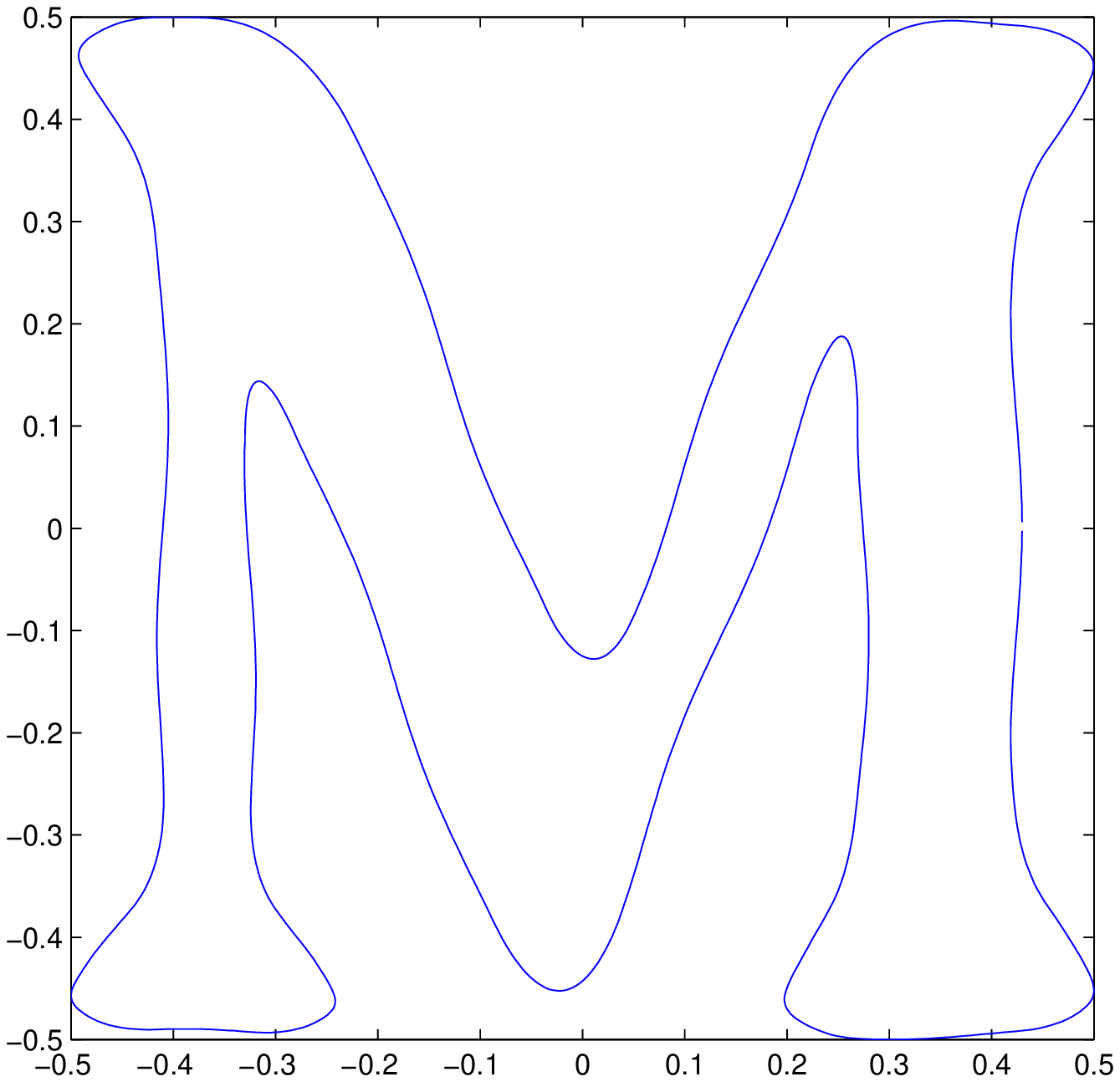}}
  \subfigure[]{\includegraphics[width=\lettersize]{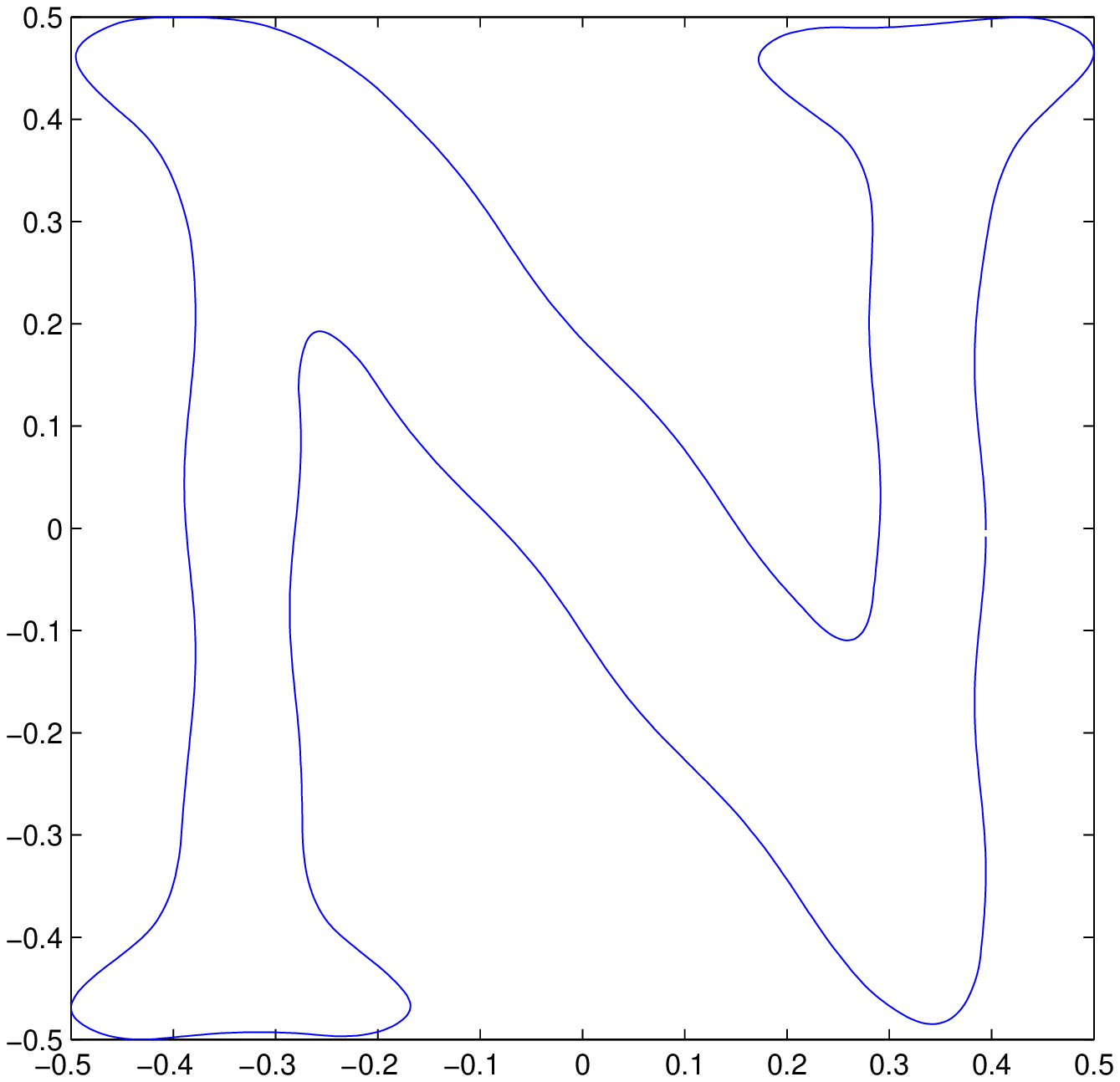}}
  \subfigure[]{\includegraphics[width=\lettersize]{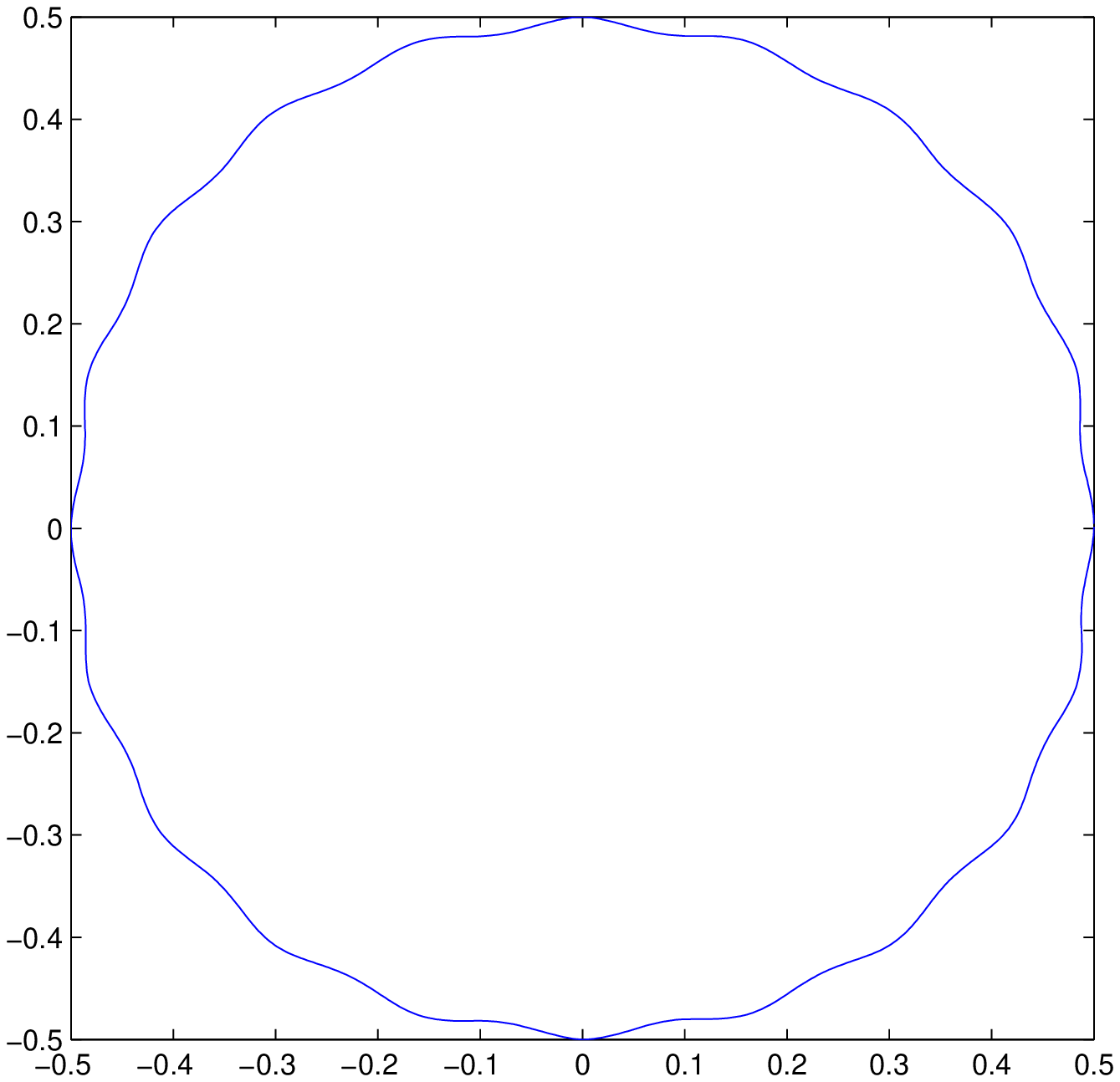}}
  \subfigure[]{\includegraphics[width=\lettersize]{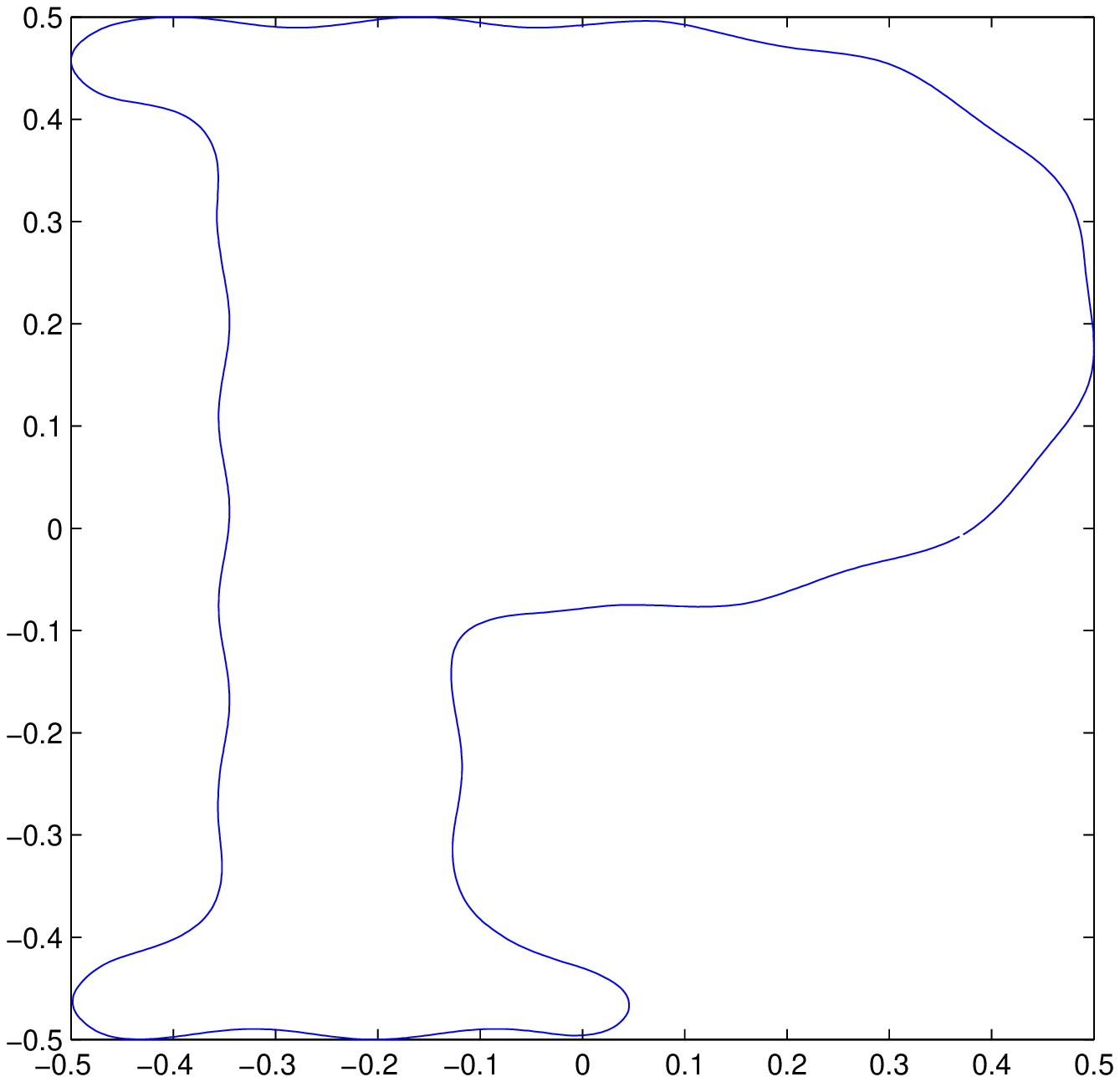}}
  \subfigure[]{\includegraphics[width=\lettersize]{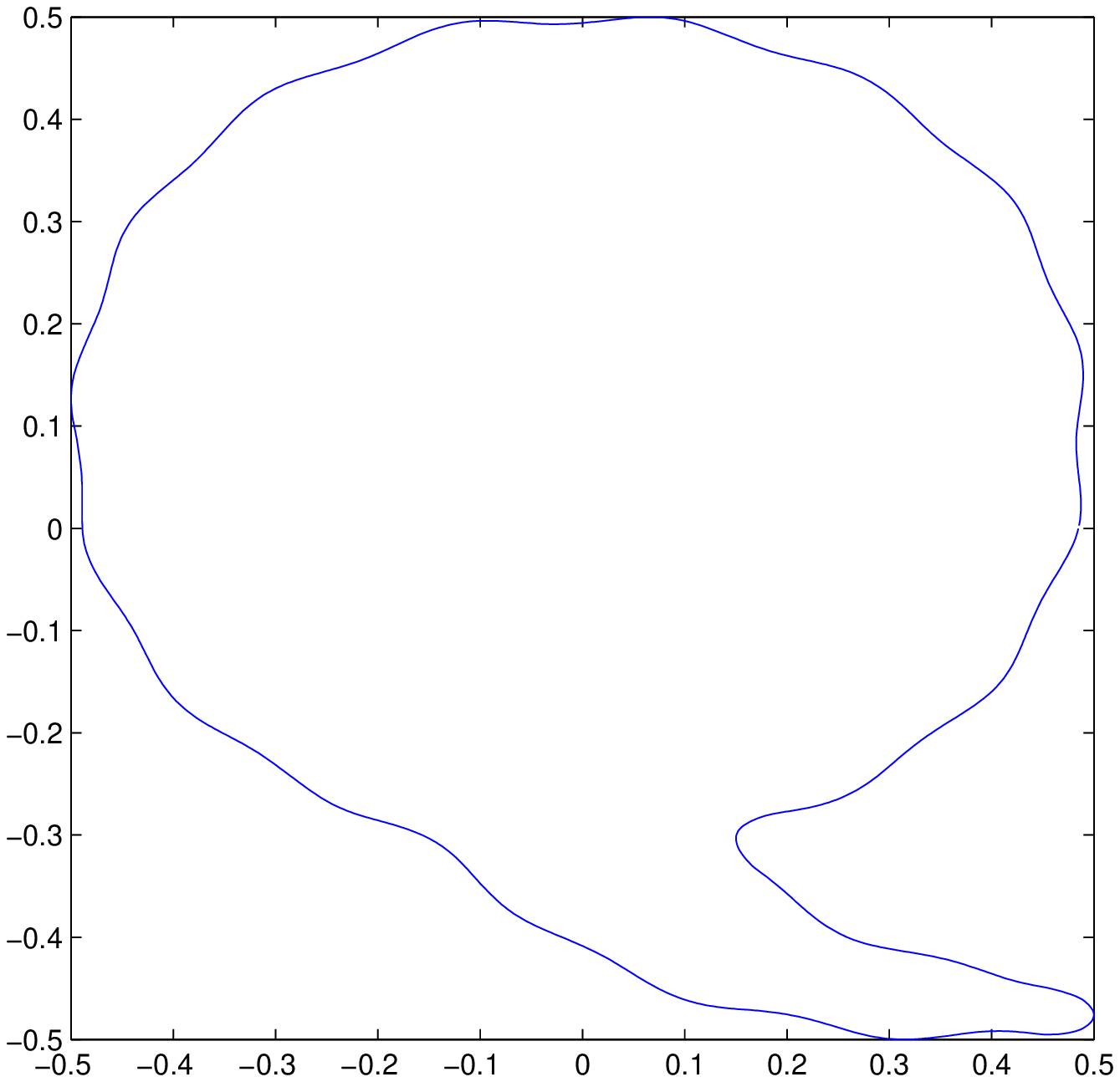}}
  \subfigure[]{\includegraphics[width=\lettersize]{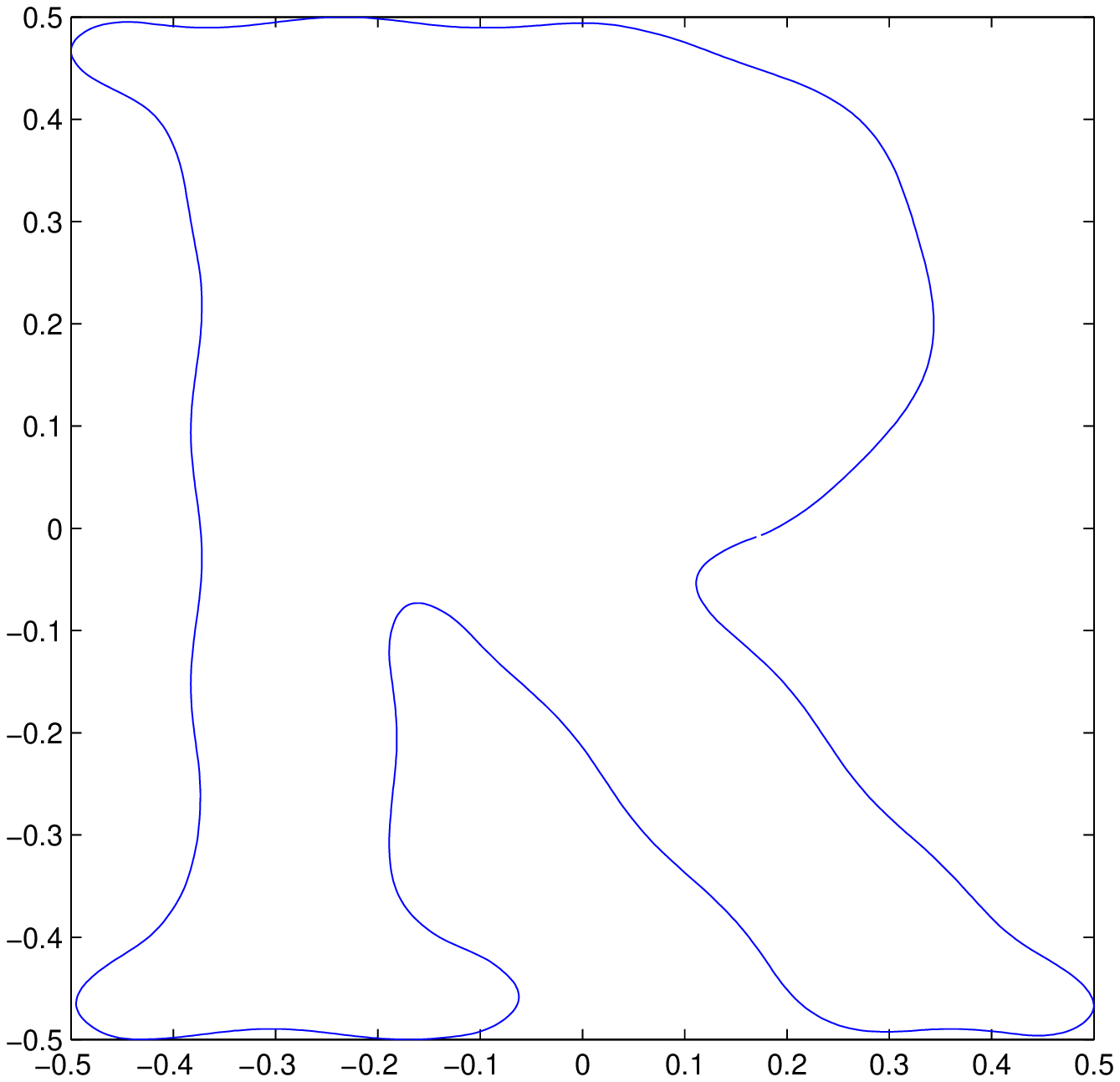}}
  \subfigure[]{\includegraphics[width=\lettersize]{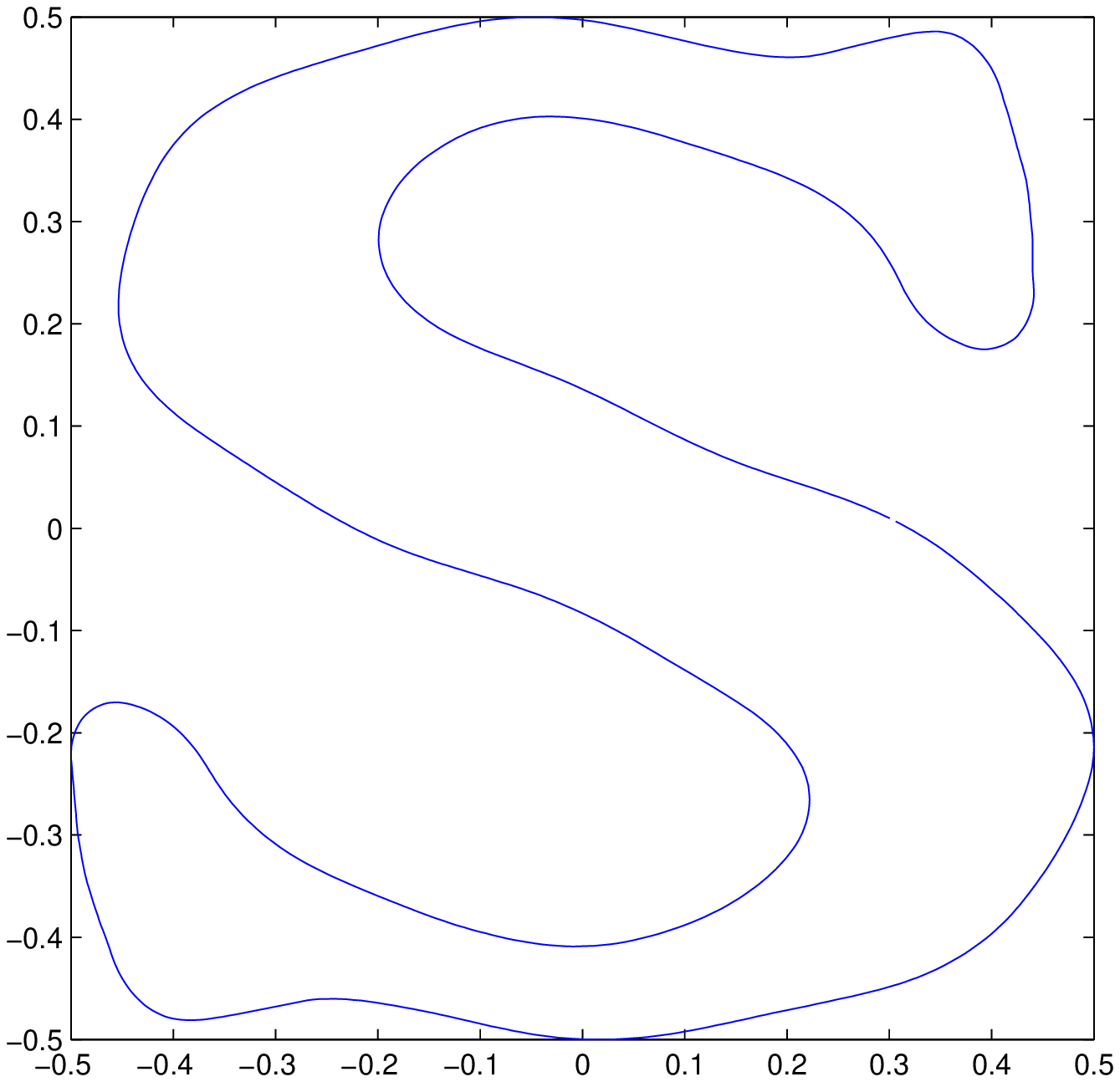}}
  \subfigure[]{\includegraphics[width=\lettersize]{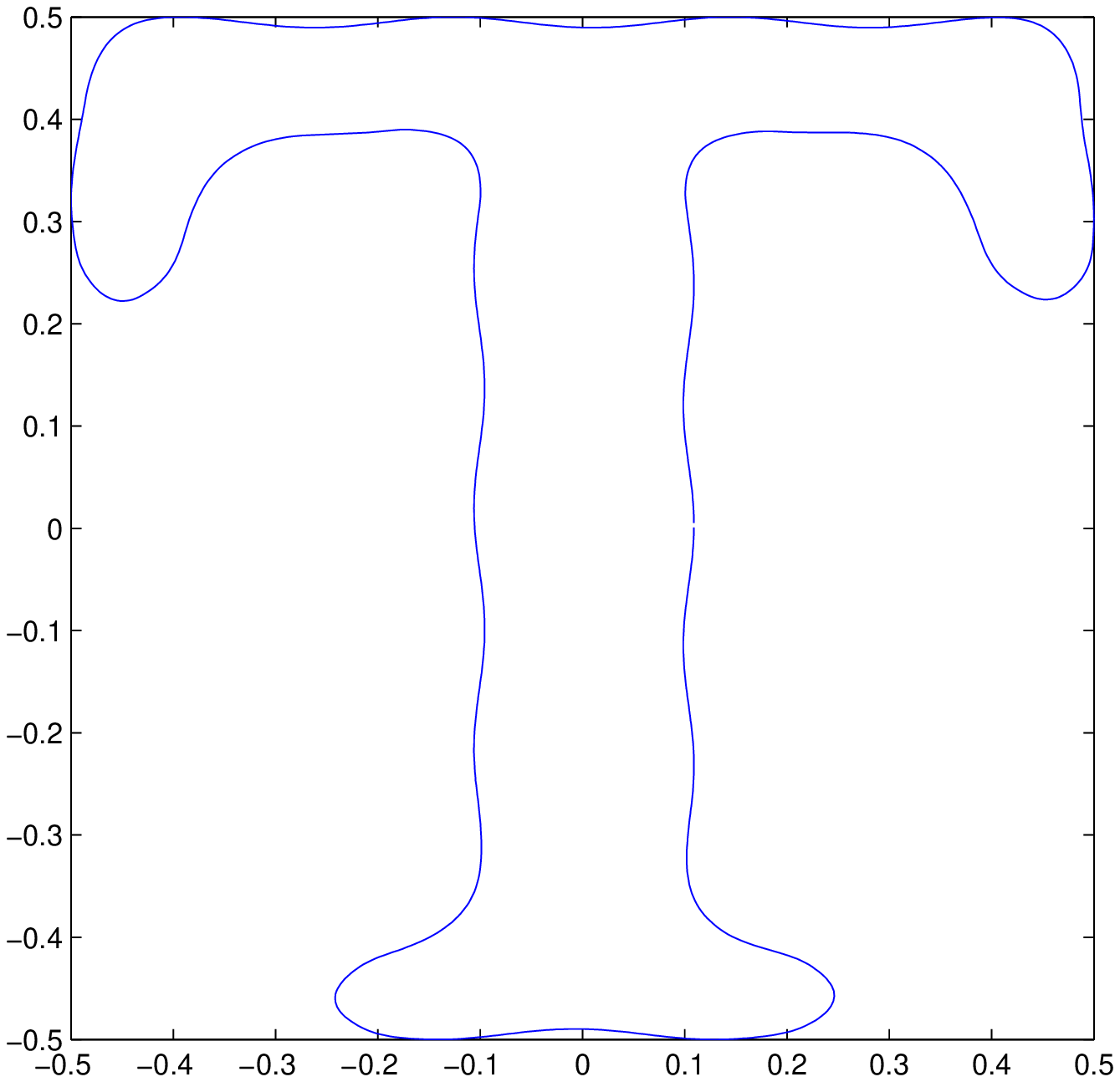}}
  \subfigure[]{\includegraphics[width=\lettersize]{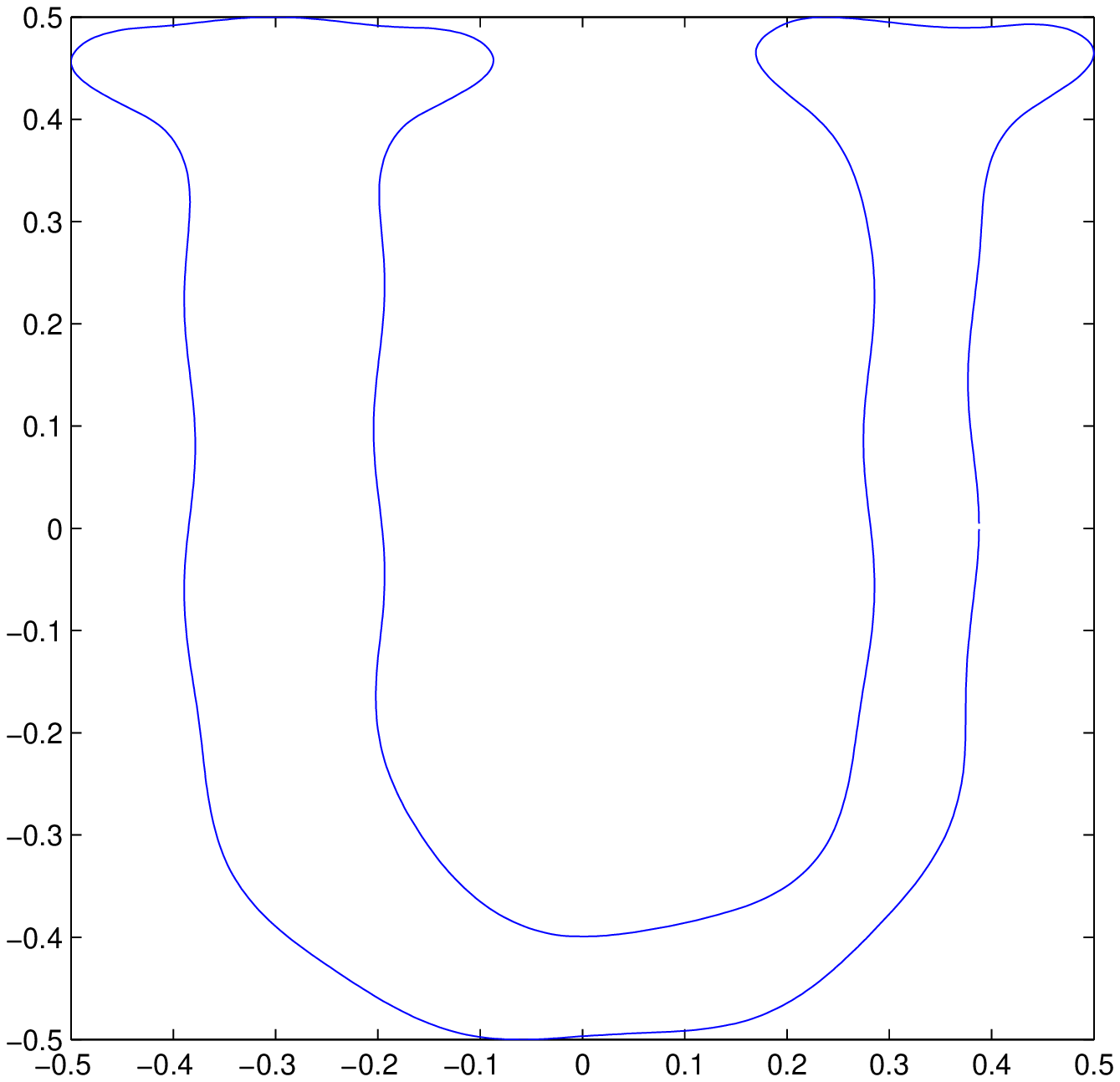}}
  \subfigure[]{\includegraphics[width=\lettersize]{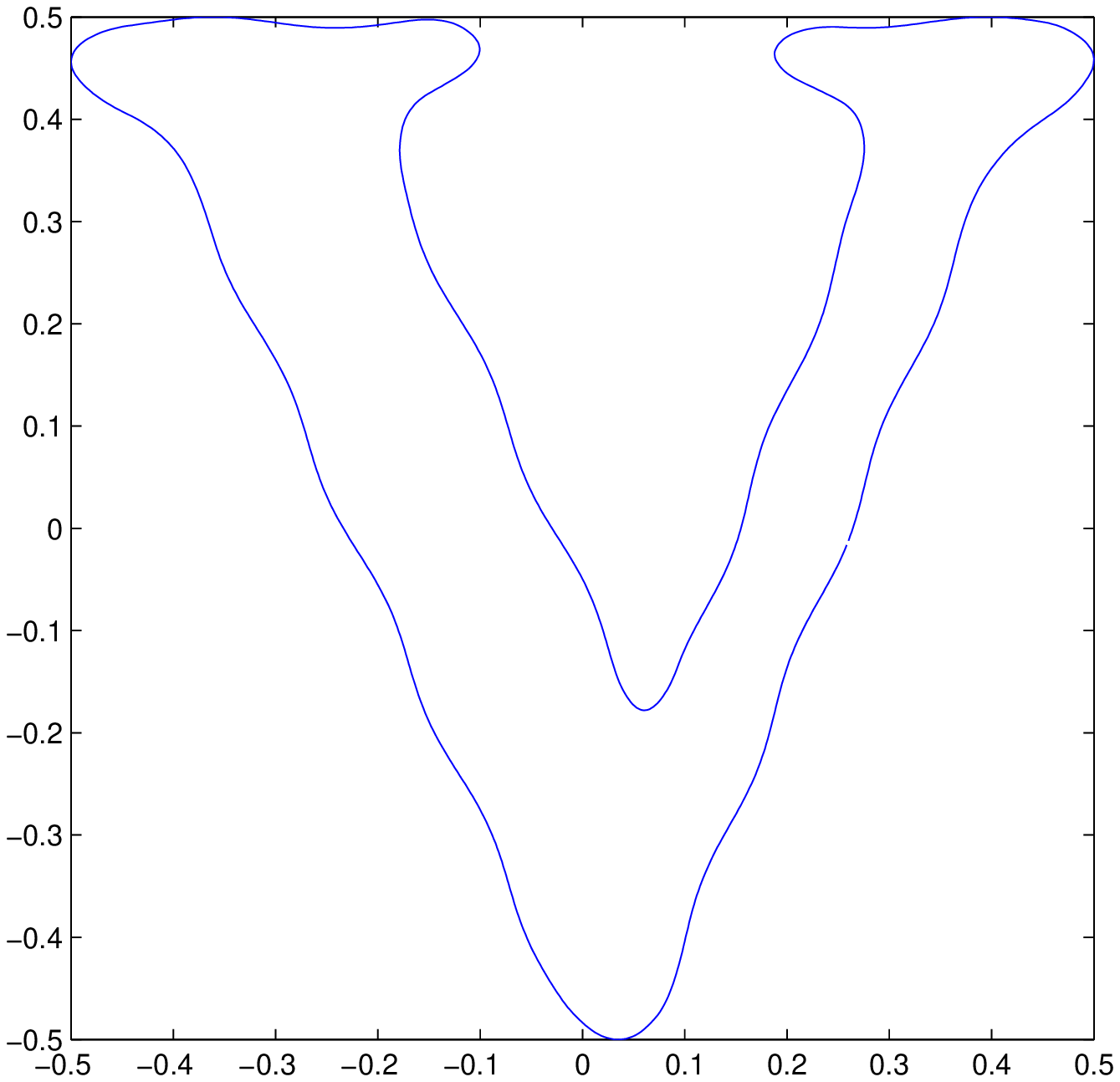}}
  \subfigure[]{\includegraphics[width=\lettersize]{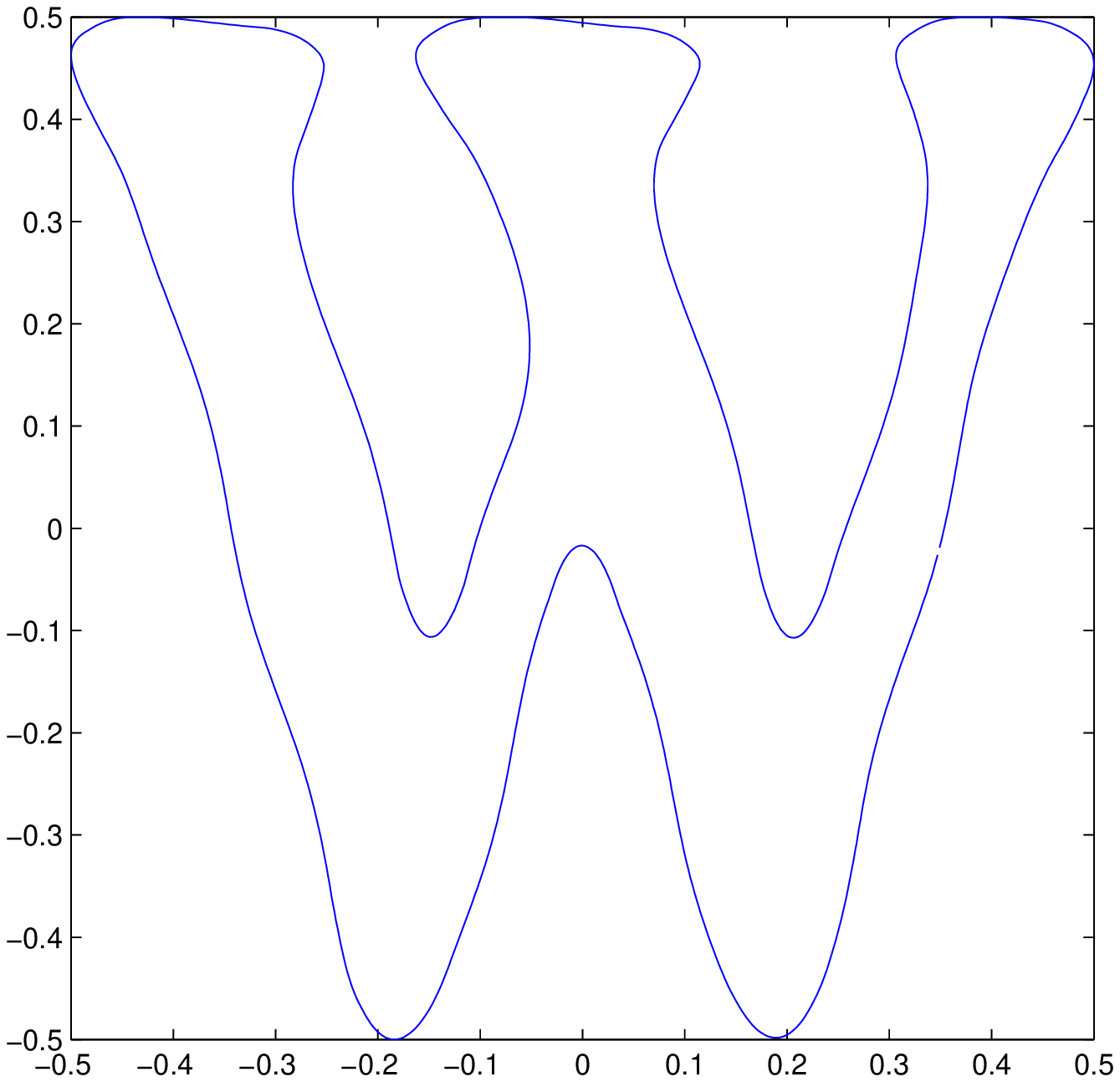}}
  \subfigure[]{\includegraphics[width=\lettersize]{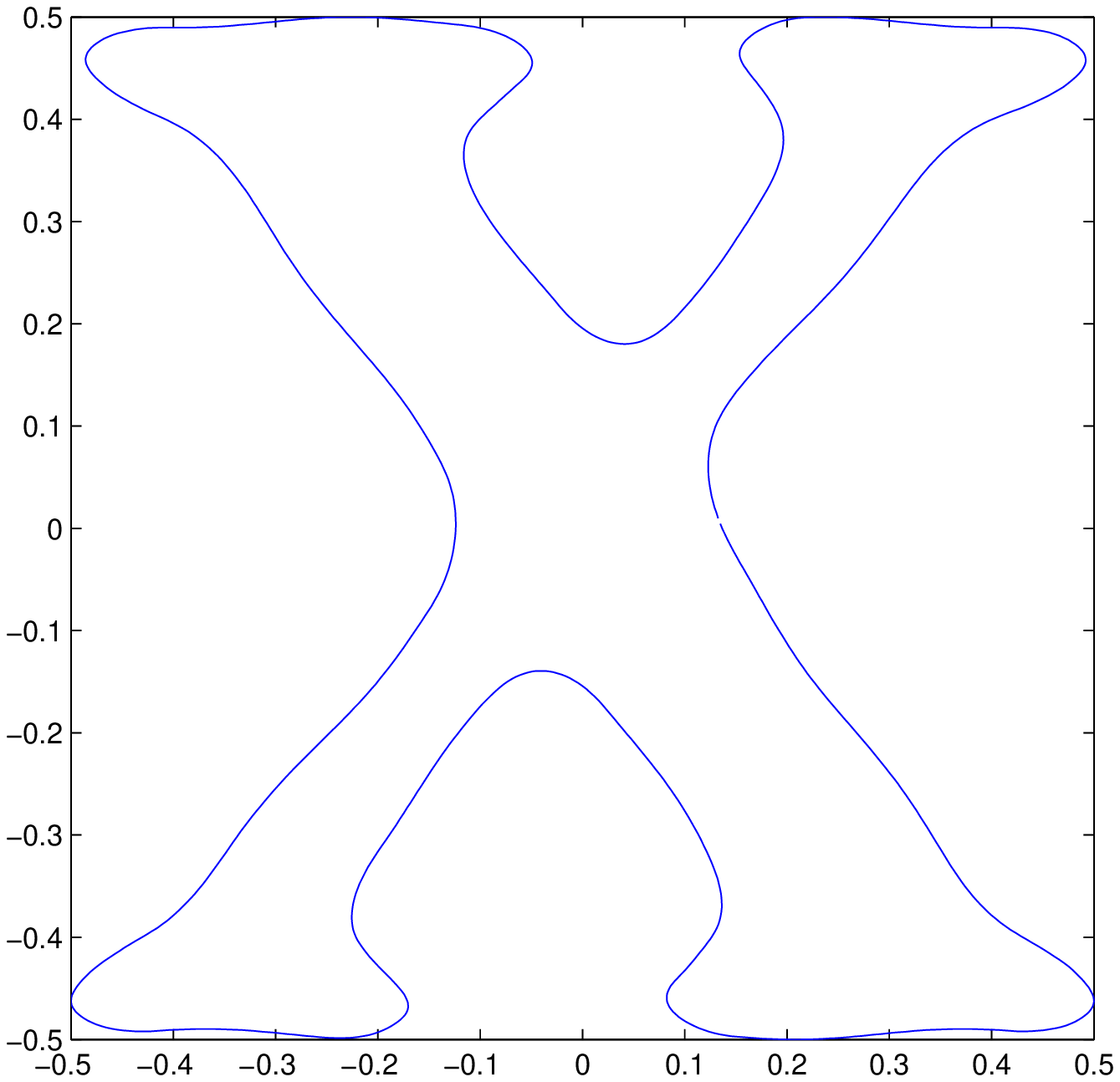}}
  \subfigure[]{\includegraphics[width=\lettersize]{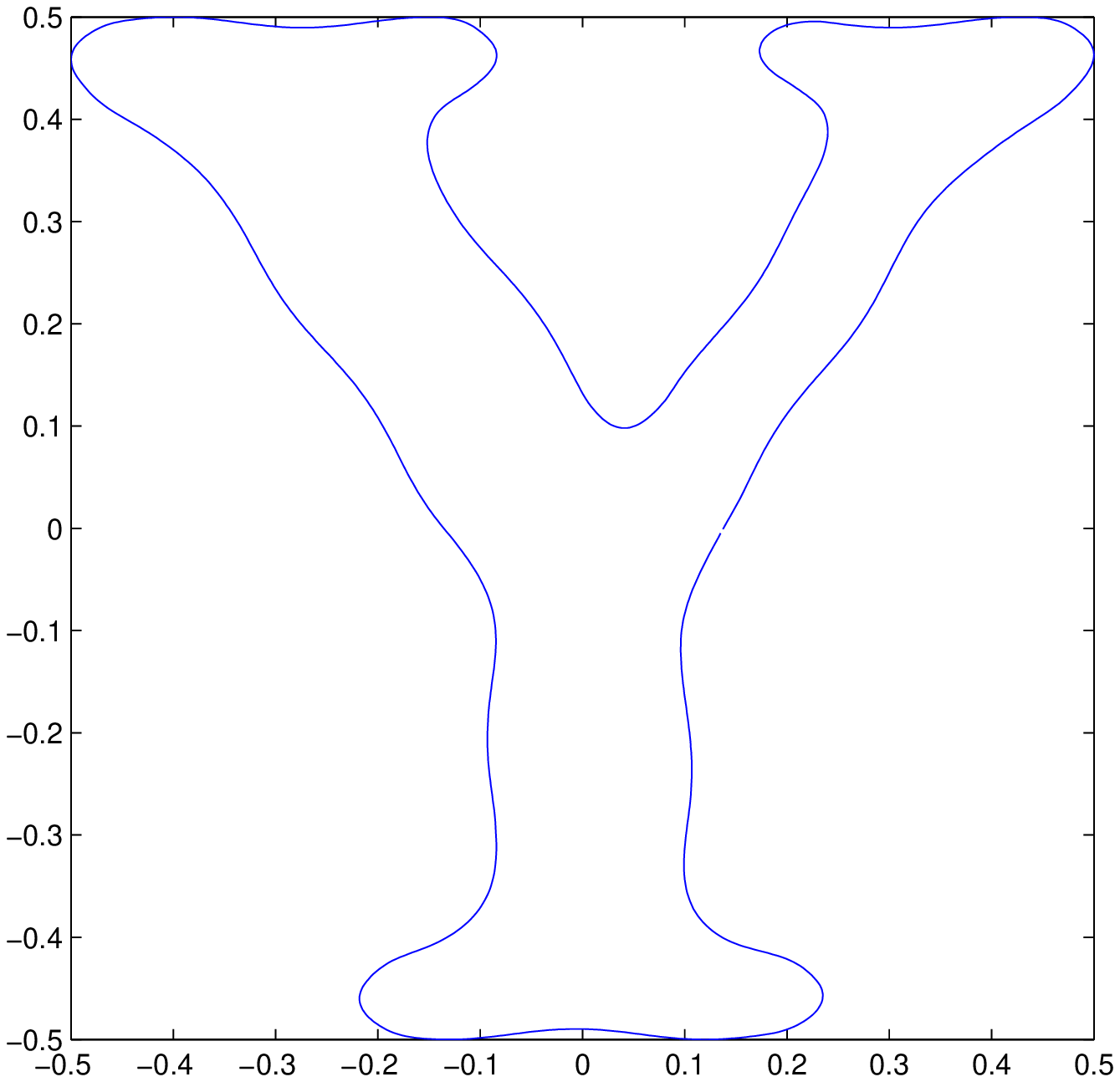}}
  \subfigure[]{\includegraphics[width=\lettersize]{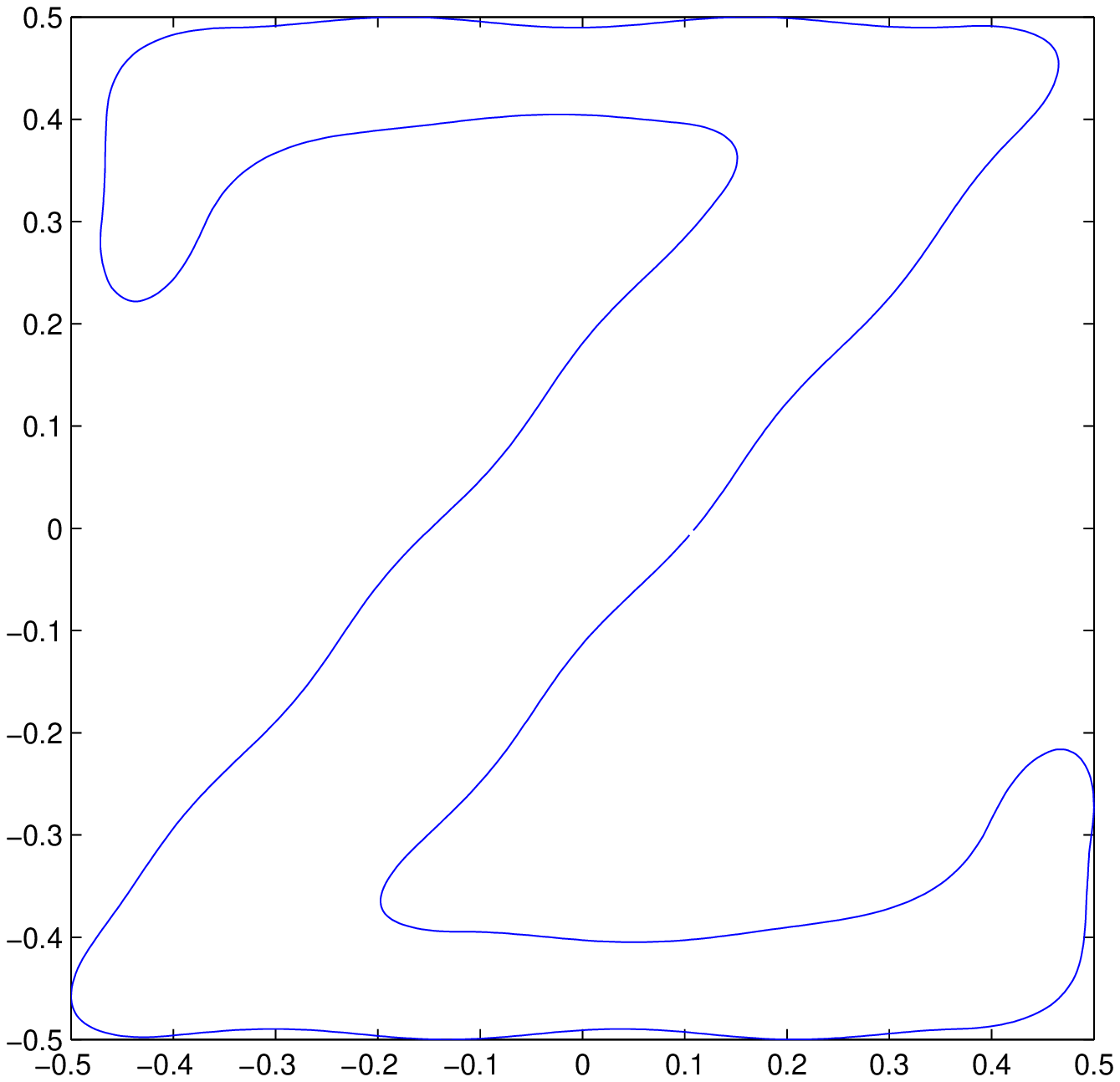}}
  \caption{Non standard letters obtained by perturbing and smoothing
    those in Figure~\ref{fig:all_letters_A_Z}.}
  \label{fig:all_ptb_letters_A_Z}
\end{figure}

\section{Conclusion}\label{sec:conclusion}
In this paper, we have designed two fast algorithms which identify
a target using a dictionary of precomputed GPTs data. The target
GPTs are computed from multistatic measurements by solving a
linear system. The first algorithm matches the computed GPTs to
precomputed ones (the dictionary elements) by finding rotation,
scaling, and translation parameters and therefore, identifies the
true target shape. The second algorithm  is based on new
invariants for the CGPTs. We have provided new shape descriptors
which are invariant under translation, rotation, and scaling. The
stability (in the presence of additive noise in multistatic
measurements) and the resolution issues for both algorithms have
been numerically investigated. The second algorithm is
computationally much cheaper than the first one. However, it is
more sensitive to measurement noise in the imaging data. To the
best of our knowledge, our procedure is the first approach for
real-time target identification in imaging using dictionary
matching. It shows that GPT-based representations are an
appropriate and natural tool for imaging. Our approach can be
extended to electromagnetic and elastic imaging as well
\cite{mc2,resol}. We also to plan to use it for target tracking
from imaging data.


\appendix
\section{Appendix: Several Technical Estimates}\label{sec:append-sever-techn}
\subsection{The truncation error in the MSR expansion}
\label{sec:app1}

Recall the expansion of the element in the MSR matrix
\eqref{eq:Vrsexp}. We prove the following estimate of the
truncation error.
\begin{proposition} \label{prop:Ers}
Let $E_{rs}$ be as in \eqref{eq:Vrsexp}. Set $\eps = \delta/R$,
the ratio between the typical length scale of the inclusion $D$
and the distance of the receivers (sources) from the inclusion.
Assume also that $\eps$ is much smaller than one. Then
\begin{equation}
|E_{rs}| \lesssim \eps^{K+2}. \label{eq:prop:Ers}
\end{equation}
\end{proposition}
\begin{proof} From the Taylor expansion of multivariate functions (\cite{Taylor1}, Chapter 1),  we verify that the truncation error $E_{rs}$ can be written as
\begin{equation*}
\begin{aligned}
\int_{\partial D} e_K(y; x_r,z) (\lambda I - \mathcal{K}_D^*)^{-1}
\bigg[\frac{\partial \Gamma(\cdot - x_s)}{\partial \nu}\bigg]&(y) ds(y) \\
+ \int_{\partial D}  \Gamma_K(y; x_r, z) &(\lambda I -
\mathcal{K}_D^*)^{-1} \bigg[\frac{\partial}{\partial \nu}
e_K(\cdot; z, x_s)\bigg](y) ds(y).
\end{aligned}
\end{equation*}
Here, $\Gamma_K(y; x_r,z)$ and $e_K(y; x_r,z)$ (and similarly
$e_K(y; z, x_s)$) are given by\begin{equation*}
\begin{aligned}
\Gamma_K(y; x_r,z) &= \sum_{k=1}^K \sum_{|\alpha| = k} \frac{(-1)^{|\alpha|}}{\alpha!} \partial^{\alpha} \Gamma(x_r - z) (y-z)^\alpha,\\
 e_K(y; x_r, z) &= \sum_{|\alpha|=K+1} \Big( \frac{1}{\alpha!} \int_0^1 (1-s)^K \partial^\alpha \Gamma(x_r - z - s(y-z)) ds \Big) (y-z)^\alpha.
\end{aligned}
\end{equation*}

Due to the invariance relation \eqref{eq:ImKrot}, the operator
$(\lambda I - \mathcal{K}_D^*)^{-1}$, as an operator from the
space $L^2(\partial D)$ to itself, is bounded uniformly with
respect to the scaling of $D$. Consequently, the first term in
$E_{rs}$ is bounded by
\begin{equation*}
C \|e_K(\cdot; x_r,z)\|_{L^\infty(\partial D)} \|\frac{\partial
\Gamma(\cdot - x_s)}{\partial \nu}\|_{L^2(\partial D)} |\partial
D|^{\frac{1}{2}} \le C \|e_K\|_{L^\infty(\partial D)}
\|\frac{\partial \Gamma(\cdot - x_s)}{\partial
\nu}\|_{L^\infty(\partial D)} |\partial D|.
\end{equation*}
Assume that $z \in D$; the distance between $\overline{D}$ and the
receivers (sources) is of order $R$. From the above expression of
$e_K$, the explicit form of $\partial^\alpha \Gamma$ in
\eqref{eq:DGamma}, and the fact that $|y-z| \le C\delta$ for $y
\in \overline{D}$, we have
\begin{equation*}
|e_K(y; x_r, z)| \le C \left(\sum_{|\alpha| = K+1}
\frac{1}{\alpha!} \|\partial^\alpha \Gamma_r(x_r -
\cdot)\|_{\mathcal{C}(\overline{D})}\right) |y-z|^{K+1} \le C
\left(\frac{\delta}{R}\right)^{K+1}.
\end{equation*}
Similarly, we have $\|\partial_\nu \Gamma(\cdot -
x_s)\|_{L^\infty(\partial D)} \le CR^{-1}$. The measure $|\partial
D|$ in dimension two is of order $\delta$. Substituting these
estimates into the bound for the first term in $E_{rs}$, we see
that it is bounded by $C\eps^{K+2}$.

The second term can be bounded from above by
\begin{equation*}
C \|\Gamma_K\|_{L^\infty(\partial D)} \|\frac{\partial e_K(\cdot;
z, x_s)}{\partial \nu}\|_{L^\infty(\partial D)} |\partial D|.
\end{equation*}
We have $\|\Gamma_K(\cdot; x_r, z)\|_{L^\infty(\partial D)} \le
C\eps$, which is the order of the leading term. Further, from the
explicit form of $e_K$, we verify that
\begin{equation*}
\|\frac{\partial e_K(\cdot; z, x_s)}{\partial
\nu}\|_{L^\infty(\partial D)} \le C \left( \|\Gamma(\cdot -
x_s)\|_{\mathcal{C}^{K+2}(\overline{D})} \delta^{K+1} +
\|\Gamma(\cdot - x_s)\|_{\mathcal{C}^{K+1}(\overline{D})} \delta^K
\right) \le C\frac{\delta^K}{R^{K+1}}.
\end{equation*}
 As
a result, the above upper bound for the second term in $E_{rs}$ is
of order $\eps^{K+2}$ as well. This proves \eqref{eq:prop:Ers}.
\end{proof}

\begin{proposition} The solution $u_s(x)$ defined by the transmission problem \eqref{eq:transm}
satisfies the symmetry property
\begin{equation}
u_s(x_r) = u_r(x_s). \label{eq:Vsym}
\end{equation}
\end{proposition}
\begin{proof} Let $\Omega^\eps_s$ be the the ball of radius $\eps$
centered at $x_s$, and $\Omega^\eps_r$ the ball of radius $\eps$
centered at $x_r$. Let $\Omega_\eps$ be the domain $B_R
\backslash(\Omega^\eps_r \cup \Omega^\eps_s \cup D)$ where $B_R$
is a sufficiently large ball with radius $R$. Then we have
\begin{equation*}
\begin{aligned}
0 &= \int_{\Omega_\eps} \bigg(u_s(x) \Delta u_r(x) - u_r(x) \Delta
u_s(x)\bigg) dx = \int_{\partial \Omega_\eps} \bigg( u_s(x)
\frac{\partial u_r}{\partial n} (x) - u_r(x) \frac{\partial
u_s}{\partial n} (x)
\bigg) ds(x)\\
&= -\int_{\partial \Omega^\eps_s} \bigg( u_s(x) \frac{\partial
u_r}{\partial n} (x) - u_r(x) \frac{\partial u_s}{\partial n} (x)
\bigg) ds(x) - \int_{\partial \Omega^\eps_r} \bigg( u_s(x)
\frac{\partial u_r}{\partial n} (x) - u_r(x) \frac{\partial
u_s}{\partial n} (x) \bigg)
 ds(x)\\
 & \quad - \int_{\partial D} \bigg( u_s(x) \frac{\partial u_r}{\partial n}
(x)\Big|_+ - u_r(x) \frac{\partial u_s}{\partial n} (x)\Big|_+
\bigg) ds(x) + \int_{\partial B_R}\bigg(  u_s(x) \frac{\partial
u_r}{\partial n} (x)\Big|_+ -
u_r(x) \frac{\partial u_s}{\partial n} (x)\Big|_+ \bigg) ds(x)\\
 &= J^\eps_s + J^\eps_r + J_D + J_R.
\end{aligned}
\end{equation*}

For $J_D$, thanks to the jump conditions in \eqref{eq:transm}, we
have that
\begin{equation*}
J_D = \kappa \int_{\partial D} \bigg( u_r(x) \frac{\partial
u_s}{\partial n} (x)\Big|_-  - u_s(x) \frac{\partial u_r}{\partial
n} (x)\Big|_- \bigg) ds(x) = \kappa \int_D  \bigg( u_r(x) \Delta
u_s(x) - u_s(x) \Delta u_r(x) \bigg) dx = 0.
\end{equation*}

The other two terms $J^\eps_s$ and $J^\eps_r$ can be treated
similarly; hence we focus on the first item. We've shown that
$u_s(x) = \Gamma(x-x_s) + \mathcal{S}_D[\phi_s]$. In a
neighborhood of $\Omega^\eps_s$,  we have
\begin{equation*}
\|u_r\|_{L^\infty} + \|\nabla u_r\|_{L^\infty} + \|\mathcal{S}_D
[\phi_s]\|_{L^\infty} + \|\nabla \mathcal{S}_D [\phi_s]
\|_{L^\infty} \le C.
\end{equation*}
Consequently,
\begin{equation*}
\left| \int_{\partial \Omega^\eps_s} u_s(x) \frac{\partial
u_r}{\partial n} (x) \right| \le C\int_{\partial B_\eps(x_s)} (1+
|\log \eps|) ds(x) \le C\eps |\log \eps|.
\end{equation*}
\begin{equation*}
\left| \int_{\partial \Omega^\eps_s} u_r(x) \left(\frac{\partial
u_s }{\partial n} (x) - \frac{\partial \Gamma }{\partial n} (x -
x_s) \right) \right| ds(x) \le \left| \int_{\partial
\Omega^\eps_s} u_r(x) \frac{\partial \mathcal{S}_D
[\phi_s]}{\partial n} (x) ds(x)\right| \le  C\eps.
\end{equation*}
These estimates imply that
\begin{equation*}
\lim_{\eps \to 0} J^\eps_s = \lim_{\eps \to 0} \int_{\partial
B_\eps(x_s)} u_r(x_s + y) \frac{\partial \Gamma}{\partial n} (y)
ds(y) = \lim_{\eps \to 0} \frac{1}{2\pi \eps} \int_0^{2\pi} \eps
u_r(x_s + \eps \theta) d\theta = u_r(x_s).
\end{equation*}
The same analysis applied to $J^\eps_r$ shows that $\lim_{\eps \to
0} J^\eps_r = -u_s(x_r)$.

To control $J_R$, we recall the fact that $\mathcal{S}_D[\phi]$
decays as $|x|^{-1}$ and $\nabla \mathcal{S}_D[\phi]$ decays as
$|x|^{-2}$ for $\phi \in L^2(\partial D)$ satisfying
$\int_{\partial D} \phi ds = 0$; these estimates imply that the
logarithmic part of $u_s$ dominates. Therefore,
\begin{equation*}
\lim_{R \to \infty} J_R = \lim_{R \to \infty} \int_{\partial B_R}
\log|x-x_s| \frac{\langle \nu_x, x-x_r\rangle}{|x-x_r|^2} -
\log|x-x_r| \frac{\langle \nu_x, x-x_s\rangle}{|x-x_s|^2} ds(x).
\end{equation*}
The integrand above can be written as
\begin{equation*}
\left(\log \frac{|x-x_s|}{|x-x_r|}\right) \frac{\langle \nu_x, x -
x_r\rangle}{|x-x_r|^2} + \log|x-x_r| \left[ \frac{\langle \nu_x, x
- x_r\rangle}{|x-x_r|^2} - \frac{\langle \nu_x, x -
x_s\rangle}{|x-x_s|^2}\right].
\end{equation*}
We verify that the first term is of order $o(\frac{1}{R})$; its
contribution to the limiting integral is hence negligible. The
second term in the integrand can be further written as
\begin{equation*}
\log|x-x_r| \left[ \langle \nu_x, x - x_r \rangle
\left(\frac{1}{|x - x_r|^2} - \frac{1}{|x-x_s|^2}\right) +
\frac{\langle \nu_x, x - x_r - (x - x_s)\rangle}{|x-x_s|^2}
\right].
\end{equation*}
From
\begin{equation*}
\frac{1}{|x - x_r|^2} - \frac{1}{|x-x_s|^2} = \frac{|x_s|^2 -
|x_r|^2 + 2\langle x, x_r - x_s \rangle}{|x - x_r|^2 |x-x_s|^2},
\end{equation*}
we verify that the second term in the integrand is of order
$O(\log R/R^2)$; hence its contribution to the limiting integral
is also zero. To summarize, we have $\lim_{R \to \infty} J_R = 0$.

From the above analysis, we take the limit $\eps \to 0, R \to
\infty$ on the equality $0 = J^\eps_s + J^\eps_r + J_D + J_R$ and
conclude that \eqref{eq:Vsym} holds.
\end{proof}

\subsection{Proof of formula \eqref{eq:DGamma}}
\label{sec:app2} Formula \eqref{eq:DGamma} is well-known. We
include a proof for reader's sake.

In order to prove \eqref{eq:DGamma}, we need to find the
derivative of the function $\log |x|$. To this end, we consider
the Taylor expansion of the logarithmic function around the point
$x$. The most convenient method for this expansion is to view the
space variables as complex numbers. For a small perturbation $z$
of the point $x$ ($x, z \in \mathbb{C}$), we calculate
\begin{equation*}
\log |x-z| - \log|x|= \frac{1}{2} \left([\log (x - z) - \log x ]+
[\log (\overline{x} - \overline{z}) - \log \overline{x}] \right).
\end{equation*}
To expand the first item on the right-hand side of the above
equality, we write it as $\log (1-\frac{z}{x})$, and since
$|\frac{z}{x}| < 1$ we obtain the expansion
\begin{equation*}
\log (1- \frac{z}{x}) = -\sum_{j=1}^\infty \frac{1}{j}
\left(\frac{z}{x}\right)^j = -\sum_{j=1}^\infty \frac{1}{j} \left(
\frac{r_z e^{i\theta_z}}{r_x e^{i\theta_x}} \right)^j.
\end{equation*}
Taking the conjugate, we obtain the expansion for
$\log(\overline{x} - \overline{z}) - \log \overline{x}$.
Consequently, we have
\begin{equation*}
\begin{aligned}
\log |x-z| - \log|x| &= -\frac{1}{2} \sum_{j=1}^\infty \frac{1}{j} \left[ \left( \frac{r_z e^{i\theta_z}}{r_x e^{i\theta_x}} \right)^j + \left( \frac{r_z e^{-i\theta_z}}{r_x e^{-i\theta_x}} \right)^j \right]\\
&= -\sum_{j=1}^\infty \frac{1}{j} \left(\frac{\cos j\theta_x}{r_x^j} [r_z^j \cos j\theta_z] + \frac{\sin j\theta_x}{r_x^j} [r_z^j \sin j\theta_z] \right)\\
&= -\sum_{j=1}^\infty \frac{1}{j} \left(\frac{\cos
j\theta_x}{r_x^j} \sum_{|\alpha| = j} a^j_\alpha z^\alpha +
\frac{\sin j\theta_x}{r_x^j} \sum_{|\alpha| = j} b^j_\alpha
z^\alpha  \right).
\end{aligned}
\end{equation*}
In the last equality, we understood the variable $z$ as real
variable and used the representation \eqref{eq:abcomp}. Compare
the last term of the above formula with the (real-variable)
multivariate expansion of $\log |x-z| - \log |x|$, we observe that
\begin{equation*}
\sum_{|\alpha| = j} \frac{(-1)^j}{\alpha!} (\partial^\alpha_x
\log|x|) z^\alpha = -\sum_{|\alpha| = j} \frac{1}{j}
\left(\frac{\cos j\theta_x}{r_x^j}  a^j_\alpha + \frac{\sin
j\theta_x}{r_x^j} b^j_\alpha  \right)z^\alpha .
\end{equation*}
For each double index $\alpha$, we get \eqref{eq:DGamma}.
\subsection{Proof of formula \eqref{eq:ortho}}
\label{sec:app3}
 The proof is a straightforward computation. The elements of the
matrix $\Ccoef^t \Ccoef$ correspond to inner products of columns
of the matrix $\Ccoef$, that is, the inner products of vectors
formed by evaluating $\sin$ and $\cos$ functions at $(k_1
\theta_1, \ldots, k_1 \theta_N)$ and at $(k_2 \theta_1, \ldots,
k_2 \theta_N)$, where $k_1, k_2 = 1, 2, \ldots, K$, $k_1 + k_2 \le
2K < N$, and $\theta_j = 2\pi j/N$, $j=1, 2, \ldots, N$. When two
$\cos$ vectors are chosen, the inner product becomes
\begin{equation*}
\sum_{j=1}^N \cos k_1 \theta_j \cos k_2 \theta_j =  \frac{1}{4}
\sum_{j=1}^N \left( e^{i\frac{2\pi (k_1 + k_2) j}{N}} +
e^{-i\frac{2\pi (k_1 + k_2) j}{N}} + e^{i\frac{2\pi (k_1 - k_2)
j}{N}} + e^{-i\frac{2\pi (k_1 - k_2) j}{N}} \right).
\end{equation*}
Since $k_1 + k_2$ is an integer less than $N$, the first two sums
always vanish because
\begin{equation*}
\sum_{j=1}^N e^{i\frac{2\pi(k_1 + k_2)j}{N}} = \frac{1-e^{i
2\pi(k_1+k_2)}}{1-e^{i\frac{2\pi(k_1+k_2)}{N}}} = 0.
\end{equation*}
When $k_1 = k_2$, the last two sums contribute and the overall
result is $N/2$. When $k_1 \ne k_2$, the inner products under
estimation is zero according to the above observation.

The case of inner product with $\sin$ and $\sin$ or $\cos$ and
$\cos$ vectors can be similarly analyzed, and it can be easily
seen that \eqref{eq:ortho} holds.


\begin{thebibliography}{1}

\bibitem{optnew1} {\sc H. Ammari, P. Garapon,  F. Jouve, H. Kang,
M. Lim, and S. Yu}, {\em A new optimal control approach for the
reconstruction of extended
 inclusions}, SIAM J. Control Opt., to appear.


\bibitem{AGJ}
{\sc H.~Ammari, J.~Garnier, and V. Jugnon}, {\em Detection,
reconstruction, and characterization algorithms from noisy data in
multistatic wave imaging}, submitted, (2011).

\bibitem{optnew2} {\sc H. Ammari, J. Garnier, H. Kang, M. Lim, and K. S\o lna},
{\em Multistatic imaging of extended targets}, SIAM J. Imaging
Sci., to appear.


\bibitem{AGKLY11}
{\sc H.~Ammari, J.~Garnier, H.~Kang, M.~Lim, and S.~Yu}, {\em
  Generalized polarization tensors for shape description}, submitted,  (2011).

\bibitem{resol} {\sc H. Ammari, J. Garnier, and K. S{\o}lna},
{\em Resolution and stability analysis in full-aperature,
linearized conductivity and wave imaging}, Proc. Amer. Math. Soc.,
to appear.

\bibitem{AK04}
{\sc H.~Ammari and H.~Kang}, {\em Reconstruction of small
inhomogeneities from
  boundary measurements}, vol.~1846, Lecture Notes in Mathematics,
  Springer-Verlag, Berlin, 2004.

\bibitem{ammari_polarization_2007}
{\sc H. Ammari and H. Kang}, {\em Polarization and moment
    tensors: with applications to inverse problems and effective
    medium theory}, vol.~162, Springer-Verlag, 2007.

\bibitem{AK_SIMA_03} {\sc H. Ammari and H. Kang},
\newblock  {\em High-order terms in the
 asymptotic expansions of the steady-state voltage potentials in
 the presence of conductivity inhomogeneities of small diameter},
\newblock SIAM J. Math. Anal., 34 (2003),  pp.~1152--1166.


\bibitem{AK_MMS_03} {\sc H. Ammari and H. Kang},
\newblock  {\em Properties of generalized polarization tensors},
\newblock  SIAM Multiscale Model. Simul., 1 (2003), pp.~335--348.


\bibitem{AKLL11}
{\sc H.~Ammari, H.~Kang, H.~Lee, and M.~Lim}, {\em Enhancement of
near cloaking
  using generalized polarization tensors vanishing structures. {P}art {I}:
  {T}he conductivity problem}, Comm. Math. Phys., to appear.

\bibitem{AKLZ12}
{\sc H.~Ammari, H.~Kang, M.~Lim, and H.~Zribi}, {\em The
generalized
  polarization tensors for resolved imaging. {P}art {I}: {S}hape reconstruction
  of a conductivity inclusion}, Math. Comp., 81 (2012), pp.~367--386.

\bibitem{mc2} {\sc H. Ammari, H. Kang, E. Kim, and J.-Y. Lee}, {\em The
generalized  polarization tensors for resolved imaging. Part II:
Shape and electromagnetic parameters reconstruction of an
electromagnetic inclusion from multistatic measurements}, Math.
Comp., 81 (2012), pp.~839--860.

\bibitem{AKT_AA_05} {\sc H. Ammari, H. Kang, and K. Touibi},
{\em Boundary layer techniques for deriving the effective
properties of composite materials},  Asymp. Anal., 41  (2005),
pp.~119--140.


\bibitem{BHV01} {\sc M. Br\"{u}hl, M. Hanke, and M.~S. Vogelius},
\newblock {\em A direct impedance tomography algorithm for locating small
inhomogeneities},
\newblock Numer. Math., 93 (2003), pp.~635--654.

\bibitem{yves} {\sc Y. Capdeboscq, A.~B. Karrman, and J.-C. N\'ed\'elec}, {\em Numerical
computation of approximate generalized polarization tensors},
Appl. Anal., to appear.


\bibitem{CMV98} {\sc D.J. Cedio-Fengya, S. Moskow, and M.S. Vogelius},
    \newblock {\em Identification of conductivity imperfections of small diameter
    by boundary measurements: Continuous dependence and computational reconstruction},
    \newblock Inverse Problems, 14 (1998), pp.~553--595.


\bibitem{dassios} {\sc G. Dassios and R. Kleinman}, {\em Low
frequency scattering},  Oxford Mathematical Monographs,  Oxford
University Press, New York, 2000.


\bibitem{FV_ARMA_89} {\sc A. Friedman and M.S. Vogelius},
    \newblock {\em Identification of small inhomogeneities of extreme conductivity by boundary
    measurements: a theorem on continuous dependence},
    \newblock Arch. Rat. Mech. Anal., 105 (1989), pp.~299--326.



\bibitem{opt} {\sc E. Haber, U.~M.~ Ascher, and D. Oldenburg},
{\em On optimization techniques for solving nonlinear inverse
problems}, Inverse Problems, 16 (2000), pp.~1263--1280.



\bibitem{LH95}
{\sc C.~L. Lawson and R.~J. Hanson}, {\em Solving Least Squares
Problems},
  vol.~15 of Classics in Applied Mathematics, Society for Industrial and
  Applied Mathematics (SIAM), Philadelphia, PA, 1995.
\newblock Revised reprint of the 1974 original.


\bibitem{milton}  {\sc G. W. Milton},
\newblock {\sl  The Theory of Composites},
\newblock Cambridge Monographs on Applied and Computational
Mathematics, Cambridge University Press, 2001.


\bibitem{PS51} {\sc G. P{\'o}lya and G. Szeg{\"o}},
    \newblock {\sl  Isoperimetric Inequalities in Mathematical Physics},
    \newblock Annals of Mathematical Studies Number 27, Princeton University Press, Princeton, NJ, 1951.


\bibitem{opt1} {\sc A. Tarantola}, {\em Inverse Problem Theory and Methods for Model Parameter
Estimation}, SIAM, Philadelphia, PA, 2005.

\bibitem{Taylor1}
{\sc M.~E. Taylor}, {\em Partial differential equations. {I}},
vol.~115 of
  Applied Mathematical Sciences, Springer-Verlag, New York, 1996.
\newblock Basic theory.


\bibitem{opt2} {\sc C.~R.~Vogel}, {\em Computational Methods for Inverse
Problems}, Frontiers in Applied Mathematics, vol.~23,  SIAM,
Philadelphia, PA, 2002.


\end{thebibliography}

\end{document}